\documentclass[11pt]{article}
\usepackage[inner=4 cm,outer=0 cm,top=4 cm,bottom=1 cm]{geometry}

\usepackage{amsmath,amssymb,amsfonts,amsthm}
\newenvironment{myabstract}{\par\noindent
{\bf Abstract . } \small }
{\par\vskip8pt minus3pt\rm}
\newcounter{item}[section]
\newcounter{kirshr}
\newcounter{kirsha}
\newcounter{kirshb}
\newenvironment{enumroman}{\setcounter{kirshr}{1}
\begin{list}{(\roman{kirshr})}{\usecounter{kirshr}} }{\end{list}}
\newenvironment{enumarab}{\setcounter{kirshb}{1}

\begin{list}{(\arabic{kirshb})}{\usecounter{kirshb}} }{\end{list}}
\newtheorem{theorem}{Theorem}[section]

\newtheorem{lemma}[theorem]{Lemma}
\newtheorem{corollary}[theorem]{Corollary}
\newenvironment{demo}[1]{\noindent{\bf #1.}\upshape\mdseries}
{\nopagebreak{\hfill\rule{2mm}{2mm}\nopagebreak}\par\normalfont}
\theoremstyle{definition}
\newtheorem{remark}[theorem]{Remark}

\newtheorem{definition}[theorem]{Definition}

\def\Nr{{\sf{Nr}}}

\def\Sg{{\mathfrak{Sg}}}
\def\Fm{{\mathfrak{Fm}}}

\def\K{{\mathfrak{K}}}
\def\CA{{\bf CA}}
\def\RCA{{\bf RCA}}

\def\QA{{\bf QA}}
\def\QEA{{\bf QEA}}
\def\Rd{{\ Rd}}
\def\QEA{{\bf PEA}}

\def\(R)RA{{\bf (R)RA}}
\def\RA{{\bf RA}}

\def\R{\mathbb{R}}
\def\N{\mathbb{N}}
\def\Q{\mathbb{Q}}
\def\C{\mathbb{C}}
\def\A{{\mathfrak{A}}}
\def\B{{\mathfrak{B}}}
\def\C{{\mathfrak{C}}}
\def\D{{\mathfrak{D}}}
\def\R{{\mathfrak{R}}}

\def\L{{\mathfrak{L}}}
\def\Rd{{\mathfrak{Rd}}}

\def\T{{\bf T}}

\def\T{{\bf T}}

\def\Ra{{\sf Ra}}
\def\set#1{ \{#1\}}
\def\Mo{{\sf Mo}}
\def\restr #1{{\restriction_{#1}}}
\def\ws{winning strategy}

 \def\CA{{\sf CA}}
\def\B{{\sf B}}
\def\G{{\sf G}}
\def\w{{\sf w}}
\def\y{{\sf y}}
\def\g{{\sf g}}

\def\r{{\sf r}}
\def\K{{\sf K}}

\def\pa{$\forall$}
\def\pe{$\exists$}

\def\ef{Ehren\-feucht--Fra\"\i ss\'e}

\def\tp{{\sf tp}}

\def\M{{\mathfrak{M}}}

\def\Ca{{\mathfrak{Ca}}}

\def\M{{\mathfrak{M}}}

\def\A{{\mathfrak{A}}}
\def\B{{\mathfrak{B}}}
\def\C{{\mathfrak{C}}}
\def\D{{\mathfrak{D}}}
\def\Ig{{\mathfrak{Ig}}}
\def\M{{\mathfrak{M}}}
\def\dom{{\sf dom}}
\def\rng{{\sf rng}}
\def\Rd{{\sf{Rd}}}
\def\At{{\sf At}}
\def\Ra{{\mathfrak{Ra}}}
\def\Tm{{\mathfrak{Tm}}}
\def\Cm{{\mathfrak{Cm}}}

\def\VT{{\sf VT}}

\def\ef{Ehren\-feucht--Fra\"\i ss\'e}

\def\Id{{\sf Id}}
\def\Rl{{\mathfrak{Rl}}}
\def\F{{\mathfrak{F}}}
\def\Sc{{\mathfrak{Sc}}}
\def\QEA{{\sf QEA}}
\def\RQEA{{\sf RQEA}}
\def\RCA{{\sf RCA}}
\def\QEA{{\bf PEA}}
\def\PA{{\bf PA}}
\def\R{{\sf R}}
\def\L{{\sf L}}

\def\Df{{\sf Df}}
\def\Uf{{\sf Uf}}
\def\Rd{{\mathfrak{Rd}}}
\def\PEA{{\sf QEA}}

\def\s{{\sf s}}
\def\Sc{{\sf Sc}}
\def\nodes{{\sf nodes}}

\def\Str{{\sf Str}}

\def\T{{\sf T}}
\def\QEA{{\sf PEA}}

\def\cyl#1{{\sf c}_{#1}}

\def\cyl#1{{\sf c}_{#1}}
\def\sub#1#2{{\sf s}^{#1}_{#2}}
\def\diag#1#2{{\sf d}_{#1#2}}

\def\V{{\sf V}}
\def\de{Dedekind-MacNeille}
\def\PA{{\sf PA}}
\def\QA{{\sf QA}}

\def\Z{{\mathbb{Z}}}
\def\Lf{{\sf Lf}}
\def\Bb{{\sf Bb}}
\def\Cs{{\sf Cs}}
\def\Ra{{\sf Ra}}
\def\Mo{{\sf M}}
\def\QEA{{\sf QEA}}\def\Mat{{\sf Mat}}

\def\RA{{\sf RA}}
\def\CRCA{{\sf CRCA}}
\def\G{{\sf G}}

\def\T{{\cal T}}

\def\Bb{\mathfrak{Bb}}

\def\w{{\sf w}}
\def\g{{\sf g}}
\def\y{{\sf y}}
\def\r{{\sf r}}
\def\L{\mathfrak{L}}
\def\G{{\cal G}}
\def\PEA{{\sf PEA}}
\def\R{\mathfrak{R}}
\title{Atom--canonicity in algebraic logic with applications to omitting types in modal fragments of $L_{\omega,\omega}$}
\author{Tarek Sayed Ahmed\\
Department of Mathematics, Faculty of Science,\\ 
Cairo University, Giza, Egypt. }
\begin{document}
\maketitle

\begin{myabstract}
Fix $2<n<\omega$. Let $L_n$ denote first order logic restricted to the first $n$ variables.  
$\CA_n$ denotes the class of cylindric algebras of dimension $n$
and for $m>n$, $\Nr_n\CA_{m}(\subseteq \CA_n$) 
denotes the class of $n$--neat reducts
of $\CA_{m}$s. 
The existence of certain finite relation algebras 
and  finite $\CA_n$s lacking relativized complete representations 
is shown to imply that the omitting types theorem ($\sf OTT$) fails for $L_n$ with respect to clique guarded semantics (which is an equivalent formalism of its packed fragments), 
and for the multi-dimensional modal logic ${\bf S5}^n$. 
Several such relation and cylindric algebras are explicitly exhibited using rainbow constructions and Monk--like algebras; two player zero sum games are used to show that
they are as required. 
Certain $\CA_n$s constructed to show non--atom canonicity of the variety $\bold S\Nr_n\CA_{n+3}$, where $\bold S$ stands for the operation of forming subalgebras,
are used to show that Vaught's theorem ($\sf VT$) for $L_{\omega, \omega}$, looked upon as a special case of $\sf OTT$ for $L_{\omega, \omega}$, fails
almost everywhere (a notion to be defined below) when restricted to $L_n$. That $\sf VT$ fails everywhere for $L_n$, which is stronger than failing almost everywhere as the name suggests, 
is reduced to the existence, for each $n<m<\omega$, of a finite relation algebra $\R_m$ having a so--called $m-1$ strong blur, 
but $\R_m$ has no $m$--dimensional relational basis. 
$\sf VT$ for other modal fragments and expansions of $L_n$, 
like its guarded fragments, $n$--products of uni-modal logics like $\bold K^n$, and first order definable expansions, is approached.  
It is shown that any multi-modal canonical logic $\L$, such that $\bold K^n\subseteq {\L\subseteq \bf S5}^n$, 
$\L$ cannot be axiomatized by canonical equations.  In particular, $\L$ is not Sahlqvist. Elementary generation and 
di--completeness for $L_n$ and its clique guarded fragments are proved. 
Positive omitting types theorems are proved for $L_n$ with respect to standard semantics 
by imposing extra conditions on theories that are not necessarily countable like quantifier elimination,  and/or types considered
like maximality.
\end{myabstract}

\section{Introduction}

There is a strong tradition in logic research, that dates back to the fifties of the last century,
of applying algebraic methods to deepen our understanding of logical concepts.
This kind of research, initiated mostly by Alfred Tarski, is now a huge subject better known as {\it Algebraic Logic}. 
Algebraic Logic is the natural interface between universal algebra and logic;  for example, 
the consequence relation of a logic translates to the quasi--equational theory of the corresponding class of algebras.

In this paper we apply Tarskian Algebraic Logic to multi-modal logic. 
More precisely, we use well-developed algebraic machinery in the theory of relation and cylindric algebras to obtain results on modal fragments of first order logic, 
such as its packed fragments, clique guarded fragments, and modal logics between $\bold K^n$ and ${\bf S5}^n$.
Such applications are not only possible, but also illuminating in the sense that it makes one delve deeper into the analysis of the 
problem at hand. 
The metalogical property we are primarily concerned with is the {\it omitting types theorem}, briefly $\sf OTT$,  
in such fragments, though occasionally, as the paper unfolds, we make digressions to other related metalogical properties.
The technical notion of a modal logic corresponds to the one of a variety of {\it Boolean algebras with operators}, of which relation and cylindric algebras are prominent examples.  
A central notion for proving completeness in modal logic is the notion of canonicity, which has an equally important algebraic expression. 
In this paper, we discuss the connection
between the logical and the algebraic perspective on the properties of $\sf OTT$ and {\it atom-canonicity} to be defined in a moment. 

The purpose of the paper is twofold.
Apart from presenting some novel ideas of applying algebra to logic, we intend to present our hitherto obtained results in 
both algebraic and modal logic in an integrated format.

{\bf Omitting types:} Let $\L$ be an extension or reduct or variant of first order logic,  like first logic itself, $L_n$ as defined in the abstract with $2<n<\omega$,  
$L_{\omega_1, \omega}$, $L_{\omega}$ as defined in \cite[\S 4.3]{HMT2}, a modal fragment of $L_n$, $\ldots$,  etc. 
An $\sf OTT$ for $\L$ is typically of the form  `A countable family of non--isolated types in a countable $\L$ theory 
$T$ can be omitted in a countable model of $T$'.  From this it directly follows, that if a  type is realizable in every model of a countable theory $T$, 
then there should be a formula consistent with $T$ that isolates this
type. A type is simply a set of formulas $\Gamma$ say.  The type $\Gamma$ is realizable in a model 
if there is an assignment that satisfies (uniformly) all formulas in $\Gamma$. 
Finally, $\phi$ isolates $\Gamma$ means that $T\vdash \phi\to \psi$ 
for all $\psi\in \Gamma$. What Orey and Henkin proved is that  the $\sf OTT$ 
holds for $L_{\omega, \omega}$ when such types are {\it finitary} meaning that they all consist of $n$-variable 
formulas for some $n<\omega$.

For $L_n$, as defined in the abstract, the situation is drastically different. It is known \cite{ANT} that the $\sf OTT$ fails in the following  (strong) sense. 
For every $2<n\leq l<\omega$, there is a countable and complete $L_n$ theory $T$, and a single
type  that is realizable in every model of $T$, but cannot be isolated by a formula 
using $l$ variables.
Here we prove stronger negative $\sf OTT$s for $L_n$ when types are required to be omitted  with respect to 
certain (much wider) generalized semantics, called {\it $m$ flat} and {\it $m$--square} with $2<n<m<\omega$. 
(Ordinary models are $\omega$--square.)
This is equivalent to the  failure of the $\sf OTT$ in the 
$m$-clique guarded fragment
of $L_n$ and its packed fragment (to be defined below). In fact, we prove that the single type consisting of co--atoms in a countable  atomic theory $T$ 
cannot be omitted in $n+3$--square models; by atomic we mean that the Tarski--Lindenbaum quotient algebra $\Fm_T$ is atomic (as a Boolean algebra).
This violates a famous theorem of Vaught that is a consequence of $\sf OTT$ for $L_{\omega, \omega}$, 
namely, that atomic theories have atomic models.

{\bf Algebraic machinery; blow up and blur constructions:} Fix $2<n<\omega$.
We assume familiarity with the basic notions of the (duality) theory of Boolean algebras with  operators $\sf BAO$s,  like {\it atom structures} and  {\it  complex algebras}.
A good reference is \cite[\S 2.5, \S 2.6,  \S 2.7]{HHbook}.
Classes of algebras considered throughout this paper, like relation algebras ($\RA$) and cylindric algebras of dimension $n$ ($\CA_n$), $n$ any ordinal,
are $\sf BAO$s endowed with  a {\it semantical notion of representability} \cite{HMT2}. 
For any such class $\sf L$ we write $\sf RL$ for the class of representable algebras in $\sf L$.
For example $\sf RRA$ is the class of representable $\RA$s, 
and $\RCA_n$ is the class of 
representable $\CA_n$s.

In this paper, we present, so--called {\it blow up and blur constructions}, an indicative  term introduced in  \cite{ANT}, 
as {\it splitting arguments}, as adopted in \cite{Andreka}, at heart involving splitting (atoms) in 
finite Monk--like algebras and  rainbow algebras \cite{HHbook, HHbook2}.  This method proves useful in obtaining results 
on non-atom--canonicity of several varieties of relation and cylindric algebras containing (and including) the varieties $\sf RRA$s and $\RCA_n$.
We recall that a class $\sf L$ of $\sf BAO$s is {\it atom--canonical} if whenever $\A\in \sf L$ is completey additive, then its \de\ completion, 
namely, the complex algebra of its atom structure (in symbols $\Cm\At\A$) is also in $\sf L$.
Atom-canonicity in completely additive varieties correspond in modal logic to the notion of a 
formula being {\it dipersistent}. A formula is {\it dipersistent} if whenever it is valid in
some {\it general discrete} frame $(\F, P)$,  that is, $P$ contains all singletions, 
then is valid in the Kripke frame $\F$ \cite[\S5.6]{modal}. Discrete frames form a natural generalization of Kripke frames and they provide a rather well--behaved semantics for modal logic.
For example, unlike Kripke frames, the notions of completeness, strong completeness, and strong global completeness 
coincide for discrete frames.

The notion of atom--canonicity plays an equally fundamental role in the theory of modal logic as in the algebraic theory 
of $\sf BAO$s. In this connection, discrete frames correpond to atomic completely additive $\sf BAO$s.

The  simple though quite sophisticated method of blow up and blur constructions  is applicable to any two classes $\bold L\subseteq \bold K$, where $\bold K$ is a class of $\sf BAO$s.
One takes  an atomic $\A\notin \bold K$ (usually but not always finite), blows it up, by splitting one or more of its atoms each to infinitely many subatoms,
obtaining an (infinite) atomic $\Bb(\A)\in \bold L$,  such that $\A$ {\it is blurred} in $\Bb(\A)$ meaning that  $\A$ {\it does not} embed in $\Bb(\A)$, but $\A$ embeds in the \de\ completion of $\Bb(\A)$,
namely, $\Cm\At\Bb(\A)$. (The notation $\Bb(\A)$ is short for {\it blowing up and blurring $\A$}).
Then any class $\bold M$ say, between $\bold L$ and $\bold K$ that is closed under forming subalgebras  
will not be atom--canonical, because $\Bb(\A)\in \bold M$, but $\Cm\At\Bb(\A)\notin \bold M$, lest $\A$ will be in $\bold M\subseteq \bold K$ 
since $\bold M$ is closed under forming subalgebras.  We say, in this case, that $\bold L$ is not atom--canonical {\it with respect to $\bold K$}.
This method will be applied to $\bold K$, when $\bold K$ is any  of the varieties $\bold S\Ra\CA_l$, $l\geq 5$ and $\bold S\Nr_n\CA_{n+k}$ and $k\geq 3$
and $\bold M=\bold L$ is the variety $\sf RRA$ and $\RCA_n$, respectively. Here (and elsewhere throughout the paper) $\Ra$ denotes the operator of forming 
relation algebra reducts as defined in \cite{HMT2},

Also the method is applied by taking, for each $k\in \omega$, $\bold K=\Ra\CA_{n+k}\cap \sf RRA$,
and $\bold L=\sf RRA$, and blowing up and blurring a non-representable finite relation algebra $\sf R$ (known as a Maddux algebra).  
This last construction  is lifted to $\CA_n$
showing that  for any $k\in \omega$, there is an atomic
$\A\in \Nr_n\CA_{n+k}\cap \RCA_n$ 
such that $\Cm\At\A\notin \RCA_n$. The lifting to $\CA_n$ is feasable because the relation algebra 
$\Bb({\sf R}) (\in \Ra\CA_{n+k}\cap \sf RRA)$
obtained after blowing up and blurring $\sf R$, 
possesses  an {\it $n$-dimensional cylindric basis} (as defined by Maddux \cite{Maddux}). 

{\bf Omitting types for the clique guarded fragments; negative results:} 
Applying the hitherto obtained algebraic results, using the machinery of algebraic logic, 
we obtain negative results of the form:

{\it There exists a countable atomic $L_n$ theory $T$ such that  the type $\Gamma$ consisting of co--atoms is realizable in every {\it $m$--square} model, 
but $\Gamma$ cannot be isolated using $\leq l$ variables, where $l, m$ are finite ordinals $>n$.}

Call it $\Psi(l, m)$, short for {\it Vaught's Theorem fails at (the parameters) $l$ and $m$.}
Let ${\sf VT}(l, m)$ stand for {\it Vaught's Theorem holds at $l$ and $m$}, so that by definition $\Psi(l, m)\iff \neg {\sf VT}(l, m)$. We also include $l=\omega$ in the equation: 
Define ${\sf VT}(\omega, \omega)$ as Vaught's Theorem holds for $L_{\omega, \omega}$: Atomic first order countable theories have atomic countable models. 

In this paper,  we provide strong evidence that Vaught's Theorem, briefly $\sf VT$ {\it fails everywhere} in the sense that for the permitted values $n\leq l, m\leq \omega$, namely, 
for $n\leq l<m\leq \omega$  and $l=m=\omega$, 
$\VT(l, m)\iff l=m=\omega.$ Failure of $\sf VT$ everywhere is reduced to finding, then blowing up a finite relation algebra having certain properties.  
From known algebraic results like non-atom--canonicity of $\RCA_n$ \cite{Hodkinson} and non-first order definability of the class 
of completely representable $\CA_n$s \cite{HH},  it can be easily inferred that
$\VT(n, \omega)$ is false, that is, Vaught's Theorem fails for $L_n$ with respect to square Tarskian semantics. One can find countable atomic $L_n$ 
theories having no 
atomic models.  In  both cases, one can easily infer that there exists a countable, simple (has no proper ideals) and atomic $\A\in \RCA_n$ with no complete representation.
Assuming that $\A\cong \Fm_T$, where $\Fm_T$ is the corresponding Tarski--Lindebaum (cylindric) algebra of formulas \cite[\S 4.3]{HMT2}, 
then $T$ will be an atomic complete theory having no atomic model, for any such model induces a complete 
representation of $\A$.
But we can (and will) go further.
From sharper algebraic results,    we prove  many other special cases for specific values of $l$ and $m$, with $l<m$,  that support the last equivalence. 
For example from non--atom canonicity of $\bold S\Nr_n\CA_{n+3}$, we prove $\Psi(n, n+k)$ for $k\geq 3$ and from non--atom canoicity of $\Nr_n\CA_{n+k}\cap \RCA_n$ with respect  to $\RCA_n$
for all $k\in \omega$, 
we prove $\Psi(l, \omega)$ for all $l\geq n$. 
In this case, we say (and prove) that   $\sf VT$ fails {\it almost} everywhere.

{\bf On the results obtained:} Though the results in this paper mainly  
address $\sf VT$ for various modal fragments of $L_n$, 
some results hitherto obtained are purely algebraic and could be of interest in themselves regardless of any connection with (multi-modal) 
logic. We mention, in this connection,  Lemmata \ref{dfb} and  \ref{df}. 
On the other hand, the proofs of the main metalogical results, namely, Theorems \ref{main}, \ref{can}, 
\ref{S5}, \ref{complexity0},  and  \ref{i},  
addressing $\sf VT$ (and some more metalogical results), 
use somewhat sophisticated machinery of algebraic logic like 
rainbow constructions described next,  as done, more specifically, 
in Theorems \ref{b2}, \ref{can}.

{\bf On the techniques used to obtain the aforementioned results:} Games and so--called rainbow constructions, as introduced in algebraic logic by Hirsch and Hodkinson \cite{HHbook},
will be 
used extensively. Rainbow constructions are based on two player (zero-sum) games and  as the name suggests they involve `colours'. 
Such games happen to be simple \ef\ forth games 
where the two players \pe lloise and \pa belard,  between them, use pebble pairs outside the board, each player pebbling one of the two structures which 
she/he sticks to it during the whole play. In the number of rounds played (that can be transfinite),  \pe\ tries 
to show that  
two simple relational structures $\sf G$ (the greens) and $\sf R$ (the reds) have similar structures  
while \pa\ tries to show that they are essentially distinct. 
Such structures may include ordered structures and complete
irreflexive graphs, such as finite ordinals, $\omega_1$, $\N$, $\Z$ or $\mathbb{R}.$    
A \ws\ for either player in the \ef\ game 
can be lifted to \ws\ in a {\it rainbow game} played on so--called atomic networks on a rainbow atom structure (for both $\CA$s and $\RA$s ) based also 
on $\sf G$ and $\sf R$. Once $\sf G$ and $\sf R$ are specified, 
the rainbow atom structure is uniquely defined. 
 Though more (rainbow) colours (like whites and shades of yellow) are involved in the rainbow 
atom structure,  the crucial thing here is that the number of rounds and nodes in networks used in the rainbow game,  
depend recursively on the number of rounds and pebble pairs in the 
simple \ef\ forth two player game played on (and only depending on) $\sf G$ and $\sf R$.\\ 
Due to the control on \ws's in terms of the relational structures ($\sf G$ and $\sf R$) chosen in advance, 
and the number of pebble pairs used outside the board,  rainbow constructions have proved highly efficient in providing delicate counterexamples 
for both $\CA$s and $\RA$s,  cf. \cite{HH, HHbook, HHbook2, Hodkinson, mlq}.

A deep result of Shelah's and 
topological arguments, such as 
the Baire Category Theorem, are used in Theorem \ref{i} on $\sf OTT$ for $L_n$.

{\bf Notation:} Throughout the paper, we use standard or/and self-explanatory notation
following mainly  the notation of \cite{1} which is in conformity with the notation of \cite{HMT2}. 
Any less common notation will be 
explained at its first occurence in the text.

We deal with the following cylindric--like algebras $\Sc$ (Pinter's substitution algebras), $\CA$ (cylindric algebras), 
$\QA$($\QEA$) quasi-polyadic (equality) algebras. For $\K$ any of these classes and $\alpha$ any ordinal, 
we write $\K_{\alpha}$ for the variety of $\alpha$--dimensional $\K$ algebras, 
and $\sf (C)RK_{\alpha}$ for the class of (completely) representable $\K_{\alpha}$s.
By the same token, while $\RA$ denotes the class of relation algebras, $\sf (C)RRA$ will denote the class of (completely) representable $\RA$s.

For a class $\sf L$, we write ${\bf El}\sf L$ for the elementary closure of 
$\sf L$, and $\At\sf L$ for the class $\{\At\A: \A\in \bold K\cap \bf At\}$ (of first order structures).
For a $\sf BAO$ $\A$, $\A^+$ denotes its canonical extension.
For algebras $\A$ and $\B$ having a Boolen reduct, we write $\A\subseteq_c \B$ $\iff$ for all $X\subseteq \A$, $\sum ^{\A}X=1\implies \sum^{\B}X=1$, and we say that $\A$ is a {\it complete subalgebra} of
$\B$. 
For a class $\bold K$ having a Booleon reduct $\bold S_c\bold K=\{\A: (\exists \B\in \bold K)(\A\subseteq_c \B)\}$. This notation will be used several times below without further notice.
{\it In what follows $\sf VT$ abbreviates `Vaught's theorem'.}\\

{\bf Layout:} Fix $2<n<\omega$.

(1) In $\S2$ several blow up and blur constructions for relation and cylindric algebras are presented to show non-atom--canonicity of the varieties 
$\bold S\Ra\CA_m$ and $\bold S\Nr_n\CA_{n+k}$, $m\geq 5$ and $k\geq 3$,  
cf. Theorems \ref{ANT}, \ref{b2}, and \ref{can}. The constructions in Theorems \ref{b2} and \ref{can} are presented 
in the framework of almost identical blow up and blur constructions
expressed via splitting atoms in finite (extremely simple) rainbow algebras.  In Corollary \ref{sahlqvist} we show, among other things, that   
there are $\RA$s and $\CA_n$s whose \de\ completions, often referred to in the literature as {\it the minimal completions},
do not embed 
into their canonical extensions which are completions of the given algebras.\\

(2) For $n\leq l<m\leq \omega$, in $\S3$ we present a chain of implications starting from the existence of certain finite relation algebras
having so-called strong $l$--blur and no $m$--square representations (equivalently no $m$--dimensional relational basis) leading up to $\Psi(l, m)$ 
which is equivalent to $\neg {\sf VT}(l, m)$, 
as defined in the introduction,   
cf. Theorem \ref{main}. Using Theorems \ref{ANT}, \ref{can} and \ref{main}, 
we prove $\Psi(n, n+3)$ and $\Psi(l, \omega)$ for every $n\leq l<\omega$, and 
that $\sf VT$ fails for any finite first order definable expansion of $L_n$ as defined in \cite{Basim,Biro}.\\

(3) In $\S4$ we prove non-atom canonicity of ${\sf RDf}_n$, the class of representable 
diagonal free $\CA_n$s, from which we conclude that $\sf VT$ fails for ${\bf S5}^n$, cf. Theorem \ref{S5} and that ${\bf S5}^n$ 
is not Sahlqvist. 
 Going further, we show that any axiomatization {\it of any multi-modal canonical  logic $\L$} between $\bold K^n$ and ${\bf S5}^n$ must contain 
infinitely many non-canonical formulas. 
In particular, such logics are not Sahlqvist. It is shown that 
any axiomatization of ${\bf S5}^n$ must 
contain modal formulas with no first order 
correspondants, cf. Theorems \ref{complexity0} and  \ref{complexity}. 
Elementary generation and di--completeness for $L_n$ and its clique guarded fragments are obtained, cf. Theorem \ref{square}.\\

(4) In $\S5$ positive results on $\sf OTT$ for $L_n$ are proved and 
counterexamples are provided to mark the boundaries of these results, cf. Theorems \ref{i} and  \ref{mark}.\\

(5) Due to its intimacy to $\sf VT$ for (variants of) $L_n$, in $\S6$ the algebraic notion of complete representations for $\CA_n$s
is characterized using neat embeddings and games.  In this connection, 
we show that each of the classes ${\bf El}{\sf CRCA}_n$ and ${\bf El}\bold S_c\Nr_n\CA_{\omega}$ 
complete subalgebras, coincides with the class 
of $\CA_n$s satsfying the so--called Lyndon conditions, cf. Theorem \ref{iii}. Finally, we show that for any class $\sf K$, 
such that $\Nr_n\CA_{\omega}\subseteq {\sf K}\subseteq \bold S_c\Nr_n\CA_{n+3}$, ${\bf El}\sf K\neq \sf K$, that is to say, $\sf K$ is not 
definable in $L_{\omega, \omega}$, cf. Theorem \ref{iiii}.

\section{Blow up and blur constructions}

\subsection {Blowing up and blurring a finite Maddux algebra}

Here we review and elaborate on the flexible construction in \cite{ANT} as our 
first instance of a blow up and blur construction. Some new consequences of the construction are obtained.
The following definition to be used in the sequel is taken from \cite{ANT}:
\begin{definition}\label{strongblur}
Let $\R$ be a relation algebra, with non--identity atoms $I$ and $2<n<\omega$. Assume that  
$J\subseteq \wp(I)$ and $E\subseteq {}^3\omega$.
\begin{enumerate}
\item We say that $(J, E)$  is an {\it $n$--blur} for $\R$, if $J$ is a {\it complex $n$--blur} as defined in 
\cite[Definition 3.1]{ANT}  
and the tenary relation $E$ is an {\it index blur} defined  as 
in item (ii) of \cite[Definition 3.1]{ANT}.

\item We say that $(J, E)$ is a {\it strong $n$--blur}, if it $(J, E)$ is an $n$--blur,  such that the complex 
$n$--blur  satisfies (with notation as in \cite{ANT}):
  $$(\forall V_1,\ldots V_n, W_2,\ldots W_n\in J)(\forall T\in J)(\forall 2\leq i\leq n)
{\sf safe}(V_i,W_i,T).$$ 
\end{enumerate}
\end{definition}
For the next Lemma, we refer the reader to \cite[Definition 12.11]{HHbook} 
for the definition of hyperbasis for relation algebras.
For a relation algebra $\R$, recall that  $\R^+$ denote its canonical extension.

\begin{lemma}\label{i} Let $\R$ be  a relation algebra and $3<n<\omega$.  Then 
$\R^+$ has an $n$--dimensional infinite relational basis $\iff\ \R$ has an infinite $n$--square representation.
$\R^+$ has an $n$--dimensional infinite hyperbasis $\iff\ \R$ has an infinite $n$--flat representation.
\end{lemma}
\begin{proof} \cite[Theorem 13.46, the equivalence $(1)\iff (5)$ for relational basis, and the equivalence $(7)\iff (11)$ for hyperbasis]{HHbook}.
\end{proof}

Fix $2<n<\omega$. We assume familiarity with the notion of an $n$--dimensional basic matrix defined on atomic relation algebra $\R$ \cite{Maddux}. 
For a relation atom structure $\alpha$, ${\sf Mat}_n(\alpha)$ denotes the set of all $n$ by $n$
basic matrices on $\alpha$. An atom structure $\sf At$ of $\CA_n$ type is 
{\it strongly representable} if every atomic algebra having this atom structure is representable; $\sf At$ is {\it weakly representable} if at least one atomic 
algebra having this atom structure
is representable. These two notions are distinct \cite{Hodkinson}, cf. Theorem \ref{can}. If $\A\subseteq \mathfrak{Nr}_n\B$, $\B\in \CA_m$, $n<m\leq \omega$, we say that 
$\B$ is an $m$--dilation of $\A$, or simply a dilation of $\A$ 
if $m$ is clear from context.

The following Theorem generalizes
the construction in \cite{ANT} and says some more facts. We denote the relation algebra ${\Bb}(\R, J, E)$ with atom structure $\bf At$ obtained by blowing up and blurring $\R$ 
(with underlying set is denoted by $At$ on \cite[p.73]{ANT}) by ${\sf split}(\R, J, E)$).  
By the same token,  we denote the algebra ${\Bb}_l(\R, J, E)$ as defined in \cite[Top of p. 78]{ANT} by 
${\sf split}_l(\R, J, E)$. This switch of notation is motivated by the fact that we wish to emphasize the role of {\it splitting some (possibly all) atoms into infinitely subatoms 
during blowing up and blurring a finite algebra thereby presenting a multiplicity of such (blow up and blur) constructions as a variation on a single theme.}  

The reader is referred to \cite[Chapter 13, Definitions 13.4, 13.6]{HHbook} for the notions 
of $m$--flat and $m$--square representations for  relation algebras ($m>2$) to be generalized below to cylindric algebras, cf. Definition \ref{relrep}.


\begin{theorem}\label{ANT} Let $2<n\leq l<m\leq \omega$.
\begin{enumerate}
\item  Let $\R$ be a finite relation algebra with an $l$--blur $(J, E)$ where $J$ is the $l$--complex blur and $E$ is the index blur.  

(a) Let $\bf At$ be the relation algebra atom structure obtained by blowing up and blurring $\R$ as specified above.
Then the set of $l$ by $l$--dimensional matrices 
${\bf At}_{ca}={\sf Mat}_l({\bf At})$ is an $l$--dimensional cylindric basis, that is a weakly representable atom structure \cite[Theorem 3.2]{ANT}. 
The algebra  ${\sf split}_l(\R, J, E)$ with 
atom structure ${\bf At}_{ra}$  is in $\RCA_l$. Furthermore, 
$\R$ embeds into $\Cm{\bf At}$ which embeds into $\Ra\Cm({\bf At}_{ca}).$ 

(b) If $(J, E)$ is a strong $m$--blur for $\R$,
then $(J, E)$ is a strong $l$--blur for $\R$.  Furthermore, ${{{\sf split}}}_l(\R, J, E)\cong {\Nr}_l{{\sf split}}_m(\R, J, E)$ and 
for any $l\leq j\leq m$, ${\sf split}(\R, J, E)$ having atom structure $\bf At$, is isomorphic to $\Ra({{\sf split}}_j(\R, J, E))$.

\item For every $n<l$, 
there is an $\R$ having a strong $l$--blur $(J, E)$
but no infinite representations (representations on an infinite base). 
Hence the atom structures defined in (a) of the previous item  (denoted by $\bf At$ and ${\bf At}_{ca}$)
for this specific $\R$ are not strongly representable.  

\item    Let $m<\omega$. If $\R$ is  a finite relation algebra having  
a strong $l$--blur, and no $m$--dimensional hyperbasis, 
then $l<m$.

\item If $n=l<m<\omega$ and $\R$ is a finite relation algebra with an $n$ blur $(J, E)$ (not necessarily strong) 
and no infinite $m$--dimensional hyperbasis, then the algebras $\Cm\At({\sf spli}t(\R, J, E))$ and $\Cm\At({\sf split}_l(\R, J, E))$ are outside
$\bold S\Ra\CA_m$ and  $\bold S{\sf Nr}_n\CA_m$, respectively, 
and the latter two varieties are not atom--canonical. 
\end{enumerate}
\end{theorem}
\begin{proof} \cite[Lemmata 3.2, 4.2, 4.3]{ANT}. 
We start by an outline of (a) of item (1).  Let $\R$ be as in the hypothesis. 
Let $3<n\leq l$. We blow up and blur $\R$. $\R$ is blown up by splitting all of the atoms each to infinitely many
defining an (infinite atoms) structure $\bf At$.
$\R$ is blurred by using a finite set of blurs (or colours) $J$. 
The term algebra denoted in \cite{ANT} by ${{\Bb}}(\R, J, E)$) over $\bf At$, denoted here by ${\sf split}(\R, J, E)$, 
 is representable using the finite number of blurs.\footnote{Such blurs are basically non--principal ultrafilters; they are used as colours together 
with the principal ultrafilters (the atoms) to represent ${\sf split}(\R, J, E)$. 
This representation is implemented in  step-by-step manner, and in fact this step by step construction 
adopted in \cite{ANT} {\it completely represents the canonical
extension} of ${\sf split}(\R, J, E)$.}

Because $(J, E)$ is a complex set of $l$--blurs, this atom structure has an $l$--dimensional cylindric basis, 
namely, ${\bf At}_{ca}={\sf Mat}_l(\bf At)$. The resulting $l$--dimensional cylindric term algebra $\Tm{\sf Mat}_l(\bf At)$, 
and an algebra $\C$ having atom structure ${\bf At}_{ca}$ (denoted in \cite{ANT} by 
$\Bb_l(\R, J, E)$) and denoted now by ${\sf split}_l(\R, J, E)$ such that $\Tm{\sf Mat}_l({\bf At})\subseteq \C\ \subseteq \Cm{\sf Mat}_l(\bf At)$ 
is shown to be  representable.  

We prove (b) of item (1): Assume that the $m$--blur $(J, E)$ is strong, then by definition $(J, E)$ is a strong  $j$ blur for all $n\leq j\leq m$.
Furthermore,  by \cite[item (3) pp. 80]{ANT},  
${\sf split}(\R,J, E)=\Ra({{\sf split}}_j(\R, J, E))$ 
and ${\sf split}_j(\R, J, E)\cong \mathfrak{Nr}_j{\sf split}_m(\R, J, E)$. 

We prove  (2):   Like in \cite[Lemma 5.1]{ANT},  one takes $l\geq 2n-1$, $k\geq (2n-1)l$, $k\in \omega$.
The Maddux integral relation algebra ${\mathfrak E}_k(2, 3)$ 
where $k$ is the number of non-identity atoms is the required $\R$. 
In this algebra a triple $(a, b, c)$ of non--identity atoms is consistent $\iff$ $|\{a, b, c\}|\neq 1$, i.e 
only monochromatic triangles are forbidden. 

We prove (3). Let $(J, E)$ be the strong $l$--blur of $\R$. Assume  for contradiction that $m\leq l$. Then we get by \cite[item (3), p.80]{ANT},  
that  $\A={{\sf split}}_n(\R, J, E)\cong \mathfrak{Nr}_n{\sf split}_l(\R, J, E)$.  But the cylindric $l$--dimensional algebra ${{\sf split}}_l(\R, J, E)$ is atomic,  having atom structure  
${\sf Mat}_l \At({{\sf split}}(\R, J, E))$, so $\A$ has an atomic $l$--dilation.
So $\A=\mathfrak{Nr}_n\D$ where $\D\in \CA_l$ is atomic.
But $\R\subseteq_c \mathfrak{Ra}\mathfrak{Nr}_n\D\subseteq_c \mathfrak{Ra}\D$. 
By \cite[Theorem 13.45 $(6)\iff (9)$]{HHbook},  $\R$ has a complete $l$--flat representation, 
thus it has a complete $m$--flat representation, because $m<l$ and $l\in \omega$. 
This is a contradiction. 

We prove (4):   Let $\B={{\sf split}}_n(\R, J, E)$. Then, since $(J, E)$ is an $n$ blur, $\B\in \RCA_n$. But  
$\C=\Cm\At\B\notin \bold S{\sf Nr}_n\CA_{m}$, because $\R\notin \bold S\Ra\CA_m$, 
$\R$ embeds into ${\sf split}(\R, J, E)$ which, in turn, embeds into 
$\Ra\Cm\At\B$. Similarly, ${\sf split}(\R, J, E)\in \sf RRA$ and 
$\Cm(\At{\sf split}(\R, J, E))\notin \bold S\Ra\CA_m$.   
Hence the alledged varieties are not atom --canonical. 
\end{proof}

\subsection{ Blowing up and blurring 
finite rainbow algebras}

In Theorem \ref{ANT}, we used a single blow up and blur construction to prove non-atom--canonicity of $\sf RRA$ and ${\sf RCA}_n$ for $2<n<\omega$.
To obtain finer results, we use {\it two blow up and blur constructions} applied to rainbow algebras.
For the $\RA$ case, following Hirsch and Hodkinson,  we blow up and blur 
the finite rainbow relation algebra (denoted below by) $\bold R_{4,3}$. For the $\CA$ case we blow up and blur 
the finite rainbow $\CA_n$  (denoted below by) $\A_{n+1, n}$. 
To put things into a unified perspective, we formulate a definition: 

\begin{definition}\label{blow} Let $\bold M$ be a variety 
of completely additive $\sf BAO$s.

\begin{enumarab}

\item Let $\A\in \bold M$ be a finite algebra. We say that {\it $\D\in \bold M$ is obtained by blowing up and blurring $\A$} if $\D$ is atomic,  
$\A$ does not embed in $\D$, but $\A$ embeds into $\Cm\At\D$. 

\item Assume that $\sf K\subseteq L\subseteq \bold M$, such that $\bold S\sf L=\sf L$.

(a) We say that {\it $\sf K$ is {\bf not} atom-canonical with respect to $\sf L$} if there exists an atomic $\D\in \sf K$ 
such that $\Cm\At\D\notin \sf L$.  In particular, $\sf K$ is not atom--canonical.

(b)  We say that a finite algebra $\A\in \bold M$ {\it detects} that 
$\bold K$ is not atom--canonical with respect to $\sf L$, if $\A\notin \sf L$, 
and there is a(n atomic)  $\D\in \sf K$ 
obtained by blowing up and blurring $\A$.
\end{enumarab}
\end{definition}

\begin{corollary}\label{cp} Let $2<n<\omega$. Then for any finite $j>0$, ${\sf RRA}\cap \Ra\CA_{2+j}$ is not atom-canonical with respect to $\sf RRA$, 
and $\RCA_n\cap \Nr_n\CA_{n+j}$ is not atom--canonical with respect to $\RCA_n.$
\end{corollary}
\begin{proof} Follows from Theorem \ref{ANT}, cf. \cite{ANT}. In more detail, let $2<n\leq l<\omega$. Choose $k\in \omega$ so that the finite Maddux algebra 
$\R=\mathfrak{E}_k(2, 3)$, with $k$ non--identity atoms, has a strong $l$ --blur
$(J, E)$. Let ${\cal R}={\sf split}(\R, J, E)\in {\sf  RRA}$ and ${\cal C}={\sf split}_n(\R, J, E)\in \RCA_n$. 
Then both $\Cm\At(\cal R)$ and $\Cm\At(\cal C)$ are not representable,  
${\cal R}=\Ra{\cal C}=\Nr_n({\sf split}_l (\R, J, E))=\Ra({\sf split}_l (\R, J, E))$, 
and ${\cal C}=\Nr_n{\sf split}_l(\R, J, E)$. In particular,  ${\cal R}\in \Ra\CA_l\cap \sf RRA$ and ${\cal C}\in \Nr_n\CA_l\cap \RCA_n$; 
and their \de\ completions are not representable. 
\end{proof}

Fix $2<n<\omega$. We use finite algebras that {\it detect} that $\sf RRA$, $\RCA_n$ are not atom--canonical 
with respect to $\bold S\Ra\CA_6$  and $\bold S\Nr_n\CA_{n+3}$, respectively. 
The first result is due to Hirsch and Hodkinson \cite{HHbook}, the second result is proved for $\CA$s
in \cite{mlq}.  Below we generalize the proof to other $\CA$--like algebras of relations.

In all cases we use a rainbow construction. For both $\RA$s and $\CA$s (and its likes) 
rainbow constructions, as mentioned in the introduction, hinge on a two player zero sum \ef\ forth game 
between two players \pa\ and \pe. This game is  
played on two  relational structures (that are usually surprisingly simple, like finite ordinals viewed as complete irreflexive graphs)
$\sf G$  (the greens) and $\sf R$ (the reds), which, together with other colours, uniquely determines a rainbow atom structure,  be it of 
an $\RA$ or a $\CA_n$.
 
A novelty that occurs here is that the presentation of both constructions for $\RA$s and $\CA$s, 
is implemented in the framework of entirely analagous blow up and blur constructions applied to strikingly similar finite rainbow atom structures.
In both cases, the relational structures 
$\sf G$ and $\sf R$ used
satisfy $|\sf G|=|\sf R|+1$. For $\sf RA$, ${\sf R}=3$ and for $\CA_n$s, ${\sf R}=n$ (the dimension), where
the finite ordinals $3$ and $n$ are viewed as complete irreflexive graphs. 
Worthy of note, is that it is commonly accepted that relation algebras have dimension three being a natural habitat for three variable first order 
logic.  Nevertheless, sometimes it is argued that the dimension should be three and a half in the somewhat loose sense that $\sf RA$s 
lie `halfway' between $\CA_3$ and $\CA_4$ manifesting behaviour of each. \\

{\bf Relation algebras:} 
We briefly review the construction in \cite[17.32, 17.34, 17.36]{HHbook} for relation algebras, presenting it in the framework of blowing up and blurring a finite rainbow relation 
algebra in the sense of the first item of Definition \ref{blow}.
Let $2\leq n\leq \omega$ and $r\leq \omega$.
Let $\R$ be an atomic relation algebra. Then the $r$--rounded game $G^n_r(\At\R)$ \cite[Definition 12.24]{HHbook} 
is the (usual) atomic game played on networks of an atomic relation algebra 
$\R$  using $n$ nodes.
  
Let $L$ be a relational signature. Let $\sf G$ (the greens) and  $\sf R$ (the reds)  be $L$ structures and $p,r\leq \omega$.
The game ${\sf EF}_r^p(\sf G, R)$,  defined in \cite[Definition 16.1.2]{HHbook}, is an \ef\ forth `pebble game'
with $r$ rounds and $p$ pairs of pebbles.  In \cite[16.2]{HHbook}, a relation algebra {\it rainbow atom structure} is associated 
for relational structures $\sf G$ and $\sf R$. We denote by  $\bold R_{A,B}$ the (full) complex algebra
over this atom structure.  The {\bf Rainbow Theorem} \cite[Theorem 16.5]{HHbook} states that: 
{\it If $\sf G, R$ are relational structures and $p,r\leq \omega$, then \pe\ has a \ws\ 
in $G_{1+r}^{2+p}(\bold R_{\sf G, \R})$
$\iff$ she has a winning strategy in ${\sf EF}_r^p(\sf G,R)$.}\\
For $5\leq l<\omega$, $\RA_l$  is the class of relation algebras whose canonical extensions have an $l$--dimensional relational basis \cite[Definition 12.30]{HHbook}.
$\RA_l$ is a variety containing properly the variety $\bold S\Ra\CA_l$.  Furthermore, $\R\in \RA_l$ $\iff$ 
\pe\ has a \ws\ in $G_{\omega}^n(\At\R).$
Cf. \cite[Proposition 12.31]{HHbook} and  \cite[Remark 15.13]{HHbook}. We now show:

\begin{theorem} \label{b2} (Hirsch-Hodkinson) For any $k\geq 6$, the varieties $\RA_k$ and $\bold S\Ra\CA_k$ are not atom--canonical. In fact, $\sf RRA$ is 
not atom--canonical with respect  to any of 
the aforementioned varieties. 
\end{theorem}
\begin{proof} We follow the notation in \cite[Lemmata 17.32, 17.34, 17.35, 17.36]{HHbook2} with the sole exception that we denote by 
$m$ (instead of $\bold K_m$) 
the complete irreflexive graph on $m$ defined the obvious way; that is we identify this graph with its set of vertices. 
Fix  $2<n<m<\omega$. Let $\R=\bold R_{m, n}$. Then by the Rainbow Theorem  \pe\ has a \ws\ in $G_{m+1}^{m+2}(\At\R)$, since it clealy 
has a \ws\ 
in the \ef\ game ${\sf EF}_m^m(m, n)$ because $m$ is `longer' than $n$. Then $\R\notin \RA_{m+2}$ by \cite[Propsition 12.25, Theorem 13.46 $(4)\iff (5)$]{HHbook}, so 
$\R\notin  \bold S\Ra\CA_{m+2}$.  Next one `splits'  every red atom  to $\omega$--many copies obtaining the infinite atomic countable 
(term) relation algebra denoted in {\it op.cit} by $\T$, which we denote by ${\sf split}(\R, \r, \omega)$ (blowing up the reds by splitting each into $\omega$--many subatoms) 
with atom structure $\alpha$, cf. \cite[item (4) top of p. 532]{HHbook}. 
Then $\mathfrak{Cm}\alpha\notin \bold S\Ra\CA_{m+2}$ 
because $\R$ embeds into $\Cm\alpha$ by mapping every red
to the join of its copies, and $\bold S\Ra\CA_{m+2}$ is closed under $\bold S$. 
Finally, one (completely) represents 
(the canonical extension of) 
${\sf split}(\R, \r, \omega)$ like in \cite{HHbook}. By taking $m=4$ and $n=3$ the required follows.
\end{proof}
{\bf Cylindric--like algebras:} From now on,  unless otherwise indicated, $n$ is fixed to be  a finite ordinal $>2$.
To define certain deterministic games   
we recall the notions of {\it atomic networks basis} and {\it atomic games} \cite{HHbook, HHbook2}. 
Let $i<n$. For $n$--ary sequences $\bar{x}$ and $\bar{y}$, we write $\bar{x}\equiv _i \bar{y}\iff \bar{y}(j)=\bar{x}(j)$ for all $j\neq i.$  

\begin{definition}\label{game} Fix finite $n>2$ and assume that $\A\in \CA_n$ is atomic.

(1) An {\it $n$--dimensional atomic network} on $\A$ is a map $N: {}^n\Delta\to  \At\A$, where
$\Delta$ is a non--empty set of {\it nodes}, denoted by $\nodes(N)$, satisfying the following consistency conditions for all $i<j<n$: 
\begin{itemize}
\item If $\bar{x}\in {}^n\nodes(N)$  then $N(\bar{x})\leq {\sf d}_{ij}\iff \bar{x}_i=\bar{x}_j$,
\item If $\bar{x}, \bar{y}\in {}^n\nodes(N)$, $i<n$ and $\bar{x}\equiv_i \bar{y}$, then  $N(\bar{x})\leq {\sf c}_iN(\bar{y})$.
\end{itemize}
For $n$--dimensional atomic networks $M$ and $N$,  we write $M\equiv_i N\iff M(\bar{y})=N(\bar{y})$ for all $\bar{y}\in {}^{n}(n\sim \{i\})$.

(2)   Assume that $m, k\leq \omega$. 
The {\it atomic game $G^m_k(\At\A)$, or simply $G^m_k$}, is the game played on atomic networks
of $\A$ using $m$ nodes and having $k$ rounds \cite[Definition 3.3.2]{HHbook2}, where
\pa\ is offered only one move, namely, {\it a cylindrifier move}: 

Suppose that we are at round $t>0$. Then \pa\ picks a previously played network $N_t$ $(\nodes(N_t)\subseteq m$), 
$i<n,$ $a\in \At\A$, $x\in {}^n\nodes(N_t)$, such that $N_t(\bar{x})\leq {\sf c}_ia$. For her response, \pe\ has to deliver a network $M$
such that $\nodes(M)\subseteq m$,  $M\equiv _i N$, and there is $\bar{y}\in {}^n\nodes(M)$
that satisfies $\bar{y}\equiv _i \bar{x}$ and $M(\bar{y})=a$.  

We write $G_k(\At\A)$, or simply $G_k$, for $G_k^m(\At\A)$ if $m\geq \omega$.

(3) The $\omega$--rounded game $\bold G^m(\At\A)$ or simply $\bold G^m$ is like the game $G_{\omega}^m(\At\A)$ 
except that \pa\ has the advantage 
to reuse the $m$ nodes in play.
\end{definition}
For a class $\bold K$ of $\sf BAO$s, $\bold K^{\sf ad}$ denotes the class of {\it completely additive} algebras in $\bold K$.
 \begin{lemma}\label{n}
Let $\K$ be any class having between $\Sc$ and $\QEA$, 
$\A\in \K_n$ and $\A\in \bold S_c\Nr_n\K_m^{\sf ad}$. Then \pe\ has a \ws\ in $\bold G^m(\At\A).$ 
In particular, If $\A$ is finite and \pa\ has a \ws\ in $\bold G^m(\At\A)$, then $\A\notin \bold S\Nr_n\K_m^{\sf ad}$. 
\end{lemma}
\begin{proof} We prove the result only for $\CA$s. The proof  lifts the ideas
in \cite[Lemma 29, 26, 27]{r} formulated for relation algebras.

First a piece of notation. Let $m$ be a finite ordinal $>0$. An $\sf s$ word is a finite string of substitutions $({\sf s}_i^j)$ $(i, j<m)$,
a $\sf c$ word is a finite string of cylindrifications $({\sf c}_i), i<m$;
an $\sf sc$ word $w$, is a finite string of both, namely, of substitutions and cylindrifications.
An $\sf sc$ word
induces a partial map $\hat{w}:m\to m$:
$\hat{\epsilon}=Id,$ $\widehat{w_j^i}=\hat{w}\circ [i|j]$
and $\widehat{w{\sf c}_i}= \hat{w}\upharpoonright(m\smallsetminus \{i\}).$
If $\bar a\in {}^{<m-1}m$, we write ${\sf s}_{\bar a}$, or
${\sf s}_{a_0\ldots a_{k-1}}$, where $k=|\bar a|$,
for an  arbitrary chosen $\sf sc$ word $w$
such that $\hat{w}=\bar a.$
Such a $w$  exists by \cite[Definition~5.23 ~Lemma 13.29]{HHbook}.

Fix $2<n<m$. Assume that $\C\in\CA_m$, $\A\subseteq_c\Nr_n\C$ is an
atomic $\CA_n$ and $N$ is an $\A$--network with $\nodes(N)\subseteq m$. Define
$N^+\in\C$ by
\[N^+ =
 \prod_{i_0,\ldots, i_{n-1}\in\nodes(N)}{\sf s}_{i_0, \ldots, i_{n-1}}{}N(i_0,\ldots, i_{n-1}).\]
For a network $N$ and  function $\theta$,  the network
$N\theta$ is the complete labelled graph with nodes
$\theta^{-1}(\nodes(N))=\set{x\in\dom(\theta):\theta(x)\in\nodes(N)}$,
and labelling defined by
$$(N\theta)(i_0,\ldots, i_{n-1}) = N(\theta(i_0), \theta(i_1), \ldots,  \theta(i_{n-1})),$$
for $i_0, \ldots, i_{n-1}\in\theta^{-1}(\nodes(N))$.
Then the following hold:
\begin{enumerate}
\item for all $x\in\C\setminus\set0$ and all $i_0, \ldots, i_{n-1} < m$, there is $a\in\At\A$, such that
${\sf s}_{i_0,\ldots, i_{n-1}}a\;.\; x\neq 0$,

\item for any $x\in\C\setminus\set0$ and any
finite set $I\subseteq m$, there is a network $N$ such that
$\nodes(N)=I$ and $x\cdot N^+\neq 0$. Furthermore, for any networks $M, N$ if
$M^+\cdot N^+\neq 0$, then
$M\restr {\nodes(M)\cap\nodes(N)}=N\restr {\nodes(M)\cap\nodes(N)},$

\item if $\theta$ is any partial, finite map $m\to m$
and if $\nodes(N)$ is a proper subset of $m$,
then $N^+\neq 0\rightarrow {(N\theta)^+}\neq 0$. If $i\not\in\nodes(N),$ then ${\sf c}_iN^+=N^+$.

\end{enumerate}

Since $\A\subseteq _c\Nr_n \C$, then $\sum^{\C}\At\A=1$. For (1), we have, {\bf ${\sf s}^i_j$ is a
completely additive operator (any $i, j<m$)}, hence ${\sf s}_{i_0,\ldots, i_{n-1}}$
is, too.
So $\sum^{\C}\set{{\sf s}_{i_0\ldots, i_{n-1}}a:a\in\At(\A)}={\sf s}_{i_0\ldots i_{n-1}}
\sum^{\C}\At\A={\sf s}_{i_0\ldots, i_{n-1}}1=1$ for any $i_0,\ldots, i_{n-1}<m$.  Let $x\in\C\setminus\set0$.  Assume for contradiction
that  ${\sf s}_{i_0\ldots, i_{n-1}}a\cdot x=0$ for all $a\in\At\A$. Then  $1-x$ will be
an upper bound for $\set{{\sf s}_{i_0\ldots i_{n-1}}a: a\in\At\A}.$
But this is impossible
because $\sum^{\C}\set{{\sf s}_{i_0\ldots, i_{n-1}}a :a\in\At\A}=1.$
To prove the first part of (2), we repeatedly use (1).
We define the edge labelling of $N$ one edge
at a time. Initially, no hyperedges are labelled.  Suppose
$E\subseteq\nodes(N)\times\nodes(N)\ldots  \times\nodes(N)$ is the set of labelled hyperedges of $
N$ (initially $E=\emptyset$) and
$x\;.\;\prod_{\bar c \in E}{\sf s}_{\bar c}N(\bar c)\neq 0$.  Pick $\bar d$ such that $\bar d\not\in E$.
Then by (1) there is $a\in\At(\A)$ such that
$x\;.\;\prod_{\bar c\in E}{\sf s}_{\bar c}N(\bar c)\;.\;{\sf s}_{\bar d}a\neq 0$.
Include the hyperedge $\bar d$ in $E$.  We keep on doing this until eventually  all hyperedges will be
labelled, so we obtain a completely labelled graph $N$ with $N^+\neq 0$.
it is easily checked that $N$ is a network.

For the second part of $(2)$, we proceed contrapositively. Assume that there is
$\bar c \in{}\nodes(M)\cap\nodes(N)$ such that $M(\bar c )\neq N(\bar c)$.
Since edges are labelled by atoms, we have $M(\bar c)\cdot N(\bar c)=0,$
so
$0={\sf s}_{\bar c}0={\sf s}_{\bar c}M(\bar c)\;.\; {\sf s}_{\bar c}N(\bar c)\geq M^+\cdot N^+$.
A piece of notation. For $i<m$, let $Id_{-i}$ be the partial map $\{(k,k): k\in m\smallsetminus\{i\}\}.$
For the first part of (3)
(cf. \cite[Lemma~13.29]{HHbook} using the notation in {\it op.cit}), since there is
$k\in m\setminus\nodes(N)$, \/ $\theta$ can be
expressed as a product $\sigma_0\sigma_1\ldots\sigma_t$ of maps such
that, for $s\leq t$, we have either $\sigma_s=Id_{-i}$ for some $i<m$
or $\sigma_s=[i/j]$ for some $i, j<m$ and where
$i\not\in\nodes(N\sigma_0\ldots\sigma_{s-1})$.
But clearly  $(N Id_{-j})^+\geq N^+$ and if $i\not\in\nodes(N)$ and $j\in\nodes(N)$, then
$N^+\neq 0 \rightarrow {(N[i/j])}^+\neq 0$.
The required now follows.  The last part is straightforward.
Using the above proven facts,  we are now ready to show that \pe\  has a \ws\ in $F^m$. She can always
play a network $N$ with $\nodes(N)\subseteq m,$ such that
$N^+\neq 0$.\\
In the initial round, let \pa\ play $a\in \At\A$.
\pe\ plays a network $N$ with $N(0, \ldots, n-1)=a$. Then $N^+=a\neq 0$.
Recall that here \pa\ is offered only one (cylindrifier) move.
At a later stage, suppose \pa\ plays the cylindrifier move, which we denote by
$(N, \langle f_0, \ldots, f_{n-2}\rangle, k, b, l).$
He picks a previously played network $N$,  $f_i\in \nodes(N), \;l<n,  k\notin \{f_i: i<n-2\}$,
such that $b\leq {\sf c}_l N(f_0,\ldots,  f_{i-1}, x, f_{i+1}, \ldots, f_{n-2})$ and $N^+\neq 0$.
Let $\bar a=\langle f_0\ldots f_{i-1}, k, f_{i+1}, \ldots f_{n-2}\rangle.$
Then by  second part of  (3)  we have that ${\sf c}_lN^+\cdot {\sf s}_{\bar a}b\neq 0$
and so  by first part of (2), there is a network  $M$ such that
$M^+\cdot{\sf c}_{l}N^+\cdot {\sf s}_{\bar a}b\neq 0$.
Hence $M(f_0,\dots, f_{i-1}, k, f_{i-2}, \ldots$ $, f_{n-2})=b$,
$\nodes(M)=\nodes(N)\cup\set k$, and $M^+\neq 0$, so this property is maintained.
The last part follows by observing that for any $\C\in \CA_n$, if $\C\in \bold S{\sf Nr}_n\CA_m\implies \C^+\in \bold S_c{\sf Nr}_n\CA_m$
(where $\C^+$ is the canonical extension of $\C$) and if  $\C$ is finite,  then of course $\C=\C^+$.
\end{proof}

The previous proof depends on the complete additivity of the ${\sf s}_i^j$s $(i<j<m)$, 
which is not necessarily true for the diagonal free reducts of $\sf QEA$s, namely, 
$\sf Sc$s and  ${\sf QA}$s. Examples can be found in \cite{AGMNS}.

For rainbow constructions for $\CA$s,
we follow \cite{HH,HHbook2}. For notions, like {\it coloured graphs}, {\it cones} which are special coloured graphs, we refer to
\cite{HH} from which we (explicitly) recall:

{\it Let $i\in {\sf G}$, and let $M$ be a coloured graph consisting of $n$ nodes
$x_0,\ldots,  x_{n-2}, z$. We call $M$ {\it an $i$ - cone} if $M(x_0, z)=\g_0^i$
and for every $1\leq j\leq m-2$, $M(x_j, z)=\g_j$,
and no other edge of $M$
is coloured green.
$(x_0,\ldots, x_{n-2})$
is called  the base of the cone, $z$ the apex of the cone
and $i$ the tint of the cone}.

Given relational structures 
$\sf G$ (the greens) and $\sf R$ (the reds) the rainbow 
atom structure of a $\QEA_n$  consists of equivalence classes of surjection from $n$ to 
coloured graphs for the relation defined in \cite {HH}.

The accessibilty (binary relations) corresponding to 
cylindric operations are like in \cite{HH}.  For transpositions ($[i, j]$, $i<j<n$) they are defined on the atoms of the rainbow atom structure via
$[a]S_{[i, j]}[b]\iff\ a=b\circ [i,j]$ where $a$ and $b$ are 
surjections into coloured graphs, and $[a]$ is the equivalent class containing $a$, and similary for $[b]$.
The hitherto constructed ${\sf QEA}_n$ atom structure depends uniquely on $\sf G$ 
and $\sf R$. 

For $2<n<\omega$, we use the graph version of the games $G_{\omega}^m(\beta)$ and $\bold G^m(\beta)$ where $\beta$ is a $\QEA_n$ 
rainbow atom structure, cf. \cite[4.3.3]{HH}. 
The typical \ws\ of \pa\ in the rainbow game played on coloured graphs played between \pe\ and \pa\  
is bombarding \pe\ with $i$--cones, $i\in \sf G$, having the 
same base 
and distinct green tints.  To respect the rules of the game \pe\ has to choose a red label for appexes of two succesive cones.  
Eventually, running out of `suitable reds',  \pe\ is forced to play an inconsistent triple of reds where indices do not match.
Thus \pa\ wins on a red clique (a graph all of whose edges are labelled by a red) with 
the \ws\ for ether player  dictated by her(his) \ws\ in a simple private \ef\ forth  
game played on the relational structures $\sf G$ and $\sf R$ with $r\leq \omega$ rounds and $p\leq \omega$ pairs of pebbles (recalled and denoted above by ${\sf EF}_r^p(\sf G, \sf R)$). 

The (complex) rainbow algebra based on $\sf G$ and $\sf R$ is denoted by $\A_{\sf G, R}$.
The dimension $n$ will always be clear from context.

We are now prepared to prove the following Theorem:
\begin{theorem}\label{can} 
Let $n$ be a finite ordinal $>2$ and $\K$ a class between $\Sc$ and $\QEA$.
Assume that $m\geq n+3$. Then ${\sf RK}_n$ is not-atom canonical 
with respect to $\bold S{\sf Nr}_n\K_m$.     
 \end{theorem}
\begin{proof}
The idea for $\CA$s is like that for $\RA$s by blowing up and blurring (the $\CA$ reduct of) $\A_{n+1, n}$ in place of $\bold R_{4,3}$ blown up and blurred in the $\RA$ case.
We work with $m=n+3$ and any $\K$ between $\Sc$ and $\QEA$. 
This gives the result for any larger $m$. 

{\bf Blowing up and blurring  $\A_{n+1, n}$ forming a weakly representable atom structure $\At$}:
Take the finite rainbow ${\sf QEA}_n$,  $\A_{n+1, n}$
where the reds $\sf R$ is the complete irreflexive graph $n$, and the greens
are  $\{\g_i:1\leq i<n-1\}
\cup \{\g_0^{i}: 1\leq i\leq n+1\}$, so that ${\sf G}=n+1$, endowed with the quasi--polyadic equality operations.
Denote its finite atom structure by ${\bf At}_f$; 
so that ${\bf At}_f=\At(\A_{n+1, n})$.
One  then replaces the red colours 
of the finite rainbow algebra of $\A_{n+1, n}$ each by  infinitely many countable reds (getting their superscripts from $\omega$), obtaining this way a weakly representable atom structure $\bf At$.
The cylindric reduct of the resulting atom structure after `splitting the reds', namely, $\bf At$,  is 
like the weakly (but not strongly) representable 
atom structure of the atomic, countable and simple algebra $\A$ as defined in \cite[Definition 4.1]{Hodkinson}; the sole difference is that we have $n+1$ greens
and not $\omega$--many as is the case in \cite{Hodkinson}; also we count it the polyadic operations of subtitutions. 
We denote the resulting term $\QEA_n$,  $\Tm\bf At$ by ${\sf split}(\A_{n+1, n}, \r, \omega)$ short hand for  blowing 
up $\A_{n+1,n}$ by splitting each {\it red graphs (atoms)} into $\omega$ many. 
By a  red graph is meant (an equivalence class of) a surjection $a:n\to \Delta$, where
$\Delta$ is a coloured graph in the rainbow signature of $\A_{n+1, n}$ with at least one edge labelled by a red label 
(some $\r_{ij}$, $i<j<n)$. 
It can be shown exactly like in \cite{Hodkinson} that \pe\ can win the rainbow $\omega$--rounded game
and build an $n$--homogeneous model $\Mo$ by using a shade of red $\rho$ {\it outside} the rainbow signature, when
she is forced a red;  \cite[Proposition 2.6, Lemma 2.7]{Hodkinson}.  Using this, one proves like in {\it op.cit} 
that  ${\sf split}(\A_{n+1, n}, \r, \omega)$ is representable as a set algebra having top element $^n\Mo$.
(The term algebra in \cite{Hodkinson}; which is the subalgebra generated by the atoms of $\A$ as defined in \cite[Definition 4.1]{Hodkinson}
is just ${\sf split}(\A_{\omega, n}, \r, \omega)$.)

{\bf Embedding $\A_{n+1, n}$ into $\Cm(\At({\sf split}(\A_{n+1, n}, \r, \omega)))$:} Let ${\sf CRG}_f$ be the class of coloured graphs on 
${\bf At}_f$ and $\sf CRG$ be the class of coloured graph on $\bf At$. We 
can assume that  ${\sf CRG}_f\subseteq \sf CRG$.
Write $M_a$ for the atom that is the (equivalence class of the) surjection $a:n\to M$, $M\in \sf CGR$.
Here we identify $a$ with $[a]$; no harm will ensue.
We define the (equivalence) relation $\sim$ on $\bf At$ by
$M_a\sim N_b$, $(M, N\in {\sf CGR})$
$\iff$ they are identical everywhere except at possibly at red edges:\\
$M_a(a(i), a(j))=\r^l\iff N_b(b(i), b(j))=\r^k,  \text { for some $l,k$}\in \omega.$\\
We say that $M_a$ is a {\it copy of $N_b$} if $M_a\sim N_b$ (by symmetry $N_b$ is a copy of $M_a$.) 
Indeed, the relation `copy of' is an equivalence relation on $\bf At$.  An atom $M_a$ is called a {\it red atom}, if $M$ has at least one red edge. 

Any red atom has $\omega$ many copies that are {\it cylindrically equivalent}, in the sense that, if $N_a\sim M_b$ with one (equivalently both) red,
with $a:n\to N$ and  $b:n\to M$, then we can assume that $\nodes(N) =\nodes(M)$ 
and that for all $i<n$, $a\upharpoonright n\sim\{i\}=b\upharpoonright n\sim \{i\}$.
In $\Cm\bf At$, we write $M_a$ for $\{M_a\}$ 
and we denote suprema taken in $\Cm\bf At$, possibly finite, by $\sum$.
Define the map $\Theta$ from $\A_{n+1, n}=\Cm{\bf At}_f$ to $\Cm\bf At$,
by specifing first its values on ${\bf At}_f$,
via $M_a\mapsto \sum_jM_a^{(j)}$ where $M_a^{(j)}$ is a copy of $M_a$. 
So each atom maps to the suprema of its  copies.  
This map is well-defined because $\Cm\bf At$ is complete. 
We check that $\Theta$ is an injective homomorphim. Injectivity follows from $M_a\leq \Theta(M_a)$, hence $\Theta(x)\neq 0$ 
for every atom $x\in \At(\A_{n+1, n})$.
Now we check the presevation of operations. 
The Boolean join is obvious.

\begin{itemize}
\item For complementation: It suffices to check preservation of  complementation `at atoms' of ${\bf At}_f$. 
So let $M_a\in {\bf At}_f$ with $a:n\to M$, $M\in \sf CGR_f\subseteq \sf CGR$. Then: 

$$\Theta(\sim M_a)=\Theta(\bigcup_{[b]\neq [a]} M_b)
=\bigcup_{[b]\neq [a]} \Theta(M_b)
=\bigcup_{[b]\neq [a]}\sum_j M_b^{(j)}$$
$$=\bigcup_{[b]\neq [a]}\sim \sum_j[\sim (M_a)^{(j)}]
=\bigcup_{[b]\neq [a]}\sim \sum_j[(\sim M_b)^j]
=\bigcup_{[b]\neq [a]}\bigwedge_j M_b^{(j)}$$
$$=\bigwedge_j\bigcup_{[b]\neq [a]}M_b^{(j)}
=\bigwedge_j(\sim M_a)^{j}
=\sim (\sum M_a^j)
=\sim \Theta(a)$$

\item Diagonal elements. Let $l<k<n$. Then:
\begin{align*}
M_x\leq \Theta({\sf d}_{lk}^{\Cm{\bf At}_f})&\iff\ M_x\leq \sum_j\bigcup_{a_l=a_k}M_a^{(j)}\\
&\iff M_x\leq \bigcup_{a_l=a_k}\sum_j M_a^{(j)}\\
&\iff  M_x=M_a^{(j)}  \text { for some $a: n\to M$ such that $a(l)=a(k)$}\\
&\iff M_x\in {\sf d}_{lk}^{\Cm\bf At}.
\end{align*}

\item Cylindrifiers. Let $i<n$. By additivity of cylindrifiers, we restrict our attention to atoms 
$M_a\in {\bf At}_f$ with $a:n\to M$, and $M\in {\sf CRG}_f\subseteq \sf CRG$. Then: 
\begin{align*}
\Theta({\sf c}_i^{\Cm{\bf At}_f}M_a)&=f (\bigcup_{[c]\equiv_i[a]} M_c)\\
&=\bigcup_{[c]\equiv_i [a]}\Theta(M_c)\\
&=\bigcup_{[c]\equiv_i [a]}\sum_j M_c^{(j)}\\
&=\sum_j \bigcup_{[c]\equiv_i [a]}M_c^{(j)}\\
&=\sum _j{\sf c}_i^{\Cm\bf At}M_a^{(j)}\\
&={\sf c}_i^{\Cm\bf At}(\sum_j M_a^{(j)})\\
&={\sf c}_i^{\Cm\bf At}\Theta(M_a).
\end{align*}

\item Substitutions: Let $i, k<n$. By additivity of the ${\sf s}_{[i,k]}$s, we again restrict ourselves to atoms of the form $M_a$ as specified in the previous items.
Now computing we get:
$\Theta({\sf s}_{[i,k]}^{\Cm{\bf At}_f} M_a)= \Theta(M_{a\circ [i,k]})=  \sum_j^{\Cm\bf At}(M_{a\circ [i,k]}^{(j)})=\sum_j {\sf s}_{[i,k]}^{\Cm\bf At}{} M_a^{(j)}={\sf s}_{[i,k]}^{\Cm \bf At}{}(\sum_{j}M_a^{(j)})=
{\sf s}_{[i,k]}^{\Cm\bf At}{}\Theta(M_a).$
\end{itemize}
We have proved that $\A_{n+1, n}$ embeds into $\Cm \bf At$, so that it is not blurred at the level of the last 
complex algebra, that is to say, $\Theta$ is the required embedding.

{\bf \pa\ has  a  \ws\ in $G^{n+3}\At(\A_{n+1, n})$:} It is straightforward to show that 
\pa\ has \ws\ first in the  \ef\ forth  private game played between \pe\ and \pa\ on the complete
irreflexive graphs $n+1$ and $n$ in 
$n+1$ rounds, namely, the game ${\sf EF}_{n+1}^{n+1}(n+1, n)$ \cite [Definition 16.2]{HHbook2}.
 \pa\  lifts his \ws\ from the private \ef\ forth game, to the graph game on ${\bf At}_f=\At(\A_{n+1,n})$ 
\cite[pp. 841]{HH} forcing a
win using $n+3$ nodes. 
He bombards \pe\ with cones
having  common
base and distinct green 
tints until \pe\ is forced to play an inconsistent red triangle (where indicies of reds do not match).
By lemma \ref{n}, $\Rd_{sc}\A_{n+1, n}\notin
\bold S_c{\sf Nr}_n\Sc_{n+3}^{\sf ad}$. Since $\A_{n+1, n}$ is finite, then $\Rd_{sc}\A_{n+1, n}$
is not in $\bold S\Nr_n\Sc_{n+3}^{\sf ad}$. Using that $\A_{n+1, n}$ is finite again, one infers that $\Rd_{sc}\A_{n+1, n}\notin \bold S\Nr_n\Sc_{n+3}$. To prove the last part,
assume for contradiction that $\Rd_{sc}\A_{n+1, n}\in \bold S\Nr_n\Sc_{n+3}$. Then $(\Rd_{sc}\A_{n+1, n})^+=\Rd_{sc}\A_{n+1, n}\in \bold S_c\Nr_n\Sc_{n+3}$ 
which is impossible.
Obeserve that we removed the condition of additivity obtaining a bigger class.
To see why, assume for contradiction,  that $\D=\Rd_{sc}\A_{n+1, n}\in \bold S\Nr_n\Sc_{n+3}$. 
Then we can assume that $\D\subseteq \Nr_n\D'$, where $\D'\in \Sc_{n+3}$, and also without loss, that $\D'$ 
is generated by 
$\D$. Since the operations are finite, so $\D'$ is finite from which it readily follows that $\D'$ 
is completely additive. We can now conclude that $\D\in \bold S_c\Nr_n\Sc_{n+3}^{\sf ad}$,
but this is a contradiction. We have proved that 
$\Rd_{sc}\A_{n+1, n}\notin \bold S\Nr_n\Sc_{n+3}$.

{\bf Non--atom canonicity:} But $\A_{n+1,n}$ embeds into $\Cm\At\A$,
hence $\Rd_{sc}\Cm\At\A=\Cm\Rd_{sc}\At\A$
is outside the variety $\bold S{\sf Nr}_n\Sc_{n+3}$, as well. We have proved that ${\sf RK}_n$ is not atom--canonical with 
respect to $\bold S\Nr_n\K_{n+3}$.
Thus  any class $\sf K$ (like $\CA$ and $\sf QA$) between $\Sc$ and $\QEA$, 
and any $k\geq 3$, 
${\sf RK}_n$ is not atom--canonical. 
\end{proof} It is known that $\A\in \CA_n$ is representable $\iff$ it has an $\omega$--dilation. 
By adjusting the number of greens in the proof of Theorem \ref{can} to be $n+1$ 
we got a result finer than Hodkinson's \cite{Hodkinson}, where Hodkinson uses an `overkill' of infinitely many greens excluding $\omega$--dilations of 
$\Cm\At\A$ with $\A$ as defined in \cite[Definition 4.1]{HHbook}. In the above proof we excluded $m$ dilations fof any $m\geq n+3$. \\

We find the following question stressing:
{\bf Is there a completely additive variety $\V$ 
of $\sf BAO$s that is atom-canonical, but not closed under \de\ completions?}.

To formulate the next Corollary we need some preparation. For $2<n<\omega$, and $l$ any ordinal, 
let $\RCA_{n, l}=\{\A\in \RCA_n\cap \Nr_n\CA_{n+l}: \Cm\At\A\notin \RCA_n\}$
and ${\sf RRA}_l=\{\A\in {\sf RRA}\cap \Ra\CA_l: \Cm\At\A\notin \sf RRA\}$.
For a class $\bold K$ of $\sf BAO$s, let $\bold K\cap{\sf Cat}$ denote the class of algebras in 
$\bold K$ with countably many atoms. We refer the reader to \cite{g} for the definition 
of single--persistence. 
From the constructions in Theorems, \ref{ANT}, \ref{b2} and \ref{can}, using known Theorems, we get the following Corollary.

\begin{corollary}\label{sahlqvist} Let $2<n<\omega$ and $k\geq 3$. Then the following hold:

\begin{enumerate}

\item  ${\sf RRA}_l\cap {\sf Cat}\neq \emptyset \iff l<\omega$ and ${\sf RCA}_{n, l}\cap {\sf Cat}\neq \emptyset \iff l<\omega$,

\item For any ordinal $k$, $\Ra\CA_{3+k}\cap \sf RRA\cap \bf At\nsubseteq {\sf CRRA}$ 
and $\Nr_n\CA_{n+k}\cap \RCA_n\cap {\bf At}\nsubseteq {\sf CRCA}_n$,

\item $\Cm\At{\sf RRA}\nsubseteq \bold S\Ra\CA_{3+k}$
and $\Cm\At\RCA_n\nsubseteq \bold S\Nr_n\CA_{n+k},$
\item There exist two atomic 
algebras in $\sf RA$ with the same atom structure, only one of which is 
representable and the other is outside $\bold S\Ra\CA_{3+k}$. An analagous result holds for $\CA$s by replacing 
$\bold S\Ra\CA_{3+k}$ by  $\bold S\Nr_n\CA_{n+k}$,

\item Any variety between 
$\sf RRA$ and $\bold S\Ra\CA_{3+k}$, as well as any variety between ${\sf RCA}_n$ and $\bold S\Nr_n\CA_{n+k}$,  
are not atom--canonical, not single--persistent,  and not closed under \de\ completions. 
 \item There exist  complete algebras outside each of the the varieties $\bold S\Ra\CA_{n+k}$ and $\bold S\Nr_n\CA_{n+k}$ with a dense representable
subalgebra,

\item There exists an atomic ${\cal T}\in \sf RRA$ and an atomic 
$\A\in \RCA_n$ such that their \de\ completions do not  embed into their canonical 
extensions.
\end{enumerate}
\end{corollary}
\begin{proof} 

1. If $\B\in \RCA_n\cap \Nr_n\CA_{\omega}=\Nr_n\CA_{\omega}$ has countably many atoms, 
then by \cite[Theorem 5.3.6]{Sayedneat}, $\B$ is completely representable,
so $\Cm\At\B\in \RCA_n$.  Strictly speaking, 
in \cite{Sayedneat} it is shown that the two classes $\CRCA_n$ and $\bold S_c\Nr_n\CA_{\omega}$ 
coincide on countable atomic algebras. One can show that they coincide on the larger class of atomic agebras having countably many atoms by observing that if
$\A$ is an atomic algebra having countably many atoms, then $\Tm\At\A$ is countable 
and $\Tm\At\A\in {\sf CRCA}_n\iff \A\in {\sf CRCA}_n$ because an algebra is completely representable $\iff$ it is atomic and its atom structure is completely representable.
 
The analogous result holds for relation algebras, because
if $\R\in \Ra\CA_{\omega}$ is atomic with 
countably many atoms, then it is completely representable \cite[Theorem 29]{r}. 
The other implication (for both $\RA$s and $\CA$s) follows from corollary \ref{cp}.

2. If $k<\omega$, then the atomic algebras ${\cal R}\in ({\sf RRA}\cap \Ra\CA_l)\sim {\sf CRRA}$ and ${\cal C}\in (\RCA_n\cap \Nr_n\CA_l)\sim {\sf CRCA}_n$ as defined in the proof of 
Theorem \ref{ANT} witness the non-inclusions for $\RA$ and $\CA$, respectively, by taking $l=n+k(<\omega)$. 
For $k\geq \omega$, the required result follows from the construction in \cite{bsl2},
where an atomic  $\R\in \Ra\CA_{\omega}(\subseteq \sf RRA)$ and an atomic $\A\in \Nr_n\CA_{\omega}(\subseteq \RCA_n$) such that $\R$ and $\A$
are not completely representable.  By the fact that $\Ra\CA_{\omega}=\Ra\CA_{\kappa}$ and $\Nr_n\CA_{\omega}=\Nr_n\CA_{\kappa}$ for any infinite ordinal 
$\kappa$, we are through with this case, 
as well. Here by the previous item,  such atomic algebras cannot have only countably many atoms. 

Now we prove the required result for $\CA$s only. 
The proof for $\RA$s is practically the same (modulo the obvious modifications in notation and Theorems used, e.g one uses Theorem \ref{b2} in place of 
Theorem \ref{can}). For item (3)--(6) one uses the algebra $\Tm\bf At$ and its \de\ completion 
$\Cm\bf At$.  For item (7), since $\RCA_n$ is canonical; a classical result of Monk's  \cite{HMT2} (see the paragraph after the next table) and $\A\in \RCA_n$, then $\A^+\in \RCA_n$. But  
$\Cm \At\A\notin \RCA_n$, so it does not embed into $\A^+$, because $\RCA_n$ is a variety, {\it a fortiori} closed under $\bold S$.
\end{proof}
In our next table,  result on atom--canonicity for various varieties of 
$\RA$s and $\CA_n$s are summarized. 
For $\CA$s the dimension $n$ is finite $>2$. 
For the ordinals $k$ and $m$, appearing in the table,  
$k\geq 3$ and $m\geq 6$. 
\vskip3mm
\begin{tabular}{|l|c|c|c|c|c|c|}    \hline
					Algebras	                                &{\sf Atom--canonical}                         & {\sf Citation}  \\

                                                               \hline
                                                                          $\sf RCA_n, \sf RRA$&no& yes,  \cite{Hodkinson}, \cite{HHbook}               \\

                                                               \hline
                                                                          $\bold S{\sf Nr}_n\CA_{n+1}$, $\bold S\Ra\CA_3$,  &yes& yes, \cite{HHbook}\\

                                                               \hline
                                                                          $\bold S{\sf Nr}_n\CA_{n+2}$, $\bold S\Ra\CA_{4}$, $\bold S\Ra\CA_5$&? &\\

                                                               \hline
                                                                          $\bold S{\sf Nr}_n\CA_{n+k}$, $\bold S\Ra\CA_{m}$&no& no, thms \ref{can}, \ref{b2} \\

\hline

\end{tabular}

\vskip3mm
\section{Applications in some  modal fragments of $L_{\omega, \omega}$:}

\subsection{Clique guarded and the packed fragments}

Fix $2<n<\omega$. Here we approach omitting types 
for $L_n$  with respect to {\it clique guarded semantics}, equivalently,
the packed fragment of $L_n$. Like before our approach is algebraic via cylindric algebras.

{\it Clique guarded semantics} for $\CA_n$s can be defined similarly to relation algebras.
We consider  (the locally well--behaved) 
$m$--square \cite[Definition 13.4]{HHbook} and $m$--flat representations of 
$\A\in \CA_n$ with $2<n<m\leq \omega$ \cite[Chapter 13]{HHbook}.
We address algebraically a restricted version of the 
Omitting Types Theorems, namely, Vaught's Theorem on existence of atomic models for atomic theories, in the framework of 
the  clique guarded $n$--variable fragments
of first order logic.  We start by defining {\it clique--guarded semantics.}

\begin{definition}\label{relrep} Assume that $2<n<m<\omega$. Let $\Mo$ be the base of a relativized representation of $\A\in \CA_n$ witnessed by an injective
homomorphism $f:\A\to \wp(V)$, where $V\subseteq {}^n\Mo$ and $\bigcup_{s\in V} \rng(s)=\Mo$. 
We write $\Mo\models a(s)$ for $s\in f(a)$. Let  $\L(\A)^m$ be the first order signature using $m$ variables
and one $n$--ary relation symbol for each element in $A$. 
Let $\L(\A)^m_{\infty, \omega}$ be the infinitary extension of $\L(\A)^m$ allowing infinite conjunctions. 
Then {\it an $n$--clique} is a set $C\subseteq \Mo$ such that
$(a_1,\ldots, a_{n-1})\in V=1^{\Mo}$
for distinct $a_1, \ldots, a_{n}\in C.$
Let
${\sf C}^m(\Mo)=\{s\in {}^m\Mo :\rng(s) \text { is an $n$--clique}\}.$
${\sf C}^m(\Mo)$ is called the {\it $n$--Gaifman hypergraph of $\Mo$}, with the $n$--hyperedge relation $1^{\Mo}$.

The {\it clique guarded semantics $\models_c$} are defined inductively. For atomic formulas and Boolean connectives they are defined
like the classical case and for existential quantifiers
(cylindrifiers) they are defined as follows:
for $\bar{s}\in {}^m\Mo$, $i<m$, $\Mo, \bar{s}\models_c \exists x_i\phi$ $\iff$ there is a $\bar{t}\in {\sf C}^m(\Mo)$, $\bar{t}\equiv_i \bar{s}$ such that
$\Mo, \bar{t}\models \phi$.

\begin{itemize}

\item We say that $\Mo$ is  {\it an $m$--square representation} of $\A$,
if  for all $\bar{s}\in {\sf C}^m(\Mo), a\in \A$, $i<n$,
and   injective map $l:n\to m$, whenever $\Mo\models {\sf c}_ia(s_{l(0)},\ldots, s_{l(n-1)})$, then there is a $\bar{t}\in {\sf C}^m(\Mo)$ with $\bar{t}\equiv _i \bar{s}$,
and $\Mo\models a(t_{l(0)}, \ldots, t_{l(n-1)})$.
We say that $\M$ is {\it a complete $m$--square representation of $\A$ via $f$}, or simply a complete representation of $\A$ if 
$f(\sum X)=\bigcup_{x\in X} f(x)$, for all 
$X\subseteq \A$ for which $\sum X$ exists.

\item We say that $\Mo$ is an {\it (infinitary) $m$--flat representation} of $\A$ if  it is $m$--square and
for all $\phi\in (\L(\A)_{\infty, \omega}^m) \L(\A)^m$, 
for all $\bar{s}\in {\sf C}^m(\Mo)$, for all distinct $i,j<m$,
$\Mo\models_c [\exists x_i\exists x_j\phi\longleftrightarrow \exists x_j\exists x_i\phi](\bar{s})$.
Complete representability is defined like for squareness.
\end{itemize}
\end{definition}
It is straightforward to show, like in the classical case, 
that $\A$ has a complete $m$--square complete representation $\Mo$ via $f$ 
$\iff$ $\A$ is atomic and $f$ is {\it atomic} in the sense 
that $\bigcup_{x\in \At\A} f(x)=1^{\Mo}$.
One can define $m$--smooth representations of an algebra $\A\in \CA_n$ as in \cite[Definition 13.12] {HHbook}.

The main ideas used in the next Theorem needed to perform a transfer from results relating neat embedding properties to relativized representations from 
$\sf RA$s to $\CA$s can be found in 
\cite[Definitions 12.1, 12.9, 12.10, 12.25, Propositions 12.25, 12.27]{HHbook}.
In all cases, the $m$--dimensional dilation stipulated in the statement of the Theorem, will have
top element ${\sf C}^m(\Mo)$, where $\Mo$ is 
the $m$--relativized representation of the given algebra, and the (completely additive) operations of the dilation
are induced by the $n$-clique--guarded semantics. For $2<n<m<\omega$ and an atomic $\A\in \CA_n$, an $m$--dimensional (hyper)basis for $\A$  
can be defined similarly to the $\RA$ case.
We formulate (explicitly) the definition of a basis:

\begin{definition}\label{basis} Let $2<n<m<\omega$ and $\A\in \CA_n$ be atomic.

An {\it $m$--dimensional basis $B$} for $\A$ consists of a set of $n$--dimensional networks whose nodes $\subseteq m$, satisfying 
the following properties: 

(a) For all  $a\in \At\A$, there is an $N\in B$ such that $N(0,1,\ldots, n-1)=a,$

(b) The {\it cylindrifier property}: For all $N\in B$, all $i<n$,  all $\bar{x}\in {}^n\nodes(N)(\subseteq {}^nm)$, all $a\in\At\A$, such that
$N(\bar{x})\leq {\sf c}_ia$,  there exists $M\in B$, $M\equiv_i N$, $\bar{y}\in {}^n\nodes(M)$ such 
that $\bar{y}\equiv_i\bar{x}$ and $M(\bar{y})=a.$  We can always assume that $\bar{y}_i$ is a new node else one takes $M=N$. 
\end{definition}
Next  we recall the definition of certain non--commutative set algebras. 

\begin{definition} Let $n<\omega$. The class ${\sf D}_n({\sf G}_n)$ is a class of set algebras having the same signature as $\CA_n$. If $\A\in {\sf D}_n({\sf G}_n)$, then  the 
top element of $\A$ is a set  $V\subseteq {}^n U$ (some non--empty set $U$),
such that if $s\in V$,  and $i<j<n$ ($\tau: n\to n$), then $s\circ [i|j] (s\circ \tau)\in V$.
\end{definition}
Let $\sf K$ be a class of $\sf BAO$s.
We write ${\sf K}\cap \bf At$ for  the class of atomic algebras in $\sf K$. Recall that $\bold S_c$ denotes the operation of foirming complete 
subalgebras.
\begin{theorem}\label{flat}\cite[Theorems 13.45, 13.36]{HHbook}.
Assume that $2<n<m<\omega$ and let $\A\in \CA_n$. Then the following hold:
\begin{enumerate}
 \item $\A\in \bold S\Nr_n\CA_m\iff \A$ has an $m$--smooth representation $\iff \A$ has an  infinitary $m$--flat representation
$\iff \A$ has an $m$--flat representation $\iff \A\in \bold S{\Nr}_n(\CA_m\cap {\sf G}_m)\iff \A\in \bold S\Nr_n(\CA_m\cap {\sf D}_m)\iff \A^+$ has an $m$-dimensional hyperbasis,

\item $\A\in \bold S\Nr_n{\sf D}_m\iff \A\in \bold S\Nr_n{\sf G}_m \iff \A$ has an $m$--square  representation $\iff$ $\A$ has an $m$--dimensional basis, 

\item If $\A$ is atomic, then: $\A$ has a complete infinitary $m$--flat representation $\iff$ $\A$ has a complete $m$--smooth representation
$\iff$ $\A\in \bold S_c\Nr_n(\CA_m\cap \bf At)\iff \A$ has an $m$--dimensional hyperbasis. 
\end{enumerate}
\end{theorem}
\begin{proof} We give a sample of how to prove some (but not all) the statements in the Theorem.
Some of the implications are highly technical, but fairly straightforward to lift from the $\RA$ case.
We start by proving that the existence of $m$--flat representations, implies the existence of $m$--dilations.
Let $\Mo$ be an $m$--flat representation of $\A$. We show that $\A\subseteq \Nr_n\D$, for some $\D\in \CA_m$,
For $\phi\in \L(\A)^m$ (as defined above),
let $\phi^{\Mo}=\{\bar{a}\in {\sf C}^m(\Mo):\Mo\models_c \phi(\bar{a})\}$, where ${\sf C}^m(\Mo)$ is the $n$--Gaifman hypergraph.
Let $\D$ be the algebra with universe $\{\phi^{M}: \phi\in \L(\A)^m\}$ and with  cylindric
operations induced by the $n$-clique--guarded (flat) semantics. 

For $r\in \A$, and $\bar{x}\in {\sf C}^m(\Mo)$, we identify $r$ with the formula it defines in $\L(\A)^m$, and 
we write $r(\bar{x})^{\Mo}\iff \Mo, \bar{x}\models_c r$.
Then certainly $\D$ is a subalgebra of the ${\sf Crs}_m$ (the class
of algebras whose units are arbitrary sets of $m$--ary sequences)
with domain $\wp({\sf C}^m(\M))$, so $\D\in {\sf Crs}_m$ with unit $1^{\D}={\sf C}^m(\Mo)$.
Since $\Mo$ is $m$--flat, then cylindrifiers in $\D$ commute, and so $\D\in \CA_m$.
Now define $\theta:\A\to \D$, via $r\mapsto r(\bar{x})^{\Mo}$. Then exactly like in the proof of \cite[Theorem 13.20]{HHbook},
$\theta$ is a neat embedding, that is, $\theta(\A)\subseteq \Nr_n\D$.
It is straightforward to check that $\theta$ is a homomorphism.  We show that $\theta$ is injective.
Let $r\in A$ be non--zero. Then $\Mo$ is a relativized representation, so there is $\bar{a}\in M$
with $r(\bar{a})$, hence $\bar{a}$ is a clique in $\Mo$,
and so $\Mo\models r(\bar{x})(\bar{a})$, and $\bar{a}\in \theta(r)$, proving the required. 

The relativized model $\Mo$ itself might not be  infinitary $m$--flat, but one can build an infinitary $m$--flat representation of $\A$, whose base is an $\omega$--saturated model
of the consistent first order theory, stipulating the existence of an $m$--flat representation, cf. \cite[Proposition 13.17, Theorem 13.46 items (6) and (7)]{HHbook}.

The inverse implication from dilations to representations harder. One constructs from the given 
$m$--dilation, an $m$--dimensional 
hyperbasis from which
the required $m$-- relativized representation is built.  
This can be done in a step--by step manner treating the hyperbasis 
as a `saturated set of mosaics', cf. \cite[Proposition 13.37]{HHbook}.

For results on {\it complete} $m$--flat representations, one works in $L_{\infty, \omega}^m$ instead of first order logic.
Assume that $\A$ has a complete $m$--flat representation $\Mo$. This time we construct an {\it atomic} $m$-dilation $\D$ of $\A$. 
The $m$-- dilation $\D$ will have again top element the Gaifman hypergraph ${\sf C}^m(\Mo)$.
The universe will be bigger; $D=\{\phi^M: \phi\in \L(A)^m_{\infty, \omega}\}$ with operations also induced by the $n$-clique--guarded semantics extended to
$L_{\infty, \omega}^m$. The $m$--dilation $\D$, formed this way,  will be a $\CA_m$ as before, but this time, it will be {\it an atomic} one.  
To prove atomicity, let $\phi^{\Mo}$ be a non--zero element in $\D$.
Choose $\bar{a}\in \phi^{\Mo}$, and consider the following infinitary conjunction (which we did not have before when working in $L_m$)
\footnote{There are set--theoretic subtleties involved here, that we prefer to ignore.}:
$\tau=\bigwedge \{\psi\in \L(\A)_{\infty,\omega}^m: \Mo\models_c \psi(\bar{a})\}.$
Then $\tau\in \L(\A)_{\infty,\omega}^m$, and $\tau^{\Mo}$ is an atom below $\phi^{\Mo}$. 
The neat embedding, defined like before, will be an atomic one, hence
it will be a complete neat embedding \cite[p. 411]{HHbook}.  
\end{proof}

A direct  consequence of the Omitting Types Theorem ($\sf OTT$) 
for $L_{\omega, \omega}$ is that if $T$ is a countable theory and $\Gamma$ is a set of formulas 
using only finitely many variables (a type)
such that $\Gamma$ is realizable in every model of $T$, 
then  $\Gamma$ is isolated, that is to say, there is a formula $\psi$ consistent with $T$, such that $T\models \psi\to \phi$ for all $\phi$ in $\Gamma$.
As mentioned in introduction $\sf OTT$ fails for $L_n$ for $2<n<\omega$ dramatically
in the following  (strong) sense \cite{ANT}: 
For every $2<n\leq l<\omega$, there is a countable and complete $L_n$ theory $T$, and a 
type that is realizable in every (ordinary) model of $T$, but cannot be isolated by a formula 
using $l$ variables (called an $l$- witness). In fact, $T$ can be chosen to be a complete and atomic theory, meaning that the Tarski--Lindenbaum algebra $\Fm_T$ is  simple (as an $\RCA_n$)
and (its Boolean reduct) is an atomic algebra. Furthermore, the non--isolated type is the non--principal type of co--atoms.
In the last result {\it only} ordinary, that is to say, $\omega$--square models are considered. In other words, even the weaker Vaught's theorem $\sf VT$ fails. 

Suppose  that we consider instead $m$--square models for some finite $m$, 
in this way substantially broadening the class of models in which types are realized.
Is it then reasonable to expect that perhaps for some value of (finite) $m$, we regain $\sf OTT$, or if we are less ambitious, we regain $\sf VT$ (upon considering atomic theories and the single 
type of co-atoms) for $L_n$ with respect to the aforementioned generalized semantics? In other words, in the case we can find an $l$--witness, for some $l\in \omega$. 
In what follows we ponder on this possibilty. 

Fix $2<n\leq l<m\leq \omega$. The restriction $l<m$ is dictated by item (3) of Theorem \ref{ANT}. Consider the statement $\Psi(l, m)$: 

{\it There is an atomic, countable and complete $L_n$ theory $T$, such that the type $\Gamma$ consisting of co--atoms is realizable 
in every $m$-- square model, but any formula isolating this type has
to contain more than $l$ variables.}  

By an $m$--square model $\Mo$ of $T$ we understand an $m$--square representation of the algebra $\Fm_T$ with base $\Mo$.

Let ${\sf VT}(l, m))=\neg \Psi(l, m)$, short for {\bf Vaught's Theorem holds `at the parameters $l$ and $m$'} where
by definition, we stipulate that ${\sf VT}(\omega, \omega)$ is just Vaught's Theorem for $L_{\omega, \omega}$: Countable 
atomic theories  have countable atomic models.

For $2<n\leq l<m\leq \omega$  and $l=m=\omega$,  it is likely and plausible that (**) $\VT(l, m)\iff l=m=\omega$.  
In other words:  {\it Vaught's Theorem holds only in the limiting  case when $l\to \infty$ and $m=\omega$ and not `before'. 
This will be  proved on  the `paths' $(l, \omega)$, $n\leq l<\omega$ ($x$ axis)
and $(n, n+k)$, $k\geq n+3$ ($y$ axis)  
using the two different blow up and blur constructions, 
given in Theorems \ref{ANT}, \ref{can}, respectively, cf. Corollary \ref{fl}.} 

In the next Theorem several conditions are given implying $\Psi(l, m)$ for various values of $l$ and $m$.
$\Psi(l, m)_f$ is the formula obtained from $\Psi(l, m)$ be replacing square by flat.
In the first item by no infinite 
$\omega$--dimensional hyperbasis (basis), we understand no representation on an 
infinite base.  By $\omega$--flat (square) representation, we mean an ordinary representation, and by complete $\omega$--flat (square) representation, we mean 
a complete representation.
For an atomic  relation algebra $\R$ and $l>3$, recall that 
${\sf Mat}_n(\At\R)$ denotes the set of all $n$--dimensional basic matrices on $\R$.

\begin{theorem} \label{main} Let $2<n\leq l<m\leq \omega$. Then
every item implies the immediately following one. 
\begin{enumerate}
\item There exists a finite relation algebra $\R$ algebra with a strong $l$--blur and no infinite $m$--dimensional hyperbasis, 
\item There is a countable atomic $\A\in \Nr_n\CA_l\cap \RCA_n$ such that $\Cm\At\A$ does not have an 
$m$--flat representation,
\item There is a countable atomic $\A\in \Nr_n\CA_l\cap \RCA_n$ such that $\Cm\At\A\notin \bold S\Nr_n\CA_m$, 
\item There is a countable atomic $\A\in \Nr_n\CA_l\cap \RCA_n$ such that $\A$ has no complete $m$--flat representation,
\item There is a countable atomic $\A\in \Nr_n\CA_l\cap \RCA_n$ such that $\A\notin \bold S_c\Nr_n\CA_m$,
\item $\Psi(l, m)_f$ is true,
\item $\Psi(l', m')_f$ is true for any $l'\leq l$ and $m'\geq m$.
\end{enumerate}
The same implications hold upon replacing infinite $m$--dimensional hyperbasis by $m$--dimensional relational basis (not necessarily infinite), 
$m$--flat by $m$--square and $\bold S\Nr_n\CA_m$ by $\bold S\Nr_n{\sf D}_m$. Furthermore, in the new chain of implications every item implies the corresponding item in 
Theorem \ref{main}. In particular, 
$\Psi(l, m)\implies \Psi(l, m)_f$.  
\end{theorem}

\begin{proof}

We proceed like $(1)\implies (2)$ in the proof of Theorem \ref{ANT} giving more details. Let $\R$ be as in the hypothesis with strong $l$--blur $(J, E)$. 
The idea is to `blow up and blur' $\R$ in place of the Maddux algebra 
$\mathfrak{E}_k(2, 3)$ dealt with in \cite[Lemma 5.1]{ANT}, where $k<\omega$ is the number of non--identity atoms 
and  $l$ depends recursively on $k$. 

Let $2<n\leq l<\omega$.  The relation algebra $\R$ is blown up by splitting all of the atoms each to infinitely many.
$\R$ is blurred by using a finite set of blurs (or colours) $J$. 
Then two partitions are defined on $\bf At$, call them $P_1$ and $P_2$.
Composition is re-defined on this new infinite atom structure; it is induced by the composition in $\R$, and a ternary relation $E$
on $\omega$, that `synchronizes' which three rectangles sitting on the $i,j,k$ $E$--related rows compose like the original algebra $\R$.
This relation is definable in the first order structure $(\omega, <)$ \cite{ANT}.

The first partition $P_1$ is used to show that $\R$ embeds in the complex algebra of this new atom structure, namely, $\Cm \bf At$. 
The second partition $P_2$ divides $\bf At$ into {\it finitely many (infinite) rectangles}, each with base $W\in J$,
and the term algebra denoted in \cite{ANT} by  ${\sf split}(\R, J, E)$,  and here by ${\sf split}(\R, J, E)$ over $\bf At$ (where $(J, E)$ is the strong $l$--blur for $\R$ 
assumed to exist by hypothesis)  consists of the sets that intersect co--finitely with every member of this partition.
One proves that ${\sf split}(\R, J, E)$ with atom structure $\bf At$ is representable using the finite number of blurs in $J$. 
Because $(J, E)$ is a strong $l$--blur, then, by definition, it is a strong $j$--blur for all $n\leq j\leq l$, so the atom structure $\bf At$ has a $j$--dimensional cylindric basis for all $n\leq j\leq l$, 
namely, ${\sf Mat}_j(\bf At)$.
For all such $j$, there is an $\RCA_j$ denoted on \cite[Top of p. 78]{ANT} by ${\Bb}_j(\R, J, E)$, which we denote here in conformity with the notation in Theorem \ref{ANT}
by ${\sf split}_j(\R, J, E)$, such 
that $\Tm{\sf Mat}_j({\bf At})\subseteq {\sf split}_j(\R, J, E)\subseteq \Cm{\sf Mat}_j(\bf At)$ 
and $\At{\sf split}_j(\R, J, E)$ is a weakly representable atom structure of dimension $j$.

Now take $\A={\sf split} _n(\R, J, E)$. We claim that $\A$ is as required.  Since $\R$ has a strong $j$--blur $(J, E)$ for all $n\leq j\leq l$, then 
$\A\cong \mathfrak{Nr}_n{\sf split}_j(\R, J, E)$ for all $n\leq j\leq l$ as proved in \cite[item (3) p. 80]{ANT}. In particular, 
taking $j=l$, $\A\in \RCA_n\cap \Nr_n\CA_l$. We show that $\Cm\At\A$  
does not have an $m$--flat representation.  
Assume for contradicton that $\Cm\At\A$ does have an $m$--flat representation $\Mo$.
Then $\Mo$ is infinite of course.  Since $\R$ embeds into ${\sf spli}(\R, J, E)$ which in turn embeds into $\mathfrak{Ra}\Cm\At\A$, then 
$\R$ has an $m$--flat representation with base $\Mo$. But since $\R$ is finite, $\R=\R^+$, and consequently $\R$ has an  
infinite $m$--dimensional hyperbasis.  This is contrary to our assumption and we are done.

$(2)\implies (3)$: By item (1) of Theorem \ref{flat}.

$(3)\implies (4)$: A complete $m$--flat representation of (any) 
$\B\in \CA_n$ induces an $m$--flat representation of $\Cm\At\B$ which implies by Theorem \ref{flat} that $\Cm\At\B\in \bold S\Nr_n\CA_m$. 
To see why, assume that $\B$ has an $m$--flat  complete representable via $f:\B\to \D$, where $\D=\wp(V)$ and 
the base of the representation $\Mo=\bigcup_{s\in V} \rng(s)$ is $m$--flat. Let $\C=\Cm\At\B$.
For $c\in C$, let $c\downarrow=\{a\in \At\C: a\leq c\}=\{a\in \At\B: a\leq c\}$; the last equality holds because 
$\At\B=\At\C$. Define, representing $\C$,  
$g:\C\to \D$ by $g(c)=\sum_{x\in c\downarrow}f(x).$ The map $g$ is well defined because $\C$ is complete so arbitrary suprema exist in $\C$. 
Furthermore, it can be easily checked that $g$ is a homomorphism into $\wp(V)$ having base $\Mo$ (basically because by assumption $f$ is a homomorphism).

$(4)\implies (5)$: By item (2) of Theorem \ref{flat}.

$(5)\implies (6)$: By \cite[\S 4.3]{HMT2}, we can (and will) assume that $\A= \Fm_T$ for a countable, atomic theory $L_n$ theory $T$.  
Let $\Gamma$ be the $n$--type consisting of co--atoms of $T$. Then $\Gamma$ is realizable in every $m$--flat model, for if $\Mo$ is an $m$--flat model omitting 
$\Gamma$, then $\Mo$ would be the base of a complete $m$--flat  representation of $\A$, and so $\A\in \bold S_c\Nr_n\CA_m$ which is impossible.
But $\A\in {\sf Nr}_n\CA_l$, so  using exactly the same  (terminology and) argument in \cite[Theorem 3.1]{ANT} we get that  
any witness isolating $\Gamma$  needs more 
than $l$--variables (see also the proof of item (2) of corollary \ref{OTT2}.) 

$(6)\implies (7)$: follows from the definitions.

For squareness the proofs are essentially the same undergoing the obvious modifications. In the first implication `infinite' in the hypothesis is not needed because any finite 
relation algebra having an infinite $m$--dimensional relational basis has a finite one, cf. \cite[Theorem 19.18]{HHbook} and item (1) of Lemma 
\ref{i}.
\end{proof}
\begin{theorem}\label{can1}
Assume that $2<n<\omega$. Then  $\Psi(n, n+3)_f$ is true. Thus, for any $m\geq n+3$,  $\Psi(n, m)_f$ is true.
\end{theorem}
\begin{proof}
By Theorem \ref{main} 
it suffices to show that there is an atomic algebra with countably many atoms 
outside $\bold S_c\Nr_n\CA_{n+3}$.

Fix finite $n>2$. Let $\N^{-1}=\{-n: n\in \N\}$. We look at $\N^{-1}$ as an ordered structure with usual order: 
$-m<-n\iff m>n$ ($m, n\in \N$). Let $\A$ be the $n$--dimensional
rainbow cylindric algebra $R(\Gamma)$  \cite[Definition 3.6.9]{HHbook2}
where $\Gamma=\sf R=\omega$, so that the reds is the set $\{\r_{ij}: i<j<\omega\}$ and the greens
constitute the set $\{\g_i:1\leq i <n-1\}\cup \{\g_0^i: i\in \N^{-1}\}$, so that $\sf G=\N^{-1}$. 
In complete coloured graphs the forbidden triples are like 
in usual rainbow constructions \cite{HH}  
but now we impose a new
forbidden triple in coloured graphs connecting two greens and one red. We stipulate that 
the triple  $(\g^i_0, \g^j_0, \r_{kl})$ is forbidden if $\{(i, k), (j, l)\}$ is not an order preserving partial function from
$\N^{-1}\to\N$. Here we identify $\omega$ with $\N$. 
The $n$--dimensional complex algebra of this atom structure, which we denote by $\C_{\N^{-1}, \N}$ is based 
on the two ordered structure $\N^{-1}$ (greens) and $\N$ (reds).

We show that \pa\ has a \ws\ in $\bold G^{n+3}(\At\C_{\N^{-1}, \N})$, implying  
by Lemma \ref{n}, that $\C_{\N^{-1}, \N}\notin \bold S_c{\sf Nr}_n\CA_{n+3}$.
The idea here is that the newly added triple
forces \pe\ to play reds $\r_{ij}$ with one of the indices forming a decreasing sequence in $\N$
in response to \pa\ playing cones having a common 
base and distinct green tints (demanding a red label for edges between 
appexes of two succesive cones.) Having the option to reuse 
the $n+3$ nodes is crucial for \pa\ to implement his \ws\ because he uses {\it finitely many} nodes to win an {\it infinite} $\omega$--rounded game.

In the initial round \pa\ plays a graph $M$ with nodes $0,1,\ldots, n-1$ such that $M(i,j)=\w_0$
for $i<j<n-1$
and $M(i, n-1)=\g_i$
$(i=1, \ldots, n-2)$, $M(0, n-1)=\g_0^0$ and $M(0,1,\ldots, n-2)=\y_{\N^{-1}}$. This is a $0$ cone.
In the following move \pa\ chooses the base  of the cone $(0,\ldots, n-2)$ and demands a node $n$
with $M_2(i,n)=\g_i$ $(i=1,\ldots, n-2)$, and $M_2(0,n)=\g_0^{-1}.$
\pe\ must choose a label for the edge $(n+1,n)$ of $M_2$. It must be a red atom $r_{mk}$, $m, k\in \N$. Since $-1<0$, then by the `order preserving' condition 
we have $m<k$.
In the next move \pa\ plays the face $(0, \ldots, n-2)$ and demands a node $n+1$, with $M_3(i,n)=\g_i$ $(i=1,\ldots, n-2)$,
such that  $M_3(0, n+2)=\g_0^{-2}$.
Then $M_3(n+1,n)$ and $M_3(n+1, n-1)$ both being red, the indices must match, so $M_3(n+1,n)=r_{lk}$ and $M_3(n+1, n-1)=r_{km}$ with $l<m\in \N$.
In the next round \pa\ plays $(0,1,\ldots n-2)$ and re-uses the node $2$ such that $M_4(0,2)=\g_0^{-3}$.
This time we have $M_4(n,n-1)=\r_{jl}$ for some $j<l<m\in \N$.
Continuing in this manner leads to a decreasing 
sequence in $\N$. Now that \pa\ has a \ws\ in $\bold G^{n+3}$, 
by Lemma \ref{n}, $\C_{\N^{-1}, \N}\notin \bold S_c{\sf Nr}_n\CA_{n+3}$.
\end{proof}
But we can go further addressing squareness, getting closer to (**) formulated above:

\begin{corollary}\label{fl} For $2<n<\omega$ and $n\leq l<\omega$, $\Psi(n, n+3)$ and $\Psi(l, \omega)$ hold.
\end{corollary}
\begin{proof} From Theorems \ref{can}, \ref{main} (by taking $l=n$ and $m=n+3$), and Theorem \ref{ANT}.

In the first case, using the notation in Theorem \ref{can}, we have $\B=\Rd_{ca}\A$, where $\A={\sf split}(\A_{n, n+1}, \r, \omega)$, 
satisfies that $\B\in {\sf Cs}_n$ is simple, 
countable and atomic, and $\Cm\At\B\notin \bold S\Nr_n{\sf D}_{n+3}$. 

For the second case, 
it suffices by Theorem \ref{main} (by taking $m=\omega$) to find a countable algebra $\C\in \Nr_n\CA_l\cap \RCA_n$
such that $\Cm\At\C\notin \RCA_n$. Let $l\geq 2n-1$, $k\geq (2n-1)l$, $l, k\in \omega$. 
One takes the finite
integral relation algebra $\R_l={\mathfrak E}_k(2, 3)$ as defined in \cite[Lemma 5.1]{ANT} 
where $k$ is the number of non-identity atoms in
$\R_l$. Then $\R_l$ has a strong $l$--blur, $(J, E)$ say, and it can only be represented 
on a finite base \cite{ANT}.  The rest follows using the same reasoning in the third item of 
theorem \ref{main}. In this case, using the notation in {\it op.cit}, ${\sf split} _n(\R_l, J, E)=\mathfrak{Nr}_n{\sf split}_l(\R_l, J, E)$, 
so ${\sf split}_n(\R, J, E)\in \Nr_n\CA_l\cap \RCA_n$ and $\Cm\At({\sf split}_n(\R, J, E))\notin \RCA_n$. 
 \end{proof}

It is timely that we tie a number of threads together.
Summarizing our main results on $\sf VT$, we need a definition:
\begin{definition} \label{everywhere}Let $2<n<\omega$.
\begin{enumerate}
\item We say that ${\sf VT}$ {\it fails for $L_n$ almost everywhere} if there exist  positive $l,m\geq n$ such that ${\sf V}(k, \omega)$ and
${\sf V}(n, t)$ are false for all finite $k\geq l$ and all $t\geq m$. 
\item We say that ${\sf VT}$ {\it fails  for $L_n$ everywhere} if for $3\leq l<m\leq \omega$ and $l=m=\omega$, ${\sf V}(l, m)$ holds $\iff$ $l=m=\omega$.
\end{enumerate}
\end{definition}
\begin{theorem}\label{OTT2}
\begin{enumerate}
\item $\sf VT$ fails for $L_n$ almost everywhere. 
\item If for each $n<m<\omega$, there exists a finite relation algebra $\R_m$ having $m-1$ strong blur and no
$m$--dimensional relational basis (equivalently $m$--square representation), then
${\sf VT}$ fails for $L_n$ everywhere, that is to say $(**)$ formulated above holds.
\end{enumerate}
\end{theorem}

\begin{proof}

Item (1) is already proved in Corollary \ref{fl}.

Now we prove the second item.  
The proof is similar to the proof of the implication $(5)\implies (6)$ in Theorem \ref{main} replacing flatness by squareness.
Assume for contradiction that there is $n\leq l<k$ such that ${\sf VT}(l, k)$ holds. 
Choose $m$ such that $l\leq m-1<k$. 
Since $\R_m$ has an 
$m-1$ blur $(J, E)$, say, then ${\sf split}(\R_m, J, E)$ has an $m-1$-dimensional cylindric basis, 
but  $\Cm\At{\sf split}(\R_m, J, E)$
does not have an 
$m$--dimensional relational basis, seeing as how $\R_m$ does not have an infinite relational basis, and $\R_m$ embeds into
$\Cm\At{\sf split}(\R_m, J, E).$ We know that $\R_m$ does not have an infinite relational basis because any relation algebra having an infinite relational basis 
has a finite one \cite{HHbook}. 

Let 
$\A={\sf split}_n(\R_m, J, E)\in \RCA_n$. Then $\A$ is simple and $\A\cong {\mathfrak Nr}_n{\sf split}_{m-1}(\R_m, J, E)$.
Take the atomic, complete and countable (consistent) theory  
$T$ such that $\Fm_T\cong \A$. 
Then in $T$ the type of co-atoms $\Gamma$ will be realizable in every $m$--square model, for if $\Mo$ is an $m$--square model omitting
$\Gamma$, then $\Mo$ would be the base of a complete $m$--square  representation of $\A$ inducing an $m$--square representation of $\Cm\At\A$.
But this impossible,
because $\R_m$ embeds into $\Ra\Cm\At\A$, so an $m$-square representation of $\Cm\At\A$ 
necessarily induces an $m$-square representation of $\R_m$, which means that $\R_m$ has an $m$-dimensional relational basis, which is contrary to assumption.

We show that $\Gamma$ has no $m-1$ witness. Suppose for contradiction that $\phi$ is an $m-1$ witness, so that $T\models \phi\to \alpha$, for
all $\alpha\in \Gamma$, where recall that $\Gamma$ is the set of coatoms.
Then since $\A$ is simple, we can assume
without loss of generality, that $\A$ is a set algebra with 
base $M$ say.
Let $\Mo=(M,R_i)_{i\in \omega}$  be the corresponding model (in a relational signature)
to this set algebra in the sense of \cite[\S 4.3]{HMT2}. Let $\phi^{\Mo}$ denote the set of all assignments satisfying $\phi$ in $\Mo$.
We have  $\Mo\models T$ and $\phi^{\Mo}\in \A$, because $\A\in \Nr_n\CA_{m-1}$.
But $T\models \exists x\phi$, hence $\phi^{\Mo}\neq 0,$
from which it follows that  $\phi^{\Mo}$ must intersect an atom $\alpha\in \A$ (recall that the latter is atomic).

Let $\psi$ be the formula, such that $\psi^{\Mo}=\alpha$. Then it cannot
be the case
that $T\models \phi\to \neg \psi$,
hence $\phi$ is not a  witness,
contradiction and we are done.
We have shown that ${\sf VT}(m-1, m)$ is false. But $l\leq m-1<k$ hence by (the `square version' of) the last implication 
of Theorem \ref{main}, ${\sf VT}(l, k)$ is false.
\end{proof}

Fix $2<n<\omega$. Then 
Theorem \ref{can} says that Vaught's Theorem fails for the {\it packed fragment of $L_n$}:

\begin{theorem}\label{packed} Assume that $2<n<m<\omega$. Let $\A\in \CA_n$ and 
$\Mo$ be an $m$--flat representation of $\A$.
Then  
$$\Mo, s\models_c \phi  \iff\  \Mo,s\models {\sf packed}(\phi),$$
for all $s\in {\sf C}^n(\Mo)$ and every $\phi\in \L(\A)^n$, where $\sf packed(\phi$) 
denotes the translation of $\phi$ to the packed fragment \cite[Definition 19.3]{HHbook}.
\end{theorem}

In the sense of the previous Theorem, the {\it the clique guarded fragments}, which are the $n$--variable fragments of first order order with clique 
(locally) guarded semantics are an alternative formulation of {\it the $n$--variable 
packed fragments}  of first order logic \cite[\S 19.2.3]{HHbook}.\\

{\bf Summary of results on $\sf VT$:} In the coming table, for any finite $j$, `$j$-hyp' is short hand for infinite $j$--dimensional hyperbasis, 
and $j$--basis is short hand for $j$--dimensional relational basis. ${\sf VT}(l, m)$ 
for $n\leq l<m\leq \omega$ and $\VT(\omega, \omega)$ are defined as before.   
In particular, in the table $l<\omega$.
\vskip3mm
\begin{tabular}{|l|c|c|c|c|c|c|}    \hline
					
                                                                          $\VT(n, \omega)$&no, \cite{ANT} and Theorem \ref{can}\\

                                                               \hline
                                                                          $\VT(n, n+3)$& no,  Theorem \ref{can} \\

                                                               \hline
                                                                          $\VT(n, n+2)_f$&no, if there is $\R$ with $n$--blur and no  $n+2$-hyp, Theorem \ref{main}\\

                                                               \hline
                                                                          $\VT(l, \omega)$&no, $\mathfrak{E}_{f(l)}(2,3)$ has strong $l$-blur, and no $\omega$-hyp, Theorem \ref{main}, \cite{ANT}\\

                                                                           \hline
                                                                          $\VT(l, m)_f, l\leq m-1$ &no, if there exists $\R$ with strong $l$-blur, and no $m$-hyp, Theorem \ref{main} \\

                                                                           \hline
                                                                          $\VT(l, m), l\leq m-1$ &no, if there exists $\R$ with strong $l$-blur, and no $m$-bases, Theorems \ref{main}, \ref{OTT2}\\

                                                                           \hline
                                                                          $\VT(\omega, \omega)$&yes, $\sf VT$ for $L_{\omega, \omega}$.\\

                                                                          \hline

\end{tabular}

\subsection{Squareness  versus flatness in terms of decidability}

Let $2<n<m<\omega$. If an algebra $\A$ has an 
$m$--square representation,
then the algebra neatly embeds into another $m$--dimensional algebra $\B$, but $\B$ is not necessarily a $\CA_m$ for  it may fail commutativity of cylindrifiers. 
This discrepancy in the formed dilations in case of $m$--flatness and $m$--squareness blatantly manifests itself in
a very important property.
The precarious condition of commutativity of cylindrifiers in the formed dilation in case of $m$--flatness when $m\geq n+3$,
makes this clique guarded fragment {\it strongly undecidable.}

It is decidable to tell for $2<n<m<\omega$,  whether an $n$--dimensional finite algebra has an $m$--square representation in polynomial time, cf. \cite[Corollary, 12.32]{HHbook}, 
which is not true for $n$--flatness when $n=3$ and $m\geq n+3$.
Let us formulate the latter result and some related ones for three dimensions.  

\begin{theorem} \label{decidability} Assume that  $m\geq 6$. Then it is undecidable to tell
whether a finite algebra in  $\CA_3$ has an $m$--flat representation, and the variety
$\bold S{\sf Nr}_3\CA_m$ cannot be finitely axiomatizable
in $k$th order logic for any positive $k$.
\end{theorem}\begin{proof}  This can be proved by lifting the analogous results for relation algebras
\cite[Theorem 18.13, Corollaries 18.14, 18.15, 18.16]{HHbook}. One uses the construction of Hodkinson in \cite{AUU}
which associates recursively to every atomic relation algebra $\sf R$, an atomic  $\A\in \CA_3$
such that ${\sf R}\subseteq \sf Ra\A$, 
the latter is the relation algebra reduct of $\A$, cf. \cite[Definition 5.3.7, Theorem 5.3.8]{HMT2}.

The idea for the second part on non--finite axiomatizability is that the existence of any such finite axiomatization in 
$k$th order logic for any positive $k$,  
gives a decision procedure for telling whether a finite algebra is 
in $\bold S{\sf Nr}_3\CA_m$
or not \cite{HHbook} which is impossible as just shown.
\end{proof}

\begin{theorem}\label{decidability2} Let $m\geq 6$. Then there are finite $\CA_3$s that have infinite $m$--flat representations, but do not
have finite ones, equivalently 
they {\it do not have a finite $m$--dimensional hyperbasis.}
\end{theorem}
\begin{proof}
To see why, assume for contradiction that every finite algebra in $\bold S{\sf Nr}_3\CA_m$ has a finite
$m$--dimensional hyperbasis. 
We claim that there is an algorithm that decides membership in $\bold S{\sf Nr}_3\CA_m$ for finite algebras which we know 
is impossible:
\begin{itemize}
\item Using a recursive axiomatization of $\bold S\sf Nr_3\CA_m$ (exists), recursively enumerate all isomorphism types of
finite $\CA_3$s that are not in $\bold S{\sf Nr}_3\CA_m.$

\item Recursively enumerate all finite algebras in $\bold S{\sf Nr}_3\CA_m$.
For each such algebra, enumerate all finite sets of $m$--dimensional hypernetworks over $\A$,
using $\N$ as hyperlabels, and check  to
see if it is a hyperbasis. When a hypebasis is located specify $\A$.
This recursively enumerates  all and only the finite algebras in $\bold S{\sf Nr}_3\CA_m$.
Since any finite $\CA_3$ is in exactly one of these enumerations, the process will decide
whether or not it is in  ${\bold S}{\sf Nr}_3\CA_m$ in a finite time.
\end{itemize}
We have shown that there are finite algebras that have infinite $m$--flat representations, but do not
have finite ones (this cannot happen with $m$--squareness).
\end{proof}
\begin{theorem}
The set of isomorphism types of algebras in $\bold S\Nr_3\CA_n$ with
only infinite flat representations is not recursively enumerable.
\end{theorem}
\begin{proof} let $C$ be the given set. The set $A$ of isomorphism types of finite
algebras is recursively enumerable, and so the set $B$ of isomorphism types of finite algebras
that are in $\bold S\Nr_3\CA_n$ having a finite relativized representation is also
recursively enumerable,  hence $A$ and $B$ together with the set
of isomorphism types of of algebras in $\bold S\Nr_3\CA_n$ with no finite recursive representation
are recursive $\iff$ $C$ is recursively enumerable.
\end{proof}
\subsection{Vaught's theorem for first order definable expansions of $L_n$}

Let ${\sf VT}(\omega, \omega)$ stand for 
the true statement:  `Vaught theorem holds for $L_{\omega, \omega}$.'
By ${\sf VT}(l, m)$ for $2<n\leq l\leq \omega$ and $n<m\leq \omega$, 
we mean either  ${\sf VT}(l, m)$, $l<m\leq \omega$, or 
${\sf VT}(\omega, \omega)$; otherwise it is undefined.

Fix $2<n<\omega$. We have seen that if $\R_k$ having a strong $k$--blur and no $k+1$--dimensional relational 
basis (as in Theorem \ref{OTT2}) 
exists for all $n<k< \omega$,  then for $n\leq l, m\leq \omega$, if {\it ${\sf VT}(l, m)$ is defined, then it is 
true $\iff\ l=\omega$ and $m=\omega$.}
Here the parameter $l$ {\it measures  the closeness} to $L_{\omega, \omega}.$

Viewed from a different angle, the statement 
$\Psi(l, \omega)$, as clarified in a moment, says that $\sf VT$, fails in any {\it first order definable expansion of $L_n$} (with respect to ordinary models)
when the newly added first order definable connectives are built up of at most $l$ variables. 

Such fragments of $L_{\omega, \omega}$, extending $L_n$, were initially approached  by Jonsson in the context of relation algebras, 
and  further studied by B\'iro, Givant, N\'emeti, Tarski, S\'agi and others. 
The original purpose was to tame unruly behaviour of $\sf RRA$ and $\RCA_n$ like non--finite axiomatizability, 
but Biro \cite{Biro} showed that such (finite first order definable expansions) are inadequate to achieve this aim. 
Call such a logic $\L_l$. The analogous expansions for the calculas of relations $\L^{\times}$ is approached in \cite{Biro, HHM}. 

Each new connective added to $L_n$ is definable by a first order formula using $\leq l$ variables with at most $n$ free. 
The formation rule of formulas is defined inductively the expected way.
For example if $c$ is an $n$--ary connective and $\phi_0,\ldots \phi_{n-1}$ are $n$ formulas, then
$c(\phi_0,\dots, \phi_{n-1})$ is a formula. 
The logics $\L_l$ and $\L_n$ 
have the same signature; {\it its only that $\L_l$ has more first order definable connectives.}
Assume without loss that the signature consists of a single binary relation $\bold R$.

Accordingly, the models of $\L_l$ are the same as the models of $L_n$.
If $\Mo=(M, R)$, with $R\subseteq M\times M$ is such a model,  then $\bold R$ is interpreted to be the same binary relation $R\subseteq M\times M$ 
in $\L_l$ and $L_n$. 
The semantics of a newly introduced (first order definable) connective is defined by its defining formula. 
For example if $\Psi$ defines the unary connective $c$, say, 
then the semantics of $c$ in a model $\Mo$, 
is (inductively) defined for $s\in {}^nM$ by:  
$c(\phi)[s]\iff (\Mo, \Psi^{\Mo})\models \phi[s]$, where $\Psi^{\Mo}$ is the set of all 
$n$--ary assignments satisfying $\Psi$ in $\Mo$, 
that is to say, $\Psi^{\Mo}=\{s\in {}^nM: \Mo\models \Psi[s]\}$.

Now we formulate an algebraic result implying that $\sf VT$ fails for any finite first order definable expansions of $L_n$ \cite{Basim}. 
We deviate from the notation in \cite{Basim} by writing $\RCA_n^+$
for a first order definable expansion of $\RCA_n$ instead of $\RCA_{n,t}$. The signature $t$ of the last class is only required to be an 
expansion of the signature of $\CA_n$, more often a finite expansion. 
The logic $\L_l$ is the algebraisable (in the standard Blok--Pigozzi sense) logic corresponding to
the variety $\RCA_n^+$ when the algebraic non-cylindric operations of $\RCA_n^+$, if any,  corresponds to the newly added 
first order definable connectives (if any) in $\L_l$, like ${\sf c}_i$ correspond to $\exists x_iR$ ($i<n$), 
using $l$ variables. 

The next Theorem is yet one more  manifestation of the recurrent phenomena 
that the failure of $\sf VT$ for an algebraisable logic  $\L_{\V}$  and failure of atom--canonicity for its algebraic counterpart $\V$ 
have the habit to go hand in hand: 
\begin{theorem}\label{firstorder}
Let $2<n<\omega$. Let ${\sf RCA}_{n}^+$ be a first order definable
expansion of ${\sf RCA}_n$ such that the non--cylindric operations are  first order definable
by formulas using only finitely many variables $l>n$.  If $\RCA_{n}^+$ is completely additive, 
then it  is not  atom--canonical.
\end{theorem}
\begin{proof}
We basically re-prove $\Psi(l, \omega)$ by choosing $l$ large enough `covering' the extra 
variables used in forming the expansion.   This already gives that $\sf VT$ fails for $\L_l$. We use the notation in Theorem \ref{ANT}.
One takes  $l=\mathfrak{n}+1$ where $\mathfrak{n}$ is the finite number of variables involved 
in defining the newly added connectives.
Then $\A={\sf split} _n(\mathfrak{E}_k, J, E)\cong {\mathfrak Nr}_n {\sf split}_l(\mathfrak{E}_k(2, 3), J, E)$ 
where $k$, the finite number of non--identity elements in $\mathfrak{E}(2, 3)$,
is specified as in \cite[Lemma 5.1]{ANT}:  $l\geq 2n-1$ and 
$k\geq (2n-1)l$, $k\in \omega$, and $(J, E)$ is the strong $l$--blur 
of $\mathfrak{E}_k(2, 3)$ which exists by the choice of $k$.

Similarly to the proof of Theorem \ref{ANT}, one proves that $\A\in \RCA_n\cap \Nr_n\CA_l$, $\A$  is countable 
and has no complete representation.
Without loss of generality,
we can assume that we have only one extra operation $f$ definable by a first order formula
$\phi$, say, using $n<k<\omega$ variables with at most $n$ free variables. Now $\phi$ defines a $\CA_{k}$ term $\tau(\phi)$
which, in turn, defines the unary operation $f$ on $\A$, via $f(a)=\tau(\phi)^{\B}(a)$.  This is well defined, in the sense that $f(a)\in \A$, because
$\A\in {\sf Nr}_n\CA_{\mathfrak{n}+1}$ and the first order formula $\phi$ defining $f$, has at most $n$ free variables. Call the expanded structure $\A^*(\in {\sf RCA}_{n}^+$).
If  there is an isomorphism $f:\A^*\to \B$, where $\B\in {\sf Cs}_{n}^+$, a simple $\RCA_{n}^+$, 
having base $M$   such that $\bigcup_{x\in \At\A^*} f(x)={}^nM$,
then $f$ gives a complete representation of  $\A$. This is a
contradiction and we are done.

For non-atom--canonicity, take $\A^*\in \RCA_{n}^+$ as in the first part. By the condition of complete additivity of the  
first order definable operations as in  the hypothesis, we get that 
$\RCA_{n}^+$ is a completely additive variety, so $\Cm\At\A^*$ is the \de\ completion of $\A^*$. 
But  we also know that $\Rd_{ca}\Cm\At\A^*=\Cm\At\A\notin \RCA_n$, 
{\it a fortiori}, $\Cm(\At\A^*)\notin \RCA_{n}^+$, completing the proof of the second required. 
\end{proof}

We immediately obtain:
\begin{theorem}\label{first} Let $2<n\leq l<\omega$. Then Vaught's Theorem fails for any first order definable extension of $L_n$ using $l$ variables.
\end{theorem}
Observe that here the number of formulas used can be infinite as long as the number of variable occuring in each does not exceed 
$l$ variables.
Theorem \ref{first} says that, like completeness via finite Hilbert style axiomatizations \cite{Biro}, 
$\sf VT$, {\it a fortiori} $\sf OTT$, cannot be rescued  by first order definable expansions using finitely many variables.

{\it Now take the special case when $\L_l$ is the first order definable expansion of $L_n$ by adding $\leq n$--ary connectives for every first order formula using $l$ variables
with free variables among the first $n$.}

We denote the statement `any  countable atomic $\L_l$ theory has an $m$--square atomic model' by ${\sf VT}(\L_l, m).$
Then ${\sf VT}(\L_l, m)$ is equivalent to ${\sf V}(l, m)$. 
To unify notation, denote $L_{\omega, \omega}$ by $\L_{\omega}$ 
and we identify ${\sf VT}(\L_{\omega}, \omega)$ with (the true) ${\sf VT}(\omega, \omega)$ ($\sf VT$ for $L_{\omega, \omega}$). 
The next theorem is a more succint reformulation of item (2) of Theorem \ref{OTT2}.
\begin{theorem}If for each $2<n<m<\omega$, there exist a finite relation algebra $\R_m$ having $m-1$ strong blur and no
$m$--dimensional relational basis (equivalently $m$--square representation), then
for $n\leq l<m\leq \omega$ and $l=m=\omega$, ${\sf V}(\L_l, m)$ holds $\iff$ $l=m=\omega$.
\end{theorem}
While the parameter $l$ measures how close we are to $L_{\omega, \omega}$,  $m$ measures the {\it `degree' of squareness} of permitted models. 
Oberve that the parameter $l$ cannot be infinite here, for if $\A\in {\sf Nr}_n\CA_{\omega}$ is 
countable and atomic, then it is completely representable \cite[Theorem 5.3.6]{Sayedneat}, so the non-principal type of co-atoms will be omitted in any complete representation of $\A$. 

{\it The underlying intuition here is that 
the closer we get to $L_{\omega, \omega}$ and $\omega$--squareness, the closer we get to $\sf VT$.
As long as $l$ and $m$ are still finite with $l<m$, we are not there yet; $\sf VT$ fails for the first order definable expansion of $L_n$
using $l$ variables, with respect to $m$--square models.}

Alternatively, one can view the first limit as $l\to \infty$ (while fixing $m=\omega$) algebraically using ultraproducts as follows. Fix $2<n<\omega$. For each $2<n\leq l<\omega$, 
let $\R_l$ be the finite Maddux algebra $\mathfrak{E}_{f(l)}(2, 3)$ with strong $l$--blur
$(J_l,E_l)$ and $f(l)\geq l$ as specified in the proof of theorem \ref{main}.  Let ${\cal R}_l={\sf split}(\R_l, J_l, E_l)\in \sf RRA$ and let $\A_l=\Nr_n{\sf split}_l(\R_l, J_l, E_l)\in \RCA_n$. 
Then $(\At{\cal R}_l: l\in \omega\sim n)$, and $(\At\A_l: l\in \omega\sim n)$ are sequences of weakly representable atom structures 
that are not strongly representable with a completely representable 
ultraproduct. 
The (complex algebra) sequences $(\Cm \At{\cal R}_l: l\in \omega\sim n)$, 
$(\Cm\At\A_l: l\in \omega\sim n$) are typical examples of what Hirsch and Hodkinson call `bad Monk (non--representable) algebras' 
converging to  `good (representable) one, namely their (non-trivial) ultraproduct.  
Also, for $2<n\leq k <m<\omega$, $\A_k=\Nr_k\A_m$. 
Such sequences witness the non--finite axiomatizability of the class 
representable agebras  and the elementary closure of the class completely representable ones, namely, 
the class of algebras satisfying the Lyndon conditions; these conditions (to be recalled below) are defined in \cite{HHbook2}. 
From the last paragraph, we get:
\begin{corollary}\label{maddux}
\begin{enumerate} 
\item Assume that $2<n<\omega$. Then the varieties ${\sf RCA}_n$ and $\sf RRA$, as well as any finite expansion of either, are not finitely axiomatizable.
Furthermore, in each case there is a sequence of algebras generated by a single $2$-dimensional element outside the given variety,
 whose ultraproduct (relative to any non principal ultrafilter on $\omega$) is in the variety. 
\item (Maddux) The set of equations using only one variable that holds in each of these varieties cannot be derived from a finite set of equations valid in 
the variety. 
\end{enumerate}
\end{corollary}
The last corollary  recovers Biro's, Monk's and Maddux's classical results \cite{Biro, Monk, Maddux} 
on non--finite axiomatizability of $\sf RRA$s (representable relation algebras) and $\RCA_n$s, and any first order definable 
expansion of each, since algebras considered are generated by a 
single $2$--dimensional element; and they are based on relation algebras having the same properties, namely, the algebras (denoted above by) ${\cal R}_l$. 

For $2<n\leq l<m\leq \omega$, the statement $\Psi(l, m)$ is the negation of a special case of an Omitting Types Theorem which we define next:

\begin{definition}  Let $2<n\leq l<m\leq \omega$.  Let $T$ be an $L_n$ theory in a signature $L$ and $\Sigma$ be a set of $L$--formulas.

\begin{enumarab}

\item  We say that $T$ {\it $m$--omits $\Sigma$}, if there 
exists an injective homomorphism  $f:\A\to \wp(V)$ where $M=\bigcup_{s\in V}\rng(s)$ is an 
$m$--square representation of $\Fm_T$, and $\bigcap_{\phi\in \Sigma}f(\phi_T)=\emptyset$.

\item We say that $T$ {\it $l$-isolates $\Sigma$}, if there exists a formula $\phi$ using $l$ variables, such that
$\phi$ is consistent with $T$,  and
$T\models \phi\to \Sigma$. If not, we say that $T$ {\it $l$--locally} omits $\Sigma$.
 
\end{enumarab}
\end{definition}

Let $\lambda$ be a cardinal. 
Then ${\sf OTT}(l, m, \lambda)$ is the statement: 

{\it If $T$ is a countable $L_n$ theory, $\bold X=(\Gamma_i: i<\lambda)$ is family of types, 
such that  $T$, $l$--locally omits $\Gamma_i$ using at most $l$ variables for each $i\in \lambda$, 
then $T$,  $m$--omits $\Sigma_i$ for each $i<\lambda$.}

Let $n<2\leq l<m\leq \omega$.  Then it is not hard to see
that if $T$ is atomic, then  ${\sf OTT}(l, m, 1) \iff {\sf V}(l, m)$  by 
taking the single type consisting of co--atoms.  

\begin{remark} To prove that any $2<n<m<\omega$, there exists a $n$--variable type free valid formula schema that cannot be proved using $m-1$ variables, but can 
be proved using $m$ variables,
Hirsch,  Hodkinson and Maddux \cite{HHM} constructed for each such $m$  a finite relation algebra $\R_m$ such that $\R_m$ has an $m-1$ dimensional hyperbasis, but 
no $m$--dimensional hyperbasis.
To prove the weaker `flat version' of (**) one needs to construct, for each $2<n<m<\omega$, a finite relation algebra $\R_m$ having a strong $m-1$ blur, but no infinite 
$m$--dimensional hyperbasis. In this case blowing up and blurring $\R_m$ gives a(n infinite) relation algebra having an $m-1$ dimensional cylindric basis, 
{\it whose \de\ completion} has no $m$--dimensional hyperbasis. 
\end{remark}

\section{Vaught's Theorem in other modal fragments of $L_{\omega, \omega}$}

\subsection{Vaught's Theorem for $n$--variable guarded fragments, $\bold S5^n$ and $n$ products of uni--modal logics}

In this subsection, the reader is assumed to be familiar with basics of the correspondence theory between multi-modal logic and the theory of Boolean algebras with operators ($\sf BAO$s). 
We shall deal only with the case when the extra Boolean operations
are unary. The 
correspondence is established by forming quotient Tarski--Lindenbaum algebras.
We take it for granted that basic (semantical) notions in modal logic, such as 
Kripke frames, models based on Kripke frames, are known. 

The starting point of this duality is that algebraic terms correspond to modal formulas. By this identification we get:
$\F\models \phi \Longleftrightarrow  \Cm\F\models \phi=1$,
where $\F=(F, R_i)_{i\in I}$ is a relational structure, or {\it Kripke frame}
($I$ a non-empty indexing set), and $\Cm\F$ its {\it complex algebra} is an algebra having signature
$(f_i:i\in I)$ where each $f_i$ is a unary modality, in other words
an operator.  We often refer to the equation $\phi=1$ (when $\phi$ is viewed as a term) as the {\it algebraic translation} of the modal formula $\phi$. Every modal formula $\phi$
defines a formula in second order logic on Kripke frames, 
which we refer to as the {\it correspondant} of $\phi$.
Occasionally we identify atom structures of $\sf BAO$s with Kripke frames.

One can view certain modal logics as fragments of first order logic. But on the other hand,  one can also
turn the glass around giving $L_n$ a modal formalism, by viewing assignments as worlds, and 
existential quantifiers, the most prominent citizens of first order logic, as diamonds \cite{cml}. Here 
the worlds are not abstract entities
but have an inner structure. The worlds considered will be {\it sequences}. 

Let $U$ be a non--empty set. For $i<n$, define the binary relation $\equiv_i$ on ${}^nU$ as follows: For $s, t\in {}^nU$, 
$s\equiv_i t\iff s(j)=t(j)$ for all $j\neq i$. (The notation $\equiv_i$ will be used several times below often without further notice.)
The usual semantics for the existential quantifier $\exists x_i$ ($i<n$) now takes the following familar modal pattern (*) :
$$^n\Mo, s\models \exists x_i\phi \Longleftrightarrow (\exists s)(s\equiv _i t)\&{}^n\Mo, s\models \phi,$$
where ${}^n\Mo$ (for some $L_n$ structure $\Mo$) is viewed as the set of worlds, $s$ is a world and $\phi$ is an 
$L_n$ formula. Here Kripke frames are of the form $({}^n\Mo, \equiv_i, D_{ij})_{i, j<n}$ where $\equiv_i$ is the  binary accessibility 
relation defined as above (by replacing $U$ by $\Mo$), which is clearly an equivalence realtion for all $i<n$, 
and $D_{ij}$ the  a unary accessibility relation defined via $s\in D_{ij}\iff s(i)=s(j)$. If $\F$ is such a frame, then its complex algebra 
$\Cm\F$ is the {\it full} ${\sf Cs}_n$ with universe $\wp(^n\Mo)$. (The terminology {\it full} is used in \cite{HMT2}). 
Since $\RCA_n={\bf SP}\Cm \bold K$,  where $\bold K$ is the class of all such frames, 
then the variety of ${\sf BAO}$s corresponding to $L_n$ is $\RCA_n$.

Algebraically,  so--called {\it persistence properties} refer to closure of a variety
$\V$ under passage from a given algebra $\A\in \V$ to some `larger' algebra $\A^*$. 
Atom--canonicity is concerned with closure under forming \de\ completions (sometimes occuring in the literature under the name of {\it the minimal completions}), 
or if $\A\in \V$ is an atomic algebra, and $\V$ is completely additive, recall that $\A^*=\Cm\At\A$ is its \de\ completion.

Canonicity,  which is the most prominent persistence property in   
modal logic, the `large algebra' $\A^*$ is the canonical embedding algebra (or perfect) extension of $\A$, a complex algebra based on the 
{\it ultrafilter frame} of $\A$ whose underlying set is the set of all Boolean ultrafilters of $\A$. 
In modal logic, 
canonicity corresponds to the notion of a formula being {\it dpersistent} \cite[Definition 5.65, Proposition 5.85]{modal}.
A modal formula in $L_n$ 
is {\it canonical} if it is validated in the {\it canonical frame} of every normal modal logic containing $\phi$ \cite[Definition 4.30]{modal}. 
Algebraically, $\phi$ is canonical $\iff$ $\phi$ translates to an equation in the signature of $\RCA_n$ 
that is preserved under canonical extensions. 

An example of formulas that are both di-persistent and canonical (d-persistent) are the so-called {\it very simple Sahlqvist formulas} \cite[Theorem 5.90]{modal} which are, as the name 
suggests, instances of Sahlqvist formulas
\cite[Definition 3.51]{HHbook}.
From non-atom--canonicity and complete additivity of $\RCA_n$ as proved in \cite{Hodkinson};  
it follows that $\RCA_n$ is {\it not persistent} relative to \de\ completions; passing to a \de\ completion of an 
(atomic) $\RCA_n$
can get us out of the variety. Thus $\RCA_n$ cannot be axiomatized by Sahlqvist equations \cite{Venema}, cf. Theorems \ref{can} and \ref{sah}. Such equations 
are the algebraic translations of Sahlqvist formulas \cite[Definition 2.92]{HHbook}. 
So $L_n$ cannot be axiomatized by Sahlqvist formulas. On the other hand, it is not hard to see that $L_n$ is 
recursively enumerable. Nevertheless, it is difficult to find explicity (necessarily) infinite axiomatizations for it.

\begin{theorem}\label{bad}  Let $2<n<\omega$. 
\begin{enumerate}
\item As a fragment of $L_{\omega, \omega}$, $L_n$ is not finitely axiomatizable in any signature containing at least one binary relation symbol.  
\item As a multi-modal logic, any axiomatization of $L_n$ in a signature with infinitely many relation symbols each of arity $n$, 
must contain infinitely many 
propositional formulas.
\item $L_n$   cannot be axiomatized by any set of modal formulas having first order correspondents,
and, though canonical,  $L_n$ does not have a canonical axiomatization.  In particular, it does not have a Sahlqvist axiomatization. 
\item $\sf VT$ fails
for $L_n$ and its  clique guarded fragments with respect to $m$--square, {\it a fortiori}, $m$--flat models, for any finite $m\geq n+3$.
\end{enumerate}
\end{theorem}

\label{ln}
\begin{proof} All properties, with the exception of $\sf VT$ to be dealt with separately, can be distilled from corollary \ref{maddux}, together with the following algebraic results: 
For $2<n<\omega$, the class of  strongly representable ${\sf CA}_n$ atom structures is {\it not elementary} 
and  any (necessarily infinite) equational  axiomatization of $\RCA_n$ 
must contain infinitely many non--canonical sentences \cite{b} and infinitely many variables \cite{Andreka}.

To show that $\sf VT$ fails for $L_n$, let $\A$ be a countable, simple and atomic ${\sf RCA}_n$, such that $\Cm\At\A$ is not representable; such an $\A$ exists by Theorem \ref{can}. 
Then $\A$ is not completely representable, for a complete 
representation of $\A$ induces a(n) (ordinary) representation of $\Cm\At\A$. 
Assume that $\A\cong \Fm_T$; for some atomic and complete theory $T$. Since $\A$ does not have a complete representation, then it does not have 
an atomic representation. So $T$ is an atomic theory 
with no atomic model. For the $m$-clique guarded fragment of $L_n$ one takes be a countable, simple and atomic ${\sf RCA}_n$, 
such that $\Cm\At\A\notin \bold S\Nr_n\CA_m$ which exists by Theorem \ref{can}, and proceeds like above using the same argument in Theorem \ref{main}.

Let $m\geq n+3$. To show that $\sf VT$ fails for the clique guarded fragment of $L_n$ with respect to $m$--square (flat) models, 
one takes $\A$ be a countable, simple and atomic ${\sf RCA}_n$, such that $\Cm\At\A$ is outside 
$\bold S\Nr_n{\sf D}_m (\bold S\Nr_n\CA_m)$ and again proceeds like above. 
\end{proof}

Any Sahlqvist equation (formula) is a canonical one  
but the converse is not true. 
The Sahlqvist correspondence theorem states that every Sahlqvist formula corresponds to a first order definable class of Kripke frames.
Sahlqvist's definition characterizes a decidable set of modal formulas with first-order correspondents. Since it is undecidable, by Chagrova's theorem, 
whether an arbitrary modal formula has a first-order correspondent \cite[Theorem 3.56]{modal}, there are formulas with first-order frame conditions 
that are not Sahlqvist \cite[Example 3.57]{modal}. 
Now from Theorems \ref{b2},  \ref{can} and \cite{Venema}, we immediately get:
\begin{theorem}\label{sah}
Let $2<n<\omega$. Then any variety between 
$\sf RRA$ and $\bold S\Ra\CA_{3+k}$, as well as any variety between ${\sf RCA}_n$ and $\bold S\Nr_n\CA_{n+k}$,  
are not Sahlqvist axiomatizable. 
\end{theorem}

In what follows we study canonicity and Sahlqvist axiomatizablity 
for many $n$--dimensional multi--modal logics (other than $L_n$).  Like before, our investigations are algebraic. We start with a somewhat technical Lemma.
Fixing needed notation, let ${\sf (R)Df}_n$ denotes the class of (representable) diagonal free reducts of $\CA_n$s. 
The next Lemma is useful to transfer results from $\RCA_n$s to their diagonal free reducts.
It generalizes a result of Johnson \cite[Theorem 5.4.26]{HMT2}. 
Johnson's result is the special case when only finite intersections are allowed. Henceforth, we write $\Rd_{df}$ short hand for 
`diagonal free reduct'. If $\A\in \RCA_n$, then evidently $\Rd_{df}\A\in {\sf RDf}_n$. The next Lemma gives a sufficient condition for the converse (which need not happen in general) to hold:
\begin{lemma}\label{dfb} Let $2<n<\omega$. If $\A\in \CA_n$, is such $\Rd_{df}\A\in {\sf RDf}_n$,
and $\A$ is generated by $\{x\in \A: \Delta x\neq n\}$ using infinite intersections (together with the other cylindric operations) then
$\A\in \RCA_n$.
\end{lemma}
\begin{proof}  Easily follows from \cite[Lemma 5.1.50, Theorem 5.1.51]{HMT2}.  Indeed asssume that $\A\in \CA_n$, $\Rd_{df}\A$ is a diagonal free  
cylindric set algebra (of dimension $n$) with base $U$,  and $R\subseteq U\times U$ are as in the hypothesis of \cite[Theorem 5.1.49]{HMT2}. 
Let $E=\{x\in A:  (\forall x, y\in {}^nU)(\forall i <n)(x_iR y_i\implies (x\in X\iff y\in X))\}$.
Then $\{x\in \A: \Delta x\neq n\}\subseteq E$ and $E\in \CA_n$ 
is closed under infinite intersections. The required follows.
\end{proof}
Using the previous Lemma we can prove:
\begin{theorem} \label{df} For $2<n<\omega$, ${\sf RDf}_n$ is not atom--canonical, hence not closed under \de\ completions. 
\end{theorem}
\begin{proof} Let $\A$ and $\Cm\At\A$ be as in the proof of theorem \ref{can}. It suffices to show by Lemma \ref{dfb} that $\Cm\At\A$ 
is generated by elements whose dimension sets have cardinality $<n$ using infinite unions. 
We show that for any rainbow atom $[a]$ $a:n\to \Gamma$, $\Gamma$ a coloured graph, that
$[a]=\prod_{i<n} {\sf c}_i[a]$. 
Clearly $\leq $ holds. Assume that $b:n\to \Delta$, $\Delta$ a coloured graph, and $[a]\neq [b]$. We show that $[b]\notin \prod_{i<n} {\sf c}_i[a]$  by which we 
will be done. Because $a$ is not equivalent to $b$, we have one of two possibilities;
either $(\exists i, j<n) (\Delta(b(i), b(j)\neq \Gamma(a(i), a(j))$ or
$(\exists i_1, \ldots, i_{n-1}<n)(\Delta(b_{i_1},\ldots, b_{i_{n-1}})\neq \Gamma(a_{i_1},\ldots, a_{i_{n-1}}))$. 

Assume the first possibility: Choose  $k\notin \{i, j\}$. This is possible because $n>2$.  Assume for contradiction that $[b]\in {\sf c}_k[a]$. 
Then $(\forall i, j\in n\setminus \{k\})(\Delta(b(i), b(j))=\Gamma(a(i) a(j)))$. By assumption  and the choice of
$k$,  we get that $(\exists i, j\in n\setminus k)(\Delta(b(i), b(j))\neq \Gamma(a(i), a(j)))$, contradiction.
For the second possibility, one chooses $k\notin \{i_1, \ldots i_{n-1}\}$ and proceeds like the first case deriving an analogous contradiction.
Plainly each ${\sf c}_i[a]$ has dimension set of cardinality $<n$. Thus the set $\{{\sf c}_i[a]: a \text { is an atom and $i<n$}\}$ 
is the required set of generators because it generates the atoms, 
and the atoms, in turn,  using infinite unions generate $\Cm\At\A$.
\end{proof}

Fix $2<n<\omega$. 
Consider Kripke frames of the form $({}^nU, \equiv_i)_{ij<n}$ ($U$ a non--empty set).
Then the Kripke complete multi-modal logic consisting of the set of modal formulas that are valid in every frame 
of the above form is just the more familiar multi-modal logic ${\bf S5}^n$.
So the logic ${\bf S5}^n$ is an $n$--modal logic 
which is basically a disguised (equivalent) form of $L_n$ without equality, briefly $L_n^{-=}$. 
From non--atom--canonicity ${\sf RDf}_n$, we obtain (like the proof in theorem \ref{bad}):

\begin{corollary}\label{S5} Vaught's theorem, hence the omitting types theorem, fail for ${\bf S5}^n$ with respect to the aforementioned (usual) Kripke semantics.
\end{corollary}

We still fix $2<n<\omega$. We follow \cite{k} for terminology.
Many multi-modal logics can be considered as a combination of unimodal logics.
So in a sense any result on multi-modal logic sheds light on combining modal logics.
But dually, we can start with `the components' and form a modal logic
that somehow encompasses them or extends them; we seek a multimodal logic in which they
`embed'.
In such a process it is very natural to ask about {\it transfer results}, namely, these properties  of the components that transfer
to the combination, like axiomatizability (completeness),
and decidability that involves complexity of the satisfiability
problem. 

There are two versions depending on whether
the combination method is syntactic or semantical, namely, 
{\it fusions} and {\it products}, respectively.
In fusions the components do not interact which makes transfer results from the components to the fusion
easy to handle. In fact, the fusion of consistent modal logics is
a conservative extension
of the components, and the fusion of finitely many logics (this is well defined because the fusion operator is associative)
has the finite model property if each of its components does. Fusion also preserves decidability.
However, determining {\it degrees of complexity}
is quite intricate here. For example, it is not known {\it $\sf PSPACE$ or $\sf EXPTIME$ completeness} transfer under
formation  of fusions.

Products in modal logic can be seen as {\it orthogonal to guarding}.
Unlike fusions, product logics  are designed  semantically.
A product logic is {\it the multi-modal logic  of products of
Kripke complete frames}, so by definition it is also Kripke complete.

Special $n$-frames are the following
$n$-ary product frames.
\begin{definition} Given frames
$\F_0=(W_0, R_0) \ldots  \F_{n-1}=(W_{n-1}, R_{n-1})$
{\it their  product}
$\F_0\times \ldots \times \F_{n-1}$ is the relational structure
$(W_0\times \ldots \times W_{n-1}, \bar{R}_0,\ldots, \bar{R}_{n-1})_{i, j<n},$
where for each $i<n$, $R_i$ is the relation
$(u_0,\ldots u_{n-1})\bar{R}_i(v_0,\ldots, v_{n-1})$ if $u_iR_iv_i$ and
$u_k=v_k$ for $k\neq i$,  
\end{definition}

${\bf K}^n$ is the logic of $n$-ary product frames, of the form $(W_i, R_i)_{i<n}$ where for each $i<n$, $R_i$ is any any relation on $W_i$.
On the other hand, 
${\bf S5}^n$ can be regarded as the logic of $n$--ary product frames of the form 
$(W_i, R_i)_{i<n}$ such that for each $i<n$, $R_i$ is an equivalence relation. 
It is known that logics between $\bold K^n$ and ${\bf S5}^n$ 
are quite complicated, cf.  \cite{k} for a detailed overview. 
Theorem \ref{complexity} to be proved in a moment adds to their 
complexity.

We have seen in theorem \ref{AT} that $\sf VT$ holds for (i) when semantics are {\it guarded}, but it  fails for (ii) when semantics are {\it clique guarded}, cf. 
theorem \ref{can}.
We not know whether $\sf VT$ holds 
for $\bold K^n$ or for that matter any $\L\subsetneq {\bf S5}^n$ (containing $\bold K^n$), or not. 
A  strongly related question that is applicable in the present (modal context) is whether the multi-dimensional modal logic $\bold K^n$ is
Sahlqvist or not? So far, we know that 
${\bf S5}^n$ is {\it not Sahlqvist}.

On the other hand, it is known that modal languages can come to grips with 
a strong fragment of second order logic. Modal 
formulas translate to second order formulas, {\it their correspondants} 
on frames.  Some of these formulas can be {\it genuinely second order}; 
they are not equivalent to first order formulas. An example is the {\it McKinsey formula}: 
$\Box \Diamond p\to \Diamond \Box p$. This can be proved by showing that its correspondant violates 
the downward L\"owenheim- Skolem theorem  \cite[Example 3.11]{modal}.
The next theorem bears on the last two issues. 
For a class $\bold L$ of frames, let $\L(\bold L)$ be the class of modal formulas valid in $\bold L$. 
It is difficult to find explicity (necessarily) infinite axiomatizations for ${\bf S5}^n$ as well:

\begin{theorem}\label{complexity0} Let $2<n<\omega$. Then, like $L_n$, ${\bf S5}^n$  
cannot be axiomatized by any set of modal formulas having first order 
correspondents. In particular ${\bf S5}^n$ is not Sahlqvist.  Even more, 
${\bf S5}^n$ does not have a canonical axiomatization. 
\end{theorem}

\begin{proof} Let $\bold L$ be the class of square frames for ${\bf S5}^n$.
Then $\L(\bold L)={\bf S5}^n$ \cite[p. 192]{k}. But the class of frames $\F$ valid in $\L(\bold L)$ coincides 
with the class of  {\it strongly representable ${\sf Df}_n$ atom structures} which  
is {\it not elementary} as proved in \cite{b}. This gives the first required result for ${\bf S5}^n$. With lemma \ref{dfb} at our disposal, 
a slightly different proof can be easily distilled from the construction adressing $\CA$s in \cite{HHbook2} or \cite{k2}. 
We adopt the construction in the 
former reference, using the Monk--like $\CA_n$s  
${\mathfrak M}(\Gamma)$, $\Gamma$ a graph, as defined in
\cite[Top of p.78]{HHbook2}. 
For a graph $\G$, let $\chi(\G)$ denote it chromatic number. 
Then it is proved in {\it op.cit} that 
for any graph $\Gamma$, ${\mathfrak M}(\Gamma)\in \RCA_n$ 
$\iff$ $\chi(\Gamma)=\infty$.
By Lemma \ref{dfb},
$\Rd_{df}{\mathfrak M}(\Gamma)\in {\sf RDf}_n\iff \chi(\Gamma)=\infty$,
because $\mathfrak{M}(\Gamma)$ 
is generated by the set $\{x\in {\mathfrak M}(\Gamma): \Delta x\neq n\}$ using infinite unions.

Now we adopt the argument in \cite{HHbook2}. Using Erdos' probabalistic graphs \cite{Erdos}, 
for each finite
$\kappa$, there is a finite graph $G_{\kappa}$ with
$\chi(G_{\kappa})>\kappa$ and with no cycles of length $<\kappa$. 
Let $\Gamma_{\kappa}$ be the disjoint union of the $G_{l}$ for
$l>\kappa$.  Then $\chi(\Gamma_{\kappa})=\infty$, and so
$\Rd_{df}\mathfrak{M}(\Gamma_{\kappa})$ is representable.
Now let $\Gamma$ be a non-principal ultraproduct
$\Pi_{D}\Gamma_{\kappa}$ for the $\Gamma_{\kappa}$s. For $\kappa<\omega$, let $\sigma_{\kappa}$ be a
first-order sentence of the signature of the graphs stating that
there are no cycles of length less than $\kappa$. Then
$\Gamma_{l}\models\sigma_{\kappa}$ for all $l\geq\kappa$. By
Lo\'{s}'s Theorem, $\Gamma\models\sigma_{\kappa}$ for all
$\kappa$. So $\Gamma$ has no cycles, and hence by $\chi(\Gamma)\leq 2$.
Thus $\mathfrak{Rd}_{df}\mathfrak{M}(\Gamma)$
is not representable.  
(Observe that the 
the term algebra $\Tm\At(\mathfrak{M}(\Gamma))$
is representable (as a $\CA_n$), 
because the class of weakly representable atom structures is elementary \cite[Theorem 2.84]{HHbook}.)
Since Sahlqvist formulas have first order correspondants, then ${\bf S5}^n$ is not Sahlqvist.
The second 
part is due to the fact that 
any (necessarily infinite) equational  axiomatization of ${\bf S5}^n$ contains 
infinitely many non--canonical equations \cite{b}.

\end{proof}

In products {\it the modalities interact}, and this very interaction obviously adds to its components.
In fusions such interactions
are simply non--existent.
In such a process negative properties persist. The reason basically is that
products reflect the interaction of modalities; which is to be blamed
for the negative result. If they miss on anything
then they miss only on the uni--dimensional aspects of modalities and
these do not really contribute to negative results.
Negative results are caused
by the interaction of modalities not by their uni--dimensional properties.

For example, in a product of two uni-modal logics, a precarious Church Rosser condition
on modalities is created via $\Diamond_i\Box_jp=\Box_j\Diamond_ip$  ($i,j <n$).
Indeed, it is
known that the theory of two commuting confluence
closure operators is undecidable.
Nevertheless, {\it commuting closure operations} alone can be harmless like in the case of
many cylindric--like algebras of dimension $2$,
but the interaction of the two modalities
expressed by the confluence is potentially harmful.
The product logic of two countable time flows is not even recursively enumerable, furthermore the modal logic of
$(\mathbb{N}, <)$ is undecidable.

Compared to fusions, there are very few general transfer results for products,  in fact here {\it the exact opposite occurs}.
Nice properties do not transfer, rather the lack of transfer is the norm, particularly
concerning finite axiomatizability and decidability.

Now addressing more logics, we formulate and prove the next `complexity' result which generalizes some results proved in the last theorem to more $n$-modal logics.
\begin{theorem}\label{complexity} Let $2<n<\omega$. Let $\L$ be any canonical logic between $\bold K^n$ and ${\bf S5}^n$. Then the following is true of $\L$:
\begin{enumerate}
\item  $\L$ is not finitely axiomatizable,
\item It is undecidable to tell whether a finite frame is a frame for $\L$. In particular, $\L$ is undecidable, 
\item $\L$ cannot be axiomatized by canonical formulas. In particular, $\L$ is not Sahlqvist.
\end{enumerate}
\end{theorem}
\begin{proof}   
First item follows from that the variety ${\sf RDf}_n$ is not finitely axiomatizable and that ${\bf S5}^n$ 
is finitely axiomatizable over any (canonical) $\L$ as specified in the statement of the theorem \cite[Theorem 2.2.7]{k}.
Item (2) follows from the main result in \cite{k2}. 
Although the logic $\bold K^n$ has the finite model property,
it encodes the tiling problem and so it is undecidable. 
In fact, there are three dimensional
formulas that are valid in all higher dimensional finite products, but can be falsified on an infinite
frame. (Here validity in higher dimensions is meaningful, because if $n<m$ then the modal logic ${\bf K}^n$
embeds into ${\bf K}^m$.)
Furthermore, in \cite{k2} it is proved that it is undecidable to tell whether a finite frame
is a frame for $\L$,  and this gives
strong non--finite axiomatizability results, cf. Theorem \ref{decidability} and obviously implies undecidability.

For item (3), again we use that ${\bf S5}^5$ is finitely axiomatizable over $\L$ 
and any equational  axiomatization of ${\sf RDf}_n$ contains infinitely many non--canonical equations as proved in \cite{b}. Thus any axiomatization of $\L$ must 
contain infinitely many non--canonical equations (observe that this plainly implies item (1)). 
Since a Sahlqvist equation is canonical, then there is no 
Sahlqvist axiomatization of $\L$.
\end{proof}

Given a variey $\V$ of $\sf BAO$s, it is often desirable to find a concrete elementary (first order definable) class of atom stuctures $S$ that generates $\V$, in the sense that
${\bf HSP}\Cm S=\V$. This corresponds to a modal logic being {\it elementary generated} \cite{g}.
The most natural candidate is $\At\V$, {\it the class of atom structures of all atomic algebras in $\V$}, 
because $\At\V$ is elementary, when $\V$ happens to be completely additive \cite[Theorem 2.84]{HHbook}. But for $\RCA_n$, which is a 
completely additive variety, 
non--atom canonicity, proved in theorem \ref{can}, implies that $\At(\RCA_n)$ does not work. For a canonical variety $\V$, as is the case
with $\RCA_n$, the class $\Str\V=\{\F: \Cm\F\in \V\}$, plainly contained in $\At\V$,  always generates $\V$ in the {\it strong 
sense} that $\bold S\Cm\Str(\V)=\V$. So here the variety is obtained without the intervention 
of homomorphic images and products; hence the terminology strong.
In this case $\V$, being generated by a class of complex algebras is said to be
{\it complete} \cite{g}.  In particular $\RCA_n$ is a complete variety.  

But $\Str\V$ is not always elementary as is indeed the case with $\V=\RCA_n$ \cite{k2}, cf. Theorem \ref{complexity0}.
Nevertheless, there are other elementary subclasses of $\At(\RCA_n)$ that generates $\RCA_n$ in the strong sense (without the help of homomorphic images and products), 
making $L_n$  elementary generated. 
One such class is the class of atom structures satisfying the so--called Lyndon conditions \cite{HHbook2}. This implies that $\RCA_n$ is {\it canonical}, because
it is generated by a class of atom structures closed under ultraproducts \cite[Fact 2.86]{HHbook}, \cite[Theorem 4.50]{modal}.
\footnote{Whether every canonical variety is generated by a class of atom structures closed under ultraproducts
remained an open question for quite some time until answered negatively \cite{g2}. 
The question is attributed to Fine.}
We summarize (some of) the above in a (more) rigorous definition:
\begin{definition}
A variety $\V$ is {\it elementary generated} if there exists an elementary class $\sf K$ of atom structures such that 
${\bf HSP}\Cm\sf K=\V$. We say that $\V$ is elementary generated {\it in the strong sense} if there is an elementary class $\sf K$ such that
$\bold S\Cm \sf K=\V$. In the latter case, $\V$ is said to be {\it strongly complete}, while in the former case $\sf V$ is said to be {\it complete}.
The definition of the last two notions does not require 
that  $\sf K$ is elementary \cite{g}. 
We say that $V$ is {\it atomically generated} 
if it generated by its atomic members. 
\end{definition}

In Theorem \ref{can} using a rainbow construction we showed that $\bold S\Nr_n\CA_{n+k}$ is not 
atom--canonical, for any $k\geq 3$.
It is known that ${\sf Str}({\RCA}_n)$
is not elementary, a result of Hirsch and Hodkinson's \cite[Corollary 3.7.2]{HHbook2} proved using Monk-like algebras, cf. Theorem \ref{complexity0}. 
Fix finite $k>2$.  Then $\V_k={\sf Str}(\bold S\Nr_n\CA_{n+k})$ is not elementary 
$\implies   \V_k$ is not--atom canonical \cite[Theorem 2.84]{HHbook}.  
{\it But the converse implication, namely, 
$\V_k$ is not atom--canonical $\implies {\sf Str}(\V_k)$ not elementary does not hold in general.}
However, it is not hard to show that there has to be a finite $k<\omega$ such that ${\sf Str}(\V_k)$ is not elementary
as shown next; specifying such a $k$ is another story. The case $l=\omega$ in the statement of the next theorem is the limiting case obtained in \cite{HHbook2} and Theorem \ref{complexity0}.

\begin{theorem} For $2<n<\omega$, there is a finite $m\geq n+2$ such that the clique guarded fragments of $L_n$ with respect to $l$--flat models for any $m\leq l\leq \omega$,
is not axiomatizable by formulas with first order correspondances.
\end{theorem}
\begin{proof} Like before, our proof is algebraic. It suffices to show, also like before,  that there is an $m\geq n+2$, such that the class ${\sf Str}(\bold S\Nr_n\CA_m)$ is not elementary.
Let $(\A_i: i\in \omega)$ be a sequence of (strongly) representable $\CA_n$s with $\Cm\At\A_i=\A_i$
and $\A=\Pi_{i/U}\A_i$ is not strongly representable with respect to any non-principal ultraflter $U$ on $\omega$.
Such algebras exist by the proof of Theorem \ref{complexity0}. Hence $\Cm\At\A\notin \bold S\Nr_n\CA_{\omega}=\bigcap_{i\in \omega}\bold S\Nr_n\CA_{n+i}$, 
so $\Cm\At\A\notin \bold S\Nr_n\CA_{l}$ for all $l>k$, for some $k\in \omega$, $k>n$. 
But for each such $l$, $\A_i\in \bold S\Nr_n\CA_l(\subseteq {\sf RCA}_n)$, 
so $(\A_i:i\in \omega)$ is a sequence of algebras such that $\Cm\At(\A_i)\in \bold S\Nr_n\CA_{l}$ $(i\in I)$, but 
$\Cm(\At(\Pi_{i/U}\A_i))=\Cm\At(\A)\notin \bold S\Nr_n\CA_l$, for all $l\geq k$.
\end{proof}

We show, using a known celebrated (algebraic) result proved in \cite{AT}, that a certain guarded fragment of $L_n$ enjoys a Vaught's theorem. Let $n$ be a finite ordinal $>1$. 
Consider now Kripke frames of the form $({}V, \equiv_i, D_{ij})_{i,j<n}$ where $V\subseteq {}^nU$ is a ${\sf D}_n$ unit ($U$ a non--empty set)
where for $i<j<n$, $\equiv_i$ and $D_{ij}$ are defined 
like in Kripke frames of $L_n$
restricted to the set $V$ of worlds.
The Kripke complete multi-modal logic $(\L_n$) consisting of the set of modal formulas that are valid in every frame of the above form 
admits a finite Hilbert style complete axiomatization and enjoys a Vaught's theorem. 
Note that $\L_n$ is a {\it guarded} fragment of $L_n$ in the sense of \cite{v}. We shall see that when we guard semantics negative properties formulated for $L_n$ in theorem 
\ref{ln} vanish.
Since ${\sf D}_n$ is the variety of modal 
algebras corresponding
to $\L_n$, that is, ${\sf D}_n ={\bf SP}\Cm\bold L^n$, it suffices to prove the next algebraic result. 

\begin{theorem}\label{AT} \cite{AT} Let $2<n<\omega$. Then the variety ${\sf D}_n$ is finitely axiomatizable and 
every atomic algebra in ${\sf D}_n$ is completely representable. Furthermore, ${\sf D}_n$ is canonical, atomically generated, 
atom--canonical, thus elementary generated in the strong sense. 
\end{theorem}
\begin{proof}
\cite{mlq}. 
The first part is proved in \cite{AT} and differently using games in \cite{mlq}. 
Assume that $\A\models {\sf Mod}\Sigma$, where $\Sigma$ is the axiomatization given in \cite{AT}. Then $\Sigma$   
is positive in the wider sense,  hence Sahlqivst,  so ${\sf D}_n$ is 
canonical and atom--canonical \cite[Theorems 2.77, 2.80]{HHbook}.
Let $\V={\sf D}_n$.
Being atomically generated in the strong sense 
follows from  canonicity and atom-canonicity using the following reasoning.
Take the class  of all atom structures of atomic algebras in $\V$, namely, the class 
$\At\V$. Then by complete addtivity of $\V$, the class $\At\V$ is elementary 
\cite[Theorem 2.84]{HHbook}. 
We show that $\At\V$ strongly generates $\V$. 
Assume that $\A\in \V$. By canonicity $\A^+\in \V$, so $\At\A^+\in \At\V$. But 
plainly $\A^+=\Cm\At\A^+$, thus $\A^+\in \Cm\At\V$. Since $\A$ embeds into $\A^+$, we get that 
$\A\in \bold S\Cm\At\V$. We have proved that $\V\subseteq \bold S\Cm\At\V$. The opposite inclusion follows from 
atom-canonicity and we are done.   
\end{proof}
Unlike $L_n$, we have:
\begin{theorem}\label{good}
The logic $\L_n$  is complete, decidable, has the finite model property,  
the Craig and Beth definablity property, and enjoys a $\sf VT$.  
\end{theorem}
\begin{proof} For Craig interpolation, we give an outline of the idea:
We know that ${\sf G}_n$ is axiomatized by a set of Sahlqvist  equations, 
so is canonical.
The first order correspondents of this set of positive equations translated to the class
$\bold L={\sf Str}({\sf G}_n)=\{\F: \Cm\F\in {\sf G}_n\}$ will be {\it Horn formulas},
hence {\it clausifiable} and so $\bold L$ is closed under finite {\it zigzag products}. By
\cite[Lemma 5.2.6, p.107]{marx}, ${\sf G}_n$ has the super amalgamation property
which is the algebraic equivalent of the 
Craig interpolation property. (Undefined terminology can be found in the same 
referred to lemma of {\it op.cit}.)
Now we show that the logic $\L_n$ enjoys a $\sf VT$.
Let $T$ be an atomic $\L_n$ theory. We assume without loss that $T$ is complete.
Hence $\Fm_T\in \bold I{\sf D}_n$ is an atomic simple algebra. 
Let $\Mo$ be the base of a complete representation of $\A$. Then $\Mo$ is an  atomic model of $T$.
The rest is known \cite{v}.
\end{proof}

We prove (for a change) 
some positive properties for $L_n$ and its clique guarded fragments.  
Let ${\sf LCA}_n$ denote the elementary class of ${\sf RCA}_n$s satisfying the Lyndon conditions. 
We stipulate that $\A\in {\sf LCA}_n\iff$ $\A$ is atomic and $\At\A$ satifies the Lyndon conditions \cite[Definition 3.5.1]{HHbook2}.
By an equivalent definition $\A\in {\sf LCA}_n$ $\iff$ \pe\ has a \ws\ in the atomic game $G_k(\At\A)$ for all $k\in \omega$. The \ws\ in $G_k$ 
is coded in a first order sentence, namely, the  $k$th Lyndon condition. It is easy to verify 
that ${\sf LCA}_n$ coincides with the elementary closure of ${\sf CRCA}_n$.
The reasoning in the last part of the proof of Theorem \ref{AT} shows that, for {\it any completely additive variety $\V$}, 
both canonicity and atom--canonicity implies elementary generation. 
However, the conditions (of canonicity and atom-canonicity) are not necessary; $\RCA_n$ is not atom-canonical \cite{Hodkinson}, 
but it is elementary generated as shown next. 

In the following Theorem, our formulation (and proofs) are algebraic. Like before by $\omega$--flat and $\omega$--square 
representations, we understand ordinary representations.

\begin{theorem}\label{square} Let $2<n<m\leq \omega$.
\begin{enumerate}
\item There are at least two elementary classes that generate 
$\RCA_n$ (equivalently the variety of  algebras having $\omega$--flat representations) in the strong sense,

\item ${\bf El}\Nr_n\CA_{\omega}\subsetneq {\sf LCA}_n$. Furthermore, 
for any elementary class $\sf K$ between ${\bf El}\Nr_n\CA_{\omega}$ and ${\sf LCA}_n$, ${\sf RCA}_n$ is elementary generated by 
$\At\sf K$,

\item  The class $\RCA_{n, m,s}$ ($\RCA_{n, m,f}$) of algebras having $m$-square (flat) representations is a variety that is elementary 
generated in the strong sense, atomically generated, canonical, but not atom-canonical for $m\geq n+3$.

\end{enumerate}
\end{theorem}

\begin{proof} Throughout the proof fix $2<n<\omega$.  For elementary generation, we know that the elementary class $\At\RCA_n$ does not work because 
$\RCA_n$ is not atom--canonical \cite{Hodkinson}, cf. Theorem \ref{can}, 
and that the class ${\sf Str}\RCA_n$ does not work either, because, sure enough a generating class,
${\sf Str}\RCA_n$  is not  an elementary one in the first place \cite{HHbook2}, cf. Theorem \ref{complexity0}.  So we have to look elsewhere.
For $\RCA_n$, one takes instead $\bold K={\sf LCAS}_n$, 
the class of atom structures satisfying the Lyndon conditions.
One proves that $\RCA_n=\bold S\Cm{\sf LCAS}_n$ exactly like the relation algebra case \cite{HHbook}. 
(The idea is the same idea used in item (3)).

We give a different class elementary generating ${\sf RCA}_n$. For an atom structure $\bf At$, let $\F(\bf At)$ be the subalgebra of $\Cm\bf At$
consisting of all sets of atoms in $\bf At$ of the form $\{a\in {\bf At}: {\bf At}\models \phi(a, \bar{b})\}(\in \Cm {\bf At})$,
for some first order formula $\phi(x, \bar{y})$ of the signature
of $\bf At$  and some tuple $\bar{b}$ of atoms.  It is easy to check that $\F(\bf At)$ is indeed 
a subalgebra of $\Cm \bf At$;  and that
$\Tm{\bf At}\subseteq \F(\bf At)\subseteq \Cm \bf At$; cf. \cite[item (3), p. 456]{HHbook} for the analogous definition for 
relation algebras. Let ${\sf FOAS}_n$  be the class of all atom structures whose first order algebra is representable. Then it can be proved, similarly 
to the $\RA$ case that  
${\sf LCAS}_n\subseteq {\sf SRAS}_n\subseteq {\sf FOAS}_n\subsetneq {\sf WRAS}_n$, where ${\sf SRSA}_n$ and ${\sf WCAS}_n$ are
the classes of strongly and weakly representable atom structures of dimension $n$, respectively as defined in \cite{HHbook2}. 
Then $\RCA_n= \bold S\Cm{\sf LCAS}_n\subseteq \bold S\Cm{\sf FOAS}_n\subseteq {\sf RCA}_n$. 
One way to show that ${\sf LCAS}_n\subsetneq {\sf FOAS}_n$ is that these two classes are elementary and are 
separated by the non-elementary class ${\sf SRAS}_n$,
that is, the non-elementary class 
${\sf SRCA}_n$, cf. \cite{HHbook2, k2},  lies (strictly) in between. 

We prove item (2): It suffices to show that the class of atomic algebras in 
${\sf Nr}_n\CA_{\omega}$ is contained in the class of atomic algebras whose atom structures are in
${\sf LCAS}_n$, since the last class is elementary. This follows from lemma \ref{n}, 
since if $\A\in \Nr_n\CA_{\omega}$ is atomic, then \pe\ has a \ws\ in $F^{\omega}(\At\A)$, hence in $G_{\omega}(\At\A)$, {\it a fortiori}, \pe\ has a \ws\ 
in $G_k(\At\A)$ for all $k<\omega$. By  definition of the Lyndon conditions, we are done.
Now we show that $\At{\bf El}\Nr_n\CA_{\omega}$ generates $\RCA_n$.
Let ${\sf FCs}_n$ denote the class of {\it full} $\Cs_n$s, that is ${\sf Cs}_n$s  
having universe $\wp(^nU)$
($U$ non--empty set). 
First we show that ${\sf FCs}_n\subseteq \Cm\At\Nr_n\CA_{\omega}$.
Let $\A\in {\sf FCs}_n$.  Then $\A\in \Nr_n\CA_{\omega}$, hence $\At\A\in \At\Nr_n\CA_{\omega}$ 
and $\A=\Cm\At\A\in \Cm\At\Nr_n\CA_{\omega}$.
The required now follows from the following chain of inclusions: 

$\RCA_n={\bf SP}{\sf FCs}_n\subseteq {\bf SP}\Cm\At(\Nr_n\CA_{\omega})\subseteq {\bf SP}\Cm\At({\bf El}\Nr_n\CA_{\omega})\subseteq 
{\bf SP}\Cm \At{\sf K}\subseteq {\bf SP}\Cm{\sf LCAS}_n\subseteq  {\sf RCA}_n.$

By Lemma \ref{n}, we have $\Nr_n\CA_{\omega}\subseteq {\sf LCA}_n$, and since the last class is elementary, then 
${\bf El}\Nr_n\CA_{\omega}\subseteq {\sf LCA}_n$.
To show strictness of the last inclusion, let $V={}^n\Q$ and let ${\A}\in {\sf Cs}_n$ have universe $\wp(V)$.
Then clearly $\A\in {\sf Nr}_{n}\CA_{\omega}$. To see why,  let 
$W={}^{\omega}\Q$ and let $\D\in {\sf Cs}_{\omega}$ have universe $\wp(W)$.
Then the map $\theta: \A\to \wp(\D)$ defined via $a\mapsto \{s\in W: (s\upharpoonright \alpha)\in a\}$, 
is an injective homomorphism from $\A$ into $\mathfrak{Rd}_{n}\D$ that is onto 
$\mathfrak{Nr}_{n}\D$.
Let $y$ denote the following $n$--ary relation:
$y=\{s\in V: s_0+1=\sum_{i>0} s_i\}.$ Let $y_s$ be the singleton containing $s$, i.e. $y_s=\{s\}$
and ${\B}=\Sg^{\A}\{y,y_s:s\in y\}.$ It is shown in \cite{SL} that $\{s\}\in \B$, for all $s\in V$. 
Now $\B$ and $\A$ having same top element $V$, share the same atom structure, namely, the singletons, so $\B\subseteq_ d \A$ 
and $\Cm\At\B=\A$. Furthermore, plainly $\A, \B\in {\sf CRCA}_n$; the identity maps establishes a complete representation for both, 
since $\bigcup_{s\in V}\{s\}=V$. 
Now $\B\in {\sf CRCA}_n\subseteq {\sf LCA}_n$, 
and as proved in \cite{SL}  
$\B\notin {\bf  El}\Nr_{n}{\sf CA}_{n+1}$, then $\B$ witnesses the required strict inclusion.

We prove item (3):  We start with squareness.
That ${\sf RCA}_{n, m,s}$ is closed under $\bold S$ and $\bold P$ is straightforward.
We check closure under homomorphic images.
Assume that $\A$ has an $m$--square representation 
and let $h:\A\to \B$  be surjective. We want to show that $\B$ has an $m$--square representation, too.
We have $\A^+$ has an $m$--dimensional basis $\cal M$.   
Let $K= ker(h)$  and let $z=-\sum^{\A^+}K$. (The last sum exists, because $\A^+$ is complete).
Let $\D$ be the relativization of $\A^+$ to $z.$
Then $\D$ is complete and atomic.
Define  $g :\A^+\to \D$ by $a\mapsto a\cdot z.$ Then $\B$ embeds in $\D$ via $b\mapsto g(a)$, 
for any $a\in h^{-1}[b]$.
Hence  $\{N\in {\cal M}: N(\bar{x})\in \D\}$ 
is a basis for $\D$, and since $\B$ (up to isomorphism) is a subalgebra of $\D$, 
we get that $\B$ has an $m$--square representation, because $\D$ does.

For elementary generation, we are done with case $m=\omega$.
Accordingly, assume that $2<n<m<\omega$. Let ${\sf CRCA}_n^{m,s}$ be the class of $\CA_n$s having complete $m$--square representations.
We first specify 
the elementary closure of ${\sf CRCA}_n^{m,s}$ using games. 
Define the class ${\sf LCA}_n^{m,s}$ as follows: 
$\A\in {\sf LCA}_n^{m,s} \iff \A$ is atomic 
and \pe\ has a \ws\ in $G_k^m(\At\A)$ for all $k<\omega$. 
It is not hard to show that the last condition, which is an `$m$ approximation'  to a Lyndon condition 
can be coded in a first order sentence. 

For brevity, denote the elementary class $\At{\sf LCA}_n^{m,s}$ by ${\sf LCAS}_n^{m,s}$. 
We now show elementary generation in the strong sense by proving that  ${\sf RCA}_{n,m,s}=\bold S\Cm {\sf LCAS}_n^{m,s}.$
Assume that $\A$ has an $m$--square representation, 
then $\A^+$ has a complete $n$--square representation \cite{HHbook}. 
This can be proved, using ideas of Hirsch and Hodkinson, 
by taking an $\omega$--saturated model of the consistent first order theory stipulating the 
existence of an $m$--square representation for $\A$, as the base of the complete $m$--
square representation for $\A^+$. 
In more detail, let $\Mo$ be an $\omega$--saturated model,
of this theory.
One  defines an injective complete embedding $h: \A^+\to \wp(1^{\Mo}).$
First note that the set $f_{\bar{x}}=\{a\in A: a(\bar{x})\}$
is an ultrafilter in $\A$, whenever $\bar{x}\in \Mo$ and $\Mo\models 1(\bar{x}).$
Now $\A^+=\Cm(\Uf\A)$, where $\Uf\A$ is the ultrafilter atom structure (frame) of $\A$ based on its Stone space.
For $S\subseteq \Uf\A$, 
let $h(S)=\{\bar{x}\in 1^{\Mo}: f_{\bar{x}}\in S\}.$
We check only injectivity using saturation. For the (ideas used in) rest of  the proof the reader is referred to \cite[Corollary 13.18]{HHbook}.
It suffices to show that for any ultrafilter $F$ of $\A$ 
which is an atom in $\A^+$, we have
$h(\{F\})\neq 0$.
Let $p(\bar{x})=\{a(\bar{x}): a\in F\}$. Then this type is finitely satisfiable.
Hence by $\omega$ saturation $p$ is realized in $\Mo$ by $\bar{y}$, say.
Now $\Mo\models 1(\bar{y})$ and $F\subseteq f_{\bar{x}}$,
since these are both ultrafilters, equality holds.
It follows that $\At\A^+\in \At({\sf CRCA}_n^{m,s})\subseteq {\sf LCA}_n^{m,s}$.
By $\A\subseteq \A^+=\Cm\At\A^+$, 
and $\Cm\At\A^+\in \Cm{\sf LCAS}_n^{m,s}$, we are done. 

For atom generation, we need to show that, for $m<\omega$, $\RCA_{n, m,s}=\bold S{\sf CRCA}_n^{m,s}$.  
One side is obvious since ${\sf CRCA}_n^{m,s}\subseteq \RCA_{n, m,s}$ and the last class is a variety. 
Now we show $\subseteq$. Let $\A\in \RCA_{n, m, s}$. Then $\A\subseteq \A^+$ and $\A^+\in {\sf CRCA}_{n}^{m,s}$. 
The case $m=\omega$ is entirely 
analogous by observing that for any $\A\in \RCA_n$, 
$\A^+\in {\sf CRCA}_n$ 
and that all algebras in $\CRCA_n$ are atomic. We leave the rest to the reader. 
Non-atom--canonicity is proved in Theorem \ref{n}, and canonicity for $\bold S\Nr_n{\sf D}_m$ follows from that
if $\B\in {\sf D}_m$  and $\A\subseteq \Nr_n\B$, 
then $\A^+\subseteq \Nr_n\B^+$.

Now we consider flatness: The class ${\sf RCA}_{n, m, f}$ coincides with the variety $\bold S\Nr_n\CA_m$, hence it is a variety. 
Let ${\sf CRCA}_n^{m, f}$ be the class of $\CA_n$s 
having complete infinitary $m$--flat representations, equivalently complete $m$--smooth representations. 
Elementary generation follows by taking the class of atom structures $\bold K={\bf El}{\sf CRCA}_n^{m, f}$ 
and proceeding like above:
If $\A\in\RCA_{n, m, f}$ then $\A^+\in {\sf CRCA}_n^{m,f}$ so $\A\subseteq \A^+=\Cm\At\A^+\in \bold S\Cm {\sf CRCA}_n^{m,f}\subseteq
\bold S\Cm {\bf El}{\sf CRCA}_n^{m,f}$. The opposite inclusion is obvious.
Non atom--canonicity follows from Theorem \ref{n}.
For being atomically  generated, one proves, 
like above,  that $\RCA_{n, m, f}=\bold S\bold L$ where $\bold L$ is the class of (atomic) $\CA_n$s having complete 
$m$--flat representations.
\end{proof}

Not every modal logic is determined by a class of discrete frames. In fact, completeness for discrete frames, unlike completeness with respect to general Kripke frames,
is a non--trivial property called {\it di--completeness}. If a Kripke incomplete logic is di--complete, then the logic in question is still 
reasonably well-behaved.
It is easy to show that ${\sf RDf}_n$ is atomically generated, too. Indeed, using a saturation argument like in the last proof, it can be shown 
that if $\A\in {\sf RDf}_n$, then $\A^+$ is completely representable (as a ${\sf Df}_n$). 
Baring in mind that di--completeness corresponds to atom generation, then from the previous theorem and the last observation, 
together with Theorem \ref{n}, and \cite[Theorem 2]{Andreka}, we get:
\begin{corollary}\label{cl}
For $2<n<m\leq \omega$, the clique guarded fragments of $L_n$ with respect to $m$--square models and with respect to $m$--flat models
are elementary generated, di--complete,  but are not Sahlqvist for $m\geq n+3$.  
Furthermore, for  $m\geq n+2$ such fragments, together with ${\bf S5}^n$,  are not  finitely axiomatizable. In case of $m$--flatness, 
and ${\bf S5}^n$, it is undecidable to tell whether a finite frame is a frame for the logic at hand, hence
the last two multi--modal logics are undecidable.  
\end{corollary}
 
Theorems \ref{square} and \ref{can} alert us to the fact that the notions of di-persistence and di-completeness are distinct. 
In fact, the two notions are distinct both ways. Relatively easy examples are the following: 
The Van Benthem formula 
is di-persistent but axiomatizes a di-incomplete logic.  Conversely, the Church--Rosser formula $\Diamond \Box p\to \Box\Diamond p$
is Kripke complete, hence di-complete, but is not di-persistent.

\begin{remark} 
Square Tarskian semantics can be seen as a {\it limiting case} of  (i) relativized (guarded) semantics, (ii) clique guarded semantics,
(iii) products. (We know that in (i) $\sf VT$ holds, while in (ii) and (iii) $\sf VT$ fails.)

In the first case, namely (i),  the limit is taken on a varying set of worlds (states) $V$ approaching the square
$^nU$. The most severe {\it relativization} of states, is that one takes $V$ to be an arbitrary set of $n$-ary sequences without any closure conditions.
Tarskian semantics is the limiting case when $V={}^nU$.
In   the semantics we dealt with in theorem \ref{AT}, we took the set of states in Kripke frames to be 
top element of  ${\sf D}_n$s. Completely analogous results hold by taking the set of worlds in Kripke frames 
to be ${\sf G}_n$ units instead of ${\sf D}_n$ units, giving another (richer) guarded fragment of $L_n$.
Here the domain of Kripke frames is required to be 
closed under  permuting sequences, as well.

Let $2<n<m\leq \omega$. In the second case (ii) viewed semantically, the ordinal $m$, which is the parameter that measures the degree of `squarness' or `flatness', is allowed to grow;
getting closer and closer to  a {\it genuine} $\omega$--flat representation. 
If $\Mo$ is the base of  an $m$-flat representation of $\A\in \CA_n$, then  witnesses of cylindrifiers 
are allowed `more space' as $m$ gets larger  measured by the $n$--Gaifmann hypergraph ${\sf C}^n(\Mo)$. 
Syntactically more and more degrees of freedom or dimensions
are created, so that an algebra $\A\in \CA_n$ has an $m$--flat representation
$\iff$ $\A$ {\it neatly embeds} into a $\CA_m$ having top element ${\sf C}^n(\Mo)$. This a `truncated' neat embedding theorem; at the limit we get Henkin's classical 
neat embedding therem, namely, 
$\RCA_n=\bold S\Nr_n\CA_{\omega}$.
 
In the last case of products, {\it the accessibility relations along the components are changed} approaching those of ${\bf S5}^n$. In ${\bf S5}^n$, recall that,  
all the accessibility relations along the components are the universal one.
\end{remark}

\section{Positive omitting types theorems for $L_n$}

In this section, unless otherwise explicitly indicated, $n$ is  finite and $>2$.
We prove positive omitting types theorems for $L_n$. 
Now we turn to proving omitting types theorems for certain (not all) $L_n$ theories. 
But first an algebraic  definition of omitting a given family of types:
\begin{definition}\label{definition} Let $\lambda$ be a cardinal. 
If $\A\in {\sf RCA}_{n}$ and $\bold X=(X_i: i<\lambda)$ is  family of subsets of $\A$, we say that {\it $\bold X$ is omitted in $\C\in {\sf Gs}_{n}$}, 
if there exists an isomorphism 
$f:\A\to \C$ such that $\bigcap f(X_i)=\emptyset$ for all $i<\lambda$.  When we want to stress the role of $f$, 
we say that $\bold X$ is omitted in $\C$ via $f$. 
If $X\subseteq \A$ and $\prod X=0$, 
then we refer to $X$ as a {\it non-principal type} of $\A$.
\end{definition}

We further need to recall certain cardinals that play a key role in (positive) omitting types theorems for $L_{\omega, \omega}$.
Let $\sf covK$ be the cardinal used in \cite[Theorem 3.3.4]{Sayed}.
The cardinal $\mathfrak{p}$  satisfies $\omega<\mathfrak{p}\leq 2^{\omega}$
and has the following property:
If $\lambda<\mathfrak{p}$, and $(A_i: i<\lambda)$ is a family of meager subsets of a Polish space $X$ (of  which Stone spaces of countable Boolean algebras are examples)  
then $\bigcup_{i\in \lambda}A_i$ is meager. For the definition and required properties of $\mathfrak{p}$, witness \cite[p. 3, pp. 44-45, corollary 22c]{Fre}. 

Both cardinals $\sf cov K$ and $\mathfrak{p}$  have an extensive literature.
It is consistent that $\omega<\mathfrak{p}<\sf cov K\leq 2^{\omega}$ \cite{Fre},
so that the two cardinals are generally different, but it is also consistent that they are equal; equality holds for example in the Cohen
real model of Solovay and Cohen.  Martin's axiom implies that  both cardinals are the continuum.

To prove the main result in this section on positive omitting types theorems, we need the following Lemma due to Shelah:
\begin{lemma} \label{sh} Assume that $\lambda$ is an infinite regular cardinal. 
Suppose that $T$ is a first order theory,
$|T|\leq \lambda$ and $\phi$ is a formula consistent with $T$,  then there exist models $\Mo_i: i<{}^{\lambda}2$, each of cardinality $\lambda$,
such that $\phi$ is satisfiable in each,  and if $i(1)\neq i(2)$, $\bar{a}_{i(l)}\in Mo_{i(l)}$, $l=1,2,$, $\tp(\bar{a}_{l(1)})=\tp(\bar{a}_{l(2)})$,
then there are $p_i\subseteq \tp(\bar{a}_{l(i)}),$ $|p_i|<\lambda$ and $p_i\vdash \tp(\bar{a}_ {l(i)})$ ($\tp(\bar{a})$ denotes the complete type realized by
the tuple $\bar{a}$)
\end{lemma}
\begin{proof} \cite[Theorem 5.16, Chapter IV]{Shelah}.
\end{proof}
In $L_{\omega, \omega}$ an atomic model for a countable  atomic theory  (which exists by $\sf VT$) 
omits all non--principal types. The last item in the next theorem is the $`L_n$ version (expressed algebraically)' of  this property.
In the theorem $n<\omega$:
\begin{theorem}\label{i} Let $\A\in \bold S_c\Nr_n\CA_{\omega}$ be countable.  Let $\lambda< 2^{\omega}$ and let 
$\bold X=(X_i: i<\lambda)$ be a family of non-principal types  of $\A$.
Then the following hold:
\begin{enumerate}
\item If $\A\in \Nr_n\CA_{\omega}$ and the $X_i$s are maximal non--principal ultrafilters,  then $\bold X$ can be omitted in a ${\sf Gs}_n$.
\item Every subfamily of $\bold X$ of cardinality $< \mathfrak{p}$ can be omitted in a ${\sf Gs}_n$; in particular, every countable 
subfamily of $\bold X$ can be omitted in a ${\sf Gs}_n$. Furthermore, if $\A$ is simple, then every subfamily 
of $\bold X$ of cardinality $< \sf covK$ can be omitted in a ${\sf Cs}_n$, 
\item If $\A$ is atomic,  with countably many atoms, 
then any family of non--principal types can be omitted in an atomic ${\sf Gs}_n$; in particular, 
$\bold X$ can be omitted in an atomic ${\sf Gs}_n$; if $\A$ is simple, we can replace ${\sf Gs}_n$ by 
${\sf Cs}_n$.
\end{enumerate}
\end{theorem}
\begin{proof}

For the first item we assume that $\A$ is simple (a condition that can be easily removed). We have $\prod ^{\B}X_i=0$ for all $i<\kappa$ because,
$\A$ is a complete subalgebra of $\B$.  that is if $S\subseteq \A$ and $y\in A$
is such that $\sum ^{\A}S=y$, then $\sum^{\B}S=y$.

The condition that $\A\in \Nr_n\CA_{\omega}(\subsetneq \bold S_c\Nr_n\CA_{\omega})$ is crucial to prove that $\A$ is a complete subalgebra of $\B$.
To see why, assume that $S\subseteq \A$ and $\sum ^{\A}S=y$, and for contradiction that there exists $d\in \B$ such that
$s\leq d< y$ for all $s\in S$. Then, assuming that $A$ generates $\B$, we can infer that $d$ uses finitely many dimensions in $\omega\sim n$, 
$m_1,\ldots, m_n$, say.
Now let $t=y\cdot -{\sf c}_{m_1}\ldots {\sf c}_{m_n}(-d)$.
We claim that $t\in \A=\Nr_{n}\B$ and $s\leq t<y$ for all $s\in S$. This contradicts
$y=\sum^{\A}S$.

The first required follows from the fact that $\Delta y\subseteq n$ and that all indices in $\omega\sim n$ that occur in $d$
are cylindrified.  In more detail, put $J=\{m_1, \ldots, m_n\}$ and let 
$i\in \omega\sim n$, then
$${\sf c}_{i}t={\sf c}_{i}(-{\sf c}_{(J)} (-d))={\sf c}_{i}-{\sf c}_{(J)} (-d)$$
$$={\sf c}_{i} -{\sf c}_{i}{\sf c}_{(J)}( -d)
=-{\sf c}_{i}{\sf c}_{(J)}( -d)
=-{\sf c}_{(J)}( -d)=t.$$
We have shown that ${\sf c}_it=t$ for all $i\in \omega\sim n$, thus $t\in \Nr_{n}\B=\A$. If $s\in S$, we show that $s\leq t$. We know that $s\leq y$. Also $s\leq d$, so $s\cdot -d=0$.
Hence $0={\sf c}_{m_1}\ldots {\sf c}_{m_n}(s\cdot -d)=s\cdot {\sf c}_{m_1}\ldots {\sf c}_{m_n}(-d)$, so
$s\leq -{\sf c}_{m_1}\ldots {\sf c}_{m_n}( -d)$, hence  $s\leq t$ as required. 

We finally check that $t<y$. If not, then
$t=y$ so $y \leq -{\sf c}_{m_1}\ldots {\sf c}_{m_n}(-d)$ and so $y\cdot  {\sf c}_m\ldots {\sf c}_{m_n}(-d)=0$.
But $-d\leq {\sf c}_m\ldots {\sf c}_{m_n}(-d)$,  hence $y\cdot -d\leq y\cdot  {\sf c}_m\ldots {\sf c}_{m_n}(-d)=0.$
Hence $y\cdot -d =0$ and this contradicts that $d<y$. We have proved that $\sum^{\B}X=1$ showing that $\A$ is indeed a complete subalgebra of $\B$.

Since $\B$ is a locally finite,  we can assume 
that $\B=\Fm_T$ for some countable consistent theory $T$.
For each $i<\kappa$, let $\Gamma_i=\{\phi/T: \phi\in X_i\}$.

Let ${\bold F}=(\Gamma_j: j<\kappa)$ be the corresponding set of types in $T$.
Then each $\Gamma_j$ $(j<\kappa)$ is a non-principal and {\it complete $n$-type} in $T$, because each $X_j$ is a maximal filter in $\A=\mathfrak{Nr}_n\B$.
(*) Let $(\Mo_i: i<2^{\omega})$ be a set of countable
models for $T$ that overlap only on principal maximal  types; these exist by lemma \ref{sh}.

Asssume for contradiction that for all $i<2^{\omega}$, there exists
$\Gamma\in \bold F$, such that $\Gamma$ is realized in $\Mo_i$.
Let $\psi:{}2^{\omega}\to \wp(\bold F)$,
be defined by
$\psi(i)=\{F\in \bold F: F \text { is realized in  }\Mo_i\}$.  Then for all $i<2^{\omega}$,
$\psi(i)\neq \emptyset$.
Furthermore, for $i\neq j$, $\psi(i)\cap \psi(j)=\emptyset,$ for if $F\in \psi(i)\cap \psi(j)$, then it will be realized in
$\Mo_i$ and $\Mo_j$, and so it will be principal.

This implies that $|\bold F|=2^{\omega}$ which is impossible. Hence we obtain a model $\models T$ omitting $\bold X$
in which $\phi$ is satisfiable. The map $f$ defined from $\A=\Fm_T$ to ${\sf Cs}_n^{\Mo}$ (the set algebra based on $\Mo$ \cite[4.3.4]{HMT2})
via  $\phi_T\mapsto \phi^{\Mo},$ where the latter is the set of $n$--ary assignments in
$\M$ satisfying $\phi$, omits $\bold X$. Injectivity follows from the facts that $f$ 
is non--zero and $\A$ is simple. 

For (2), we can assume that $\A\subseteq_c \Nr_n\B$, $\B\in \Lf_{\omega}$. 
We work in $\B$. Using the notation on \cite[p. 216 of proof of Theorem 3.3.4]{Sayed} replacing $\Fm_T$ by $\B$, we have $\bold H=\bigcup_{i\in \lambda}\bigcup_{\tau\in V}\bold H_{i,\tau}$
where $\lambda <\mathfrak{p}$, and $V$ is the weak space ${}^{\omega}\omega^{(Id)}$,  
can be written as a countable union of nowhere dense sets, and so can 
the countable union $\bold G=\bigcup_{j\in \omega}\bigcup_{x\in \B}\bold G_{j,x}$.  
So for any $a\neq 0$,  there is an ultrafilter $F\in N_a\cap (S\setminus \bold H\cup \bold G$)
by the Baire category theorem. This induces a homomorphism $f_a:\A\to \C_a$, $\C_a\in {\sf Cs}_n$ that omits the given types, such that
$f_a(a)\neq 0$. (First one defines $f$ with domain $\B$ as on p.216, then restricts $f$ to $\A$ obtaining $f_a$ the obvious way.) 
The map $g:\A\to \bold P_{a\in \A\setminus \{0\}}\C_a$ defined via $x\mapsto (g_a(x): a\in \A\setminus\{0\}) (x\in \A)$ is as required. 
In case $\A$ is simple, then by properties of $\sf covK$, $S\setminus (\bold H\cup \bold G)$ is non--empty,  so
if $F\in S\setminus (\bold H\cup \bold G)$, then $F$ induces a non--zero homomorphism $f$ with domain $\A$ into a $\Cs_n$ 
omitting the given types. By simplicity of $\A$, $f$ is injective.

For the last item if $\A\in {\bold S}_n\Nr_n\CA_{\omega}$ is atomic and has countably many atoms, 
then any complete representation of $\A$, equivalently, an atomic representation of $\A$, equivalently, a representation of $\A$ 
omitting the set of co--atoms is as required.  This complete representation exists by \cite[Theorem 5.3.6]{Sayed}.
If $\A$ is simple and completely representable, 
then it is completely represented  
by a ${\sf Cs}_n$, and we are done. 
\end{proof}
$\sf LCRA$ is the class of $\sf RA$s satisfying the Lyndon conditions defined analogously to $\sf LCA_n$.
The next Theorem will be used several times below. We assume familiarity with constructing atomic 
relation algebras by specifying forbidden triples of atoms, see e.g.  \cite{Maddux, HHbook}.

\begin{theorem}\label{bsl} Let $\kappa$ be an infinite cardinal. Then there exists an atomless $\C\in \CA_{\omega}$ such that the relation algebra $\R=\Ra\C$ is atomic and satisfies
that $|\R|=2^{\kappa}$  and $\R\in {\sf LCRA}$, but $\R$ is not completely representable. For all $2<n<\omega$, $\B=\Nr_n\C$ is atomic, $|\B|=2^{\kappa}$, $\B\in {\sf LCA}_n$, 
but $\B$ is not completely representable, too. Furthermore, for any such $n$, $\R=\Ra\Nr_n\C$. 
In particular, $\bold S_c\Ra\CA_{\omega}\nsubseteq\sf CRRA$, $\bold S_c\Nr_n\CA_{\omega}\nsubseteq {\sf CRCA}_n$  
and $\sf CRRA$ and ${\sf CRCA}_n$ are not elementary \cite{HH}.
\end{theorem}
\begin{proof}
Let $\kappa$ be an infinite cardinal. 
We specify the atoms and forbidden triples. The atoms are $\Id, \; \g_0^i:i<2^{\kappa}$ and $\r_j:1\leq j<
\kappa$, all symmetric.  The forbidden triples of atoms are all
permutations of $({\sf Id}, x, y)$ for $x \neq y$, \/$(\r_j, \r_j, \r_j)$ for
$1\leq j<\kappa$ and $(\g_0^i, \g_0^{i'}, \g_0^{i^*})$ for $i, i',
i^*<2^{\kappa}.$  In other words, we forbid all the monochromatic triangles.
Write $\g_0$ for $\set{\g_0^i:i<2^{\kappa}}$ and $\r_+$ for
$\set{\r_j:1\leq j<\kappa}$. Call this atom
structure $\alpha$. 
Let $\R=\Tm(\alpha)$. Then it is proved in \cite{bsl} that $\R$ has no complete representation.
Let $S$ be the set of all atomic $\R$-networks $N$ with nodes
$\kappa$ such that $\{\r_i: 1\leq i<\kappa: \r_i \text{ is the label
of an edge in $N$}\}$ is finite.
Then $S$ is an amalgamation class, that is for all $M, N
\in S$ if $M \equiv_{ij} N$,  there is $L \in S$ with
$M \equiv_i L \equiv_j N$.
So $\Ca(S)\in \CA_\omega$.
 Now let $X$ be the set of finite $\R$-networks $N$ with nodes
$\subseteq\kappa$ such that:
(1) each edge of $N$ is either (a) an atom of
$\A$ or (b) a cofinite subset of $\r_+=\set{\r_j:1\leq j<\kappa}$ or (c)
a cofinite subset of $\g_0=\set{\g_0^i:i<2^{\kappa}}$ and

(2)  $N$ is `triangle-closed', i.e. for all $l, m, n \in \nodes(N)$ we
have $N(l, n) \leq N(l,m);N(m,n)$.  That means if an edge $(l,m)$ is
labelled by $\sf Id$ then $N(l,n)= N(m,n)$ and if $N(l,m), N(m,n) \leq
\g_0$ then $N(l,n)\cdot \g_0 = 0$ and if $N(l,m)=N(m,n) =
\r_j$ (some $1\leq j<\omega$) then $N(l,n)\cdot \r_j = 0$.
For $N\in X$ let $N'\in\Ca(S)$ be defined by:
$\set{L\in S: L(m,n)\leq
N(m,n) \mbox{ for } m,n\in \nodes(N)}.$
For $i\in \omega$, let $N\restr{-i}$ be the subgraph of $N$ obtained by deleting the node $i$.
If $N\in X, \; i<\omega$, then it is proved in \cite{bsl} that $\cyl i N' =
(N\restr{-i})'$. 

Let $X' = \set{N':N\in X} \subseteq \Ca(S)$.
Then the subalgebra of $\Ca(S)$ generated by $X'$ is obtained from
$X'$ by closing under finite unions.
Let $\C$ be the subalgebra of $\Ca(S)$ generated by $X'$.
Then we claim that $\B = \Ra(\C)$.
To see why, each element of $\R$ is a union of a finite number of atoms,
possibly a co-finite subset of $\g_0$ and possibly a co-finite subset
of $\r_+$.  Clearly $\A\subseteq\Ra(\C)$.  Conversely, each element
$z \in \Ra(\C)$ is a finite union $\bigcup_{N\in F}N'$, for some
finite subset $F$ of $X$, satisfying $\cyl i z = z$, for $i > 1$. Let $i_0,
\ldots, i_k$ be an enumeration of all the nodes, other than $0$ and
$1$, that occur as nodes of networks in $F$.  Then, $\cyl
{i_0} \ldots
\cyl {i_k}z = \bigcup_{N\in F} \cyl {i_0} \ldots
\cyl {i_k}N' = \bigcup_{N\in F} (N\restr{\set{0,1}})' \in \R$.  So $\Ra(\C)
\subseteq \A$.
$\R$ is relation algebra reduct of $\C\in\CA_\omega$ but has no complete representation.
Let $\B=\mathfrak{Nr}_n \C$ ($2<n<\omega$). Then
$\B\in \Nr_n\CA_{\omega}$, is atomic, but has no complete representation, for a complete representation of $\B$ induces one of $\R$. 

We show that the $\omega$--dilation $\C$ is atomless. 
For any $N\in X$, we can add an extra node 
extending
$N$ to $M$ such that $\emptyset\subsetneq M'\subsetneq N'$, so that $N'$ cannot be an atom in $\C$.

We further show that $\B\in {\bf El}{\sf CRCA}_n$ reproving that ${\sf CRCA}_n$ is not elementary \cite{HH}. Since $\B\in \Nr_n\CA_{\omega}$, 
then by Lemma \ref{n}, \pe\ has a \ws\ in $G_{\omega}(\At\B)$, hence in 
$G_k(\At\B)$ for all finite
$k$, so using ultrapowers followed by an elementary chain argument, we get that $\B\equiv \C$,
where $\C$ is a countable and completely representable $\CA_n$, 
hence  $\B\in {\bf El}{\sf CRCA}_n$ and we are done. The same technique works for $\RA$s using $\R$.
Since ${\sf LCA}_n={\bf El}{\sf CRCA}_n$, that $\sf LCRA={\bf El}{\sf CRRA}$, we are done. 
\end{proof}
For an ordinal $\alpha$, let $\PEA_{\alpha}$ denotes the class of $\alpha$--dimensional polyadic equality algebras as 
defined in \cite[Definition 5.4.1]{HMT2}, For $\alpha$ an infinite ordinal, $n<\alpha$, and $\B\in \PEA_{\alpha}$, the {\it $n$-neat reduct of $\B$}, in symbols 
$\mathfrak{Nr}_n\B$, is defined like the $\CA$ case, it is straightforward to show that $\mathfrak{Nr}_n\B\in \QEA_n$.
Let $2<n<\omega$. Theorem \ref{bsl} shows that there atomic algebras in $\Nr_n\CA_{\omega}$ lacking 
complete representations. But the $\omega$--dilation constructed therein (denoted by $\C$) is atomless.
What if we require that the $\omega$--dilation $\D$ say,  is atomic, can we in this case completely represent $\Nr_n\D$? 
The next theorem  gives an affirmative answer in case the $\omega$--dilation is a $\PEA_{\alpha}$ for any infinite $\alpha$. But first a lemma:
\begin{lemma}\label{join} Assume that $\A$ and $\D$ are Boolean algebras such  
that $\A\subseteq_c \D$. If $\D$ is atomic, then $\A$ is atomic.
\end{lemma}

To prove the theorem, for simplicity of notation we choose $\alpha=\omega$.
\begin{theorem}\label{pa} If $2<n<\omega$ and $\D\in \PEA_{\omega}$ is atomic, then
any complete subalgebra of ${\mathfrak Nr}_n\D$ is completely representable (as a $\QEA_n$).  
\end{theorem}\begin{proof}
We identify notationally set algebras with their domain. 
Assume that  $\A\subseteq_c{\Nr}_n\D$, where $\D\in \PEA_{\omega}$ is atomic. 
We want to completely represent $\A$.  Let $c\in \A$ be non--zero. We will find  a homomorphism $f:\A\to \wp(^nU)$ 
such that $f(c)\neq 0$, and $\bigcup_{y\in Y}f(y)={}^nU$, whenever $Y\subseteq \A$ satisfies $\sum^{\A}Y=1$.

Assume for the moment  (to be proved in a while) that $\A\subseteq_c \D$. Then  $\A$ is atomic, because $\D$ is. For brevity, let $X=\At\A$. 
Let $\mathfrak{m}$ be the local degree of $\D$, $\mathfrak{c}$ its effective cardinality 
and let $\beta$ be any cardinal such that $\beta\geq \mathfrak{c}$
and $\sum_{s<\mathfrak{m}}\beta^s=\beta$; such notions are defined in \cite{DM}.

We can assume that $\D=\Nr_{\omega}\B$, with $\B\in \PEA_{\beta}$\cite[Theorem 5.4.17]{HMT2}. 
 For any ordinal $\mu\in \beta$, and $\tau\in {}^{\mu}\beta$,  write $\tau^+$ for $\tau\cup Id_{\beta\setminus \mu}(\in {}^\beta\beta$).
Consider the following family of joins evaluated in $\B$,
where $p\in \D$, $\Gamma\subseteq \beta$ and
$\tau\in {}^{\omega}\beta$:
(*) $ {\sf c}_{(\Gamma)}p=\sum^{\B}\{{\sf s}_{{\tau^+}}p: \tau\in {}^{\omega}\beta,\ \  \tau\upharpoonright \omega\setminus\Gamma=Id\},$ and (**):
$\sum {\sf s}_{{\tau^+}}^{\B}X=1.$
The first family of joins exists \cite[Proof of Theorem 6.1]{DM}, and the second exists, 
because $\sum ^{\A}X=\sum ^{\D}X=\sum ^{\B}X=1$ and $\tau^+$ is completely additive, since
$\B\in \PEA_{\beta}$. 

The last equality of suprema follows from the fact that $\D=\Nr_{\omega}\B\subseteq_c \B$ and the first
from the fact that $\A\subseteq_c \D$. We prove the former, the latter is exactly the same replacing
$\omega$ and $\beta$, by $n$ and $\omega$, respectivey, proving that $\Nr_n\D\subseteq_c \D$, hence $\A\subseteq_c \D$.  

We prove that $\Nr_{\omega}\B\subseteq_c \B$. The proof is entirely analogous to the proof in the first item of Theorem \ref{i} 
except that we use infinite cylindrifications (in the signature of $\PEA_{\omega}$).
In more detail, assume that $S\subseteq \D$ and $\sum ^{\D}S=1$, and for contradiction, that there exists $d\in \B$ such that
$s\leq d< 1$ for all $s\in S$. Let  $J=\Delta d\setminus \omega$ and take  $t=-{\sf c}_{(J)}(-d)\in {\D}$.
Then  ${\sf c}_{(\beta\setminus \omega)}t={\sf c}_{(\beta\setminus \omega)}(-{\sf c}_{(J)} (-d))
=  {\sf c}_{(\beta\setminus \omega)}-{\sf c}_{(J)} (-d)
=  {\sf c}_{(\beta\setminus \omega)} -{\sf c}_{(\beta\setminus \omega)}{\sf c}_{(J)}( -d)
= -{\sf c}_{(\beta\setminus \omega)}{\sf c}_{(J)}( -d)
=-{\sf c}_{(J)}( -d)
=t.$
We have proved that $t\in \D$.
We now show that $s\leq t<1$ for all $s\in S$, which contradicts $\sum^{\D}S=1$.
If $s\in S$, we show that $s\leq t$. By $s\leq d$, we have  $s\cdot -d=0$.
Hence by ${\sf c}_{(J)}s=s$, we get $0={\sf c}_{(J)}(s\cdot -d)=s\cdot {\sf c}_{(J)}(-d)$, so
$s\leq -{\sf c}_{(J)}(-d)$.  It follows that $s\leq t$ as required. Assume for contradiction that 
$1=-{\sf c}_{(J)}(-d)$. Then ${\sf c}_{(J)}(-d)=0$, so $-d =0$ which contradicts that $d<1$. We have proved that $\sum^{\B}S=1$,
so $\D\subseteq_c \B$.

Let $F$ be any Boolean ultrafilter of $\B$ generated by an atom below $a$. We show that $F$
will preserve the family of joins in (*) and (**).
We next use a simple topological argument. 
One forms nowhere dense sets in the Stone space of $\B$ 
corresponding to the aforementioned family of joins 
as follows:

The Stone space of (the Boolean reduct of) $\B$ has underlying set,  the set of all Boolean ultrafilters
of $\B$. For $b\in \B$, let $N_b$ be the clopen set $\{F\in S: b\in F\}$.
The required nowhere dense sets are defined for $\Gamma\subseteq \beta$, $p\in \D$ and $\tau\in {}^{\omega}\beta$ via:
$A_{\Gamma,p}=N_{{\sf c}_{(\Gamma)}p}\setminus \bigcup_{\tau:\omega\to \beta}N_{{\sf s}_{\tau^+}p}$, 
and $A_{\tau}=S\setminus \bigcup_{x\in X}N_{{\sf s}_{\tau^+}x}.$
The principal ultrafilters are isolated points in the Stone topology, so they lie outside the nowhere dense sets defined above.
Hence any such ultrafilter preserve the joins in (*) and (**). 
Fix a principal ultrafilter $F$ preserving (*) and (**) with $a\in F$. 
For $i, j\in \beta$, set $iEj\iff {\sf d}_{ij}^{\B}\in F$.

Then by the equational properties of diagonal elements and properties of filters, it is easy to show that $E$ is an equivalence relation on $\beta$.
Define $f: \A\to \wp({}^n(\beta/E))$, via $x\mapsto \{\bar{t}\in {}^n(\beta/E): {\sf s}_{t\cup Id_{\beta\sim n}}^{\B}x\in F\},$
where $\bar{t}(i/E)=t(i)$ ($i<n$) and $t\in {}^n\beta$. 
Then ity s not hard to check that  $f$ is a well--defined homomorphism (from (*)) and that $f$ is complete 
such that $f(c)\neq 0$. The last follows by observing that $Id\in f(c)$.
(Let $V={}^\beta\beta^{(Id)}$. To show that $f$ is well defined, it suffices to show  that for all $\sigma, \tau\in V$, 
if $(\tau(i), \sigma(i))\in E$ for all $i\in \beta$,   then 
for any $x\in \A$, ${\sf s}_{\tau}x\in F\iff {\sf s}_{\sigma}x\in F.$)  

We show that the non--zero homomorphism 
$f$ is an atomic, hence, a complete representation.
 By construction, for every $s\in {}^{n}(\beta/E)$, 
there exists $x\in X(=\At\A)$, such that ${\sf s}_{s\cup Id_{\beta\sim n}}^{\B}x\in F$, 
from which we get the required, namely, that  
$\bigcup_{x\in X}f(x)={}^n(\beta/E).$
\end{proof}

The next Theorem uses the full power of Theorem \ref{sh} addressing omitting types in possibly uncountable theories.
The Theorem is stated without proof in \cite{Sayed}, cf. \cite[Theorem 3.2.9]{Sayed}.
\begin{theorem}\label{mark} Let $\kappa$ be a regular infinite cardinal and $n<\omega$. Assume that $\A\in \Nr_n\CA_{\omega}$ with $|A|\leq {\kappa}$, 
that $\lambda$ is a cardinal $< 2^{\kappa}$, and that $\bold X=(X_i: i<\lambda)$ is a family of non-principal types  of $\A$.
 If the $X_i$s are maximal non--principal ultrafilters of $\A$,  then $\bold X$ can be omitted in a ${\sf Gs}_n$.
Furthermore, the condition of maximality cannot be dispensed with.
\end{theorem}
\begin{proof} From Lemma \ref{sh} using the same reasoning in item (1) of Theorem \ref{i}.
The second part follows from the construction in Theorem \ref{bsl} where an atomic algebra  $\B\in \Nr_n\CA_{\omega}$ 
with uncountably many atoms that is not completely representable is  given. This implies that the maximality condition 
cannot be dispensed with; else the set  of co--atoms of $\B$ call it $X$ will be a non--principal type that cannot be omitted, 
because any ${\sf Gs}_n$ omitting $X$ yields a complete representation of 
$\B$, witness too  the last paragraph in \cite{Sayed}.  
\end{proof}

\section{Complete representations}

For finite $n>2$, the existence of a countable atomic $\A\in \RCA_n$ lacking  a complete representation implies that Vaught's theorem fails for $L_n$.  
Accordingly, we devote this section to studing the algebraic (semantical) notion of complete representations.
We characterize the elementary closure of the class ${\sf CRCA}_n$ of completely representable $\CA_n$s using neat embeddings, 
and we study first order definability of 
several classes strongly related to the (non--elementary) class ${\sf CRCA}_n$ \cite{HH}.

Fix a finite ordinal $n>2$. For a class $\bold K$ (recall that) ${\bf El}\bold K$ denotes its elementary closure. 
Following \cite{HMT2}, ${\sf Cs}_n$ denotes the class of {\it cylindric set algebras of dimension $n$}, and ${\sf Gs}_n$ 
denotes the class of {\it generalized cylindric set algebra of dimension $n$}; $\C\in {\sf Gs}_n$, if $\C$ has top element
$V$ a disjoint union of cartesian squares,  that is $V=\bigcup_{i\in I}{}^nU_i$, $I$ is a non-empty indexing set, $U_i\neq \emptyset$  
and  $U_i\cap U_j=\emptyset$  for all $i\neq j$. The operations of $\C$ are defined like in cylindric set algebras of dimension $n$ 
relativized to $V$. 
Recall that:

\begin{definition} An algebra $\A\in {\sf CRCA}_n$ $\iff$ there exists $\C\in {\sf Gs}_n$, and an isomorphism $f:\A\to \C$ such that for all $X\subseteq \A$, 
$f(\sum X)=\bigcup_{x\in X}f(x)$, whenever $\sum X$ exists in $\A$. In this case, we say that $\A$ is {\it completely representable via $f$.}
\end{definition}
It is known that $\A$ is completely representable via $f:\A\to \C$, where $\C\in {\sf Gs}_n$ has top element $V$ say 
$\iff$ $\A$ is atomic and $f$ is {\it atomic} in the sense that 
$f(\sum \At\A)=\bigcup_{x\in \At\A}f(x)=V$ \cite{HH}.
  
We often identify set algebras with their domain referring to an injection $f:\A\to \wp(V)$ ($\A\in \CA_n$) as a complete representation  of $\A$ via $f$, or simply a complete representation of $\A$, 
where $V$  is a ${\sf Gs}_n$ unit.
Recall that $\bold S_c$ denotes the operation of forming {\it complete} subalgebras and that we write $\A\subseteq_c \B$ if $\A\in \bold S_c\{\B\}$.  
For a Boolean algebra $\A$ and $a\in \A$, $\Rl_a\A$ is the Boolean with universe $\{x\in \A: x\leq a\}$ and Boolean operations those of $\A$ 
relativized to the universe. 
We write $\A\subseteq_d \B$ if $\A$ is dense in $\B$,
and we let $\bold S_d$ denote the operation of forming {\it dense subalgebras}.

For a class $\sf K$ of $\sf BAO$s, recall that $\sf K\cap \bf At$ denotes the class of atomic algebras in $\K$.
Let ${\sf FCs}_{n}=\{\A\in {\sf Cs}_{n}: A=\wp(^nU) \text { some non--empty set $U$} \}.$ 
 \begin{theorem}\label{iii}
For $2<n<\omega$ the following hold:
\begin{enumerate}
\item ${\sf CRCA}_{n} \subseteq \bold S_c{\Nr}_{n}(\CA_{\omega}\cap {\bf At})\cap {\bf At}\subseteq \bold S_c{\Nr}_{n}\CA_{\omega}\cap \bf At,$
\item If  $\A\in {\sf CRCA}_{n}$, then \pe\ has  a \ws\ in $G_{\omega}(\At\A)$ and $\bold G^{\omega}(\At\A),$ 
\item All  reverse inclusions and implications in the previous two items hold, 
if algebras considered have countably many atoms,
\item At least two classes in the first item are distinct. Non of all these classes 
is elementary, but their elementary closure coincides with ${\sf LCA}_n$, 
\item ${\sf CRCA}_{n}={\bf S_c P}{\sf FCs}_{n}$,
\item  ${\Nr}_{n}\CA_{\omega}\cap {\bf At}\nsubseteq {\sf CRCA}_{n}$, $\Nr_{n}\CA_{\omega}\cap {\bf At}\subsetneq\bold S_c\Nr_{n}\CA_{\omega}\cap {\bf At}$ 
and  ${\sf CRCA}_n\subsetneq \bold S_c\Nr_{n}\CA_{\omega}\cap {\bf At}$,  
\item Neither of the classes ${\sf CRCA}_n$ and $\bold S_d\Nr_n\CA_{\omega}$ are contained in each other. Furthermore, 
$\Nr_n\CA_{\omega}\subsetneq \bold S_d\Nr_n\CA_{\omega}\subsetneq \bold S_c\Nr_n\CA_{\omega}$,

\end{enumerate}
\end{theorem}
\begin{proof}
1.   Let $\A\in {\sf CRCA}_n$. Assume that $\Mo$ is the base of a complete representation of $\A$, whose
unit is a generalized cartesian space,
that is, $1^{\Mo}=\bigcup {}^nU_i$, where $^{n}U_i\cap {}^{n}U_j=\emptyset$ for distinct $i$ and $j$, in some
index set $I$, that is, we have an isomorphism $t:\B\to \C$, where $\C\in {\sf Gs}_{n}$ 
has unit $1^{\Mo}$, and $t$ preserves arbitrary meets carrying
them to set--theoretic intersections.
For $i\in I$, let $E_i={}^{n}U_i$. Take  $f_i\in {}^{\omega}U_i$ 
and let $W_i=\{f\in  {}^{\omega}U_i^{(f_i)}: |\{k\in \omega: f(k)\neq f_i(k)\}|<\omega\}$.
Let ${\C}_i=\wp(W_i)$. Then $\C_i$ is atomic; indeed the atoms are the singletons. 
 
Let $x\in \mathfrak{Nr}_{n}\C_i$, that is ${\sf c}_ix=x$ for all $n\leq i<\omega$.
Now if  $f\in x$ and $g\in W_i$ satisfy $g(k)=f(k) $ for all $k<n$, then $g\in x$.
Hence $\mathfrak{Nr}_{n}\C_i$
is atomic;  its atoms are $\{g\in W_i:  \{g(i):i<n\}\subseteq U_i\}.$
Define $h_i: \A\to \mathfrak{Nr}_{n}\C_i$ by
$h_i(a)=\{f\in W_i: \exists a'\in \At\A, a'\leq a;  (f(i): i<n)\in t(a')\}.$
Let $\D=\bold P _i \C_i$. Let $\pi_i:\D\to \C_i$ be the $i$th projection map.
Now clearly  $\D$ is atomic, because it is a product of atomic algebras,
and its atoms are $(\pi_i(\beta): \beta\in \At(\C_i))$.  
Now  $\A$ embeds into $\mathfrak{Nr}_{n}\D$ via $J:a\mapsto (\pi_i(a) :i\in I)$. If $x\in \mathfrak{Nr}_{n}\D$,
then for each $i$, we have $\pi_i(x)\in \mathfrak{Nr}_{n}\C_i$, and if $x$
is non--zero, then $\pi_i(x)\neq 0$. By atomicity of $\C_i$, there is an $n$--ary tuple $y$, such that
$\{g\in W_i: g(k)=y_k\}\subseteq \pi_i(x)$. It follows that there is an atom
of $b\in \A$, such that  $x\cdot  J(b)\neq 0$, and so the embedding is atomic, hence complete.
We have shown that $\A\in \bold S_c{\sf Nr}_{n}\CA_{\omega}\cap \bf At$, and since $\A$ is atomic because $\A\in {\sf CRCA}_n$
we are done with the first inclusion.  The second inclusion is straightforward since $\CA_{\omega}\cap {\bf At}\subseteq \CA_{\omega}$.

2.   \cite[Theorem 3.3.3]{HHbook2}. Follows too from the first item taken together with lemma \ref{n}.

3.  Follows by observing that the class ${\sf CRCA}_n$ coincides with the class 
$\bold S_c\Nr_{n}\CA_{\omega}$ on atomic algebras having countably many atoms, cf. \cite[Theorem 5.3.6]{Sayedneat}, taken together with \cite[Theorem 3.3.3]{HHbook2}. Strictly speaking, 
in \cite{Sayedneat} it is shown that the two classes $\CRCA_n$ and $\bold S_c\Nr_n\CA_{\omega}$ 
coincide on countable atomic algebras. One can show that they coincide on the larger class of atomic agebras having countably many atoms by observing that if
$\A$ is an atomic algebra having countably many atoms, then $\Tm\At\A$ is countable 
and $\Tm\At\A\in {\sf CRCA}_n\iff \A\in {\sf CRCA}_n$ because an algebra is completely representable $\iff$ it is atomic and its atom structure is completely representable.

4. To show that non of the classes in the first item is elementary, let $\D$ be an atomic 
$\RCA_n$ with countably many atoms that is not completely representable, but is elementary equivalent
to some $\B\in {\sf CRCA}_n$. Such algebras exist; see e.g.  \cite{HH}. 
Another such algebra is the algebra $\C_{\N^{-1}, \N}$ used in Theorem \ref{iiii}.  It is the case that $\C_{\N^{-1}, \N}\notin {\sf CRCA}_n$, 
because $\C_{\N^{-1}, \N}\notin \bold S\Nr_n\CA_{n+3}\supseteq \bold S_c\Nr_n\CA_{\omega}\supseteq {\sf CRCA}_n$. Furthermore, $\C\in {\bf El}{\sf CRCA}_n$. 
Then $\D$ is not in any of the aforementiond classes because it has countably many atoms, and by the first item $\B$ is in all 
three classes, proving the required. 

We show, as claimed, that all the given classes coincide with ${\sf LCA}_n$. 
Assume that $\A\in {\sf LCA}_n$.
Take a countable elementary subalgebra $\C$ of $\A$.
Since ${\sf LCA}_{n}$ is elementary, then $\C\in {\sf LCA}_{n}$, so for 
$k<\omega$, \pe\ has a \ws\ $\rho_k$,  in $G_k(\At\C)$. 
Let $\D$ be a non--principal ultrapower of $\C$.  Then \pe\ has a \ws\ $\sigma$ in $G_{\omega}(\At\D)$ --- essentially she uses
$\rho_k$ in the $k$'th component of the ultraproduct so that at each
round of $G_{\omega}(\At\D)$,  \pe\ is still winning in co--finitely many
components, this suffices to show she has still not lost. Now one can use an
elementary chain argument to construct countable elementary
subalgebras $\C=\A_0\preceq\A_1\preceq\ldots\preceq\ldots \D$ in the following way.
One defines  $\A_{i+1}$ to be a countable elementary subalgebra of $\D$
containing $\A_i$ and all elements of $\D$ that $\sigma$ selects
in a play of $G_{\omega}(\At\D)$ in which \pa\ only chooses elements from
$\A_i$. Now let $\B=\bigcup_{i<\omega}\A_i$.  This is a
countable elementary subalgebra of $\D$, hence necessarily atomic,  and \pe\ has a \ws\ in
$G_{\omega}(\At\B)$, so $\B$ is completely representable.
Thus $\A\equiv \C\equiv \B$, hence $\A\in {\bf El}{\sf CRCA}_n$. We have shown that ${\sf LCA}_n\subseteq {\bf El}\sf CRCA_n.$

If $\A\in \bold S_c{\sf Nr}_n\CA_{\omega}\cap {\bf At}$, then by lemma \ref{n}, 
\pe\ has a \ws\ in $F^{\omega}(\At\A)$, hence in $G_{\omega}(\At\A)$, {\it a fortiori}, in  $G_k(\At\A)$ for all $k<\omega$, 
so $\A\in {\sf LCA}_{n}.$
Since ${\sf LCA}_n$ is elementary,  we get that ${\bf El}(\bold S_c{\sf Nr}_n\CA_{\omega}\cap {\bf At})\subseteq {\sf LCA}_n$.
But ${\sf CRCA}_n\subseteq \bold S_c{\sf Nr}_n\CA_{\omega}\cap \bf At$,  hence 
${\sf LCA}_n={\bf El}{\sf CRCA}_n\subseteq {\bf El}(\bold S_c{\sf Nr}_n\CA_{\omega}\cap {\bf At})\subseteq {\sf LCA}_{n}$. 

Now $\bold S_c{\Nr}_{n}\CA_{\omega}\cap {\bf At}\subseteq {\bf El}\bold S_c{\Nr}_{n}\CA_{\omega}\cap \bf At$,
and the latter class is elementary (if $\bold K$ is elementary, then $\bold K\cap \bf At$ is elementary), 
so ${\bf El}(\bold S_c{\Nr}_{n}\CA_{\omega}\cap \bf At)\subseteq {\bf El}\bold S_c{\Nr}_{n}\CA_{\omega}\cap \bf At.$
Conversely, if $\C$ is in the last class, then $\C$ is atomic and $\C\equiv \D$, for some $\D\in \bold S_c{\Nr}_{n}\CA_{\omega}$.
Hence $\D$ is atomic, so $\D\in  \bold S_c{\Nr}_{n}\CA_{\omega}\cap \bf At$, 
thus $\C\in  {\bf El}(\bold S_c{\Nr}_{n}\CA_{\omega}\cap \bf At)$.
We have shown that 
${\bf El}\bold S_c{\Nr}_{n}\CA_{\omega}\cap {\bf At}={\bf El}(\bold S_c{\sf Nr}_n\CA_{\omega}\cap {\bf At})={\sf LCA}_{n}={\bf El}{\sf CRCA}_{n}$. 

Finally, by Lemma \ref{n}, $\bold S_c{\Nr}_{n}(\CA_{\omega}\cap {\bf At})\cap {\bf At}\subseteq {\sf LCA}_{n}$, so
${\bf El}\bold S_c[{\Nr}_{n}(\CA_{\omega}\cap {\bf At})\cap {\bf At}]\subseteq {\sf LCA}_{n}$. The other inclusion follows from 
${\sf CRCA}_{n}\subseteq \bold S_c{\Nr}_{n}(\CA_{\omega}\cap \bf At)\cap \bf At$, so 
${\sf LCA}_{n}={\bf El}\CRCA_{n}\subseteq {\bf El}[\bold S_c{\Nr}_{n}(\CA_{\omega}\cap \bf At)\cap \bf At]$.   
We have shown that all classes coincide with ${\sf LCA}_{n}$, 
which is the elementary closure of ${\sf CRCA}_{n}$, and we are done. 

5.  The inclusion $\subseteq$ is straightforward.
Conversely, assume that $\A\subseteq_c \bold P_{i\in I}\wp({}^{n}U_i).$
Then  $\B=\bold P_{i\in I}\wp({}^{n}U_i)\cong \wp(V)$, where $V$ is the disjoint
union of the $^{n}U_i$, is clearly completely representable. Then  
since $\A\subseteq_c \B$,  then $\A$ is completely representable, too. 
To see why, suppose that $f:\B\to \wp(V)$ establishes a complete representation of
$\B$. 
We claim that $g=f\upharpoonright \A$ is a complete representation of $\A$. 
Let $X\subseteq \A$ be such that $\sum^{\A}X=1$. 
Then by $\A\subseteq_c \B$, we have  $\sum ^{\B}X=1$. Furthermore, for all $x\in X(\subseteq \A)$ we have $f(x)=g(x)$, so that 
$\bigcup_{x\in X}g(x)=\bigcup_{x\in X} f(x)=V$, since $f$ is a complete representation, 
and we are done. 

6. First $\nsubseteq$ follows from Theorem \ref{bsl}.
Second $\subsetneq$ follows from the construction in \cite{SL} recalled in Part 1.
Last $\subsetneq$ follows from the first two parts in this item together with the inclusions in the first item.

7. That $\bold S_d\Nr_n\CA_{\omega}\cap {\bf At}\nsubseteq {\sf CRCA}_n$ follows from Theorem \ref{bsl}. 
To show that, conversely  ${\sf CRCA}_n\nsubseteq \bold S_d\Nr_n\CA_{\omega}\cap {\bf At}$,  we 
slighty modify the construction in \cite[Lemma 5.1.3, Theorem 5.1.4]{Sayedneat} lifted to any finite $n>2$. 
The algebras $\A$ and $\B$ constructed in {\it op.cit} satisfy that
$\A\in {\sf Nr}_n\CA_{\omega}$, $\B\notin {\sf Nr}_n\CA_{n+1}$ and $\A\equiv \B$.
As they stand, $\A$ and $\B$ are not atomic, but it 
can be  fixed that they are to be so giving the same result, by interpreting the uncountably many tenary relations in the signature of 
$\Mo$ defined in \cite[Lemma 5.1.3]{Sayedneat}, which is the base of $\A$ and $\B$ 
to be {\it disjoint} in $\Mo$, not just distinct.  The construction is presented this way in \cite{IGPL}, where (the equivalent of) 
$\Mo$ is built in a 
more basic step-by--step fashon.

We work with $2<n<\omega$ instead of only $n=3$. The proof presented in {\it op.cit} lift verbatim to any such $n$.
Let $u\in {}^nn$. Write $\bold 1_u$ for $\chi_u^{\Mo}$ (denoted by $1_u$ (for $n=3$) in \cite[Theorem 5.1.4]{Sayedneat}.) 
Now we have $\A, \B$ are atomic $\RCA_n$s such 
$\A\in {\sf Nr}_n\CA_{\omega}$, $\B\notin {\bf El}{\sf Nr}_n\CA_{n+1}$ and $\A\equiv \B$.
So $\At\B\notin \At{\bf El}\Nr_n\CA_{n+1}(\supseteq \At{\bf El}\Nr_n\CA_{\omega}$). 
Also $\B\in {\sf CRCA}_n$ because
$\B \in {\sf Gs}_n$ (the class of generalized set algebras of dimension $n$)  and  
$\bigcup \At\B=\bigcup_{u\in{}^nn}\bigcup \At\B_u=\bigcup_{u\in {}^nn}\bold 1_u= 1^{\B}.$ Thus the identity may establishes a complete representation of $\B$.
It follows that $\At\B$ satisfies the 
Lyndon conditions.

We have 
$\A\in {\sf Nr}_n\CA_{\omega}$, $\B\notin {\sf Nr}_n\CA_{n+1}$, 
$\A\equiv \B$.
and $\A$ and $\B$ are atomic. Both algebras are based on the model ${\sf M}$.
Denote by $\A_u$ the Boolean algebra $\Rl_{\bold 1_u}\A=\{x\in \A: x\leq \bold 1_u\}$ 
and similarly  for $\B$, writing $\B_u$ short hand  for the Boolean algebra $\Rl_{\bold 1_u}\B=\{x\in \B: x\leq \bold 1_u\}.$
Using that $\Mo$ has quantifier elimination we get using the same argument in \cite[Theorem 5.1.4]{Sayedneat} 
that $\A\in \Nr_n\CA_{\omega}$.  The property that $\B\notin \Nr_n\CA_{n+1}$ is also still maintained.

To see why consider the substitution operator $_{n}{\sf s}(0, 1)$ (using one spare dimension) as defined in the proof of \cite[Theorem 5.1.4]{Sayedneat}.
Assume for contradiction that 
$\B=\Nr_{n}\C$, with $\C\in \CA_{n+1}.$ Let $u=(1, 0, 2,\ldots, n-1)$. Then $\A_u=\B_u$
and so $|\B_u|>\omega$. The term  $_{n}{\sf s}(0, 1)$ acts like a substitution operator corresponding
to the transposition $[0, 1]$; it `swaps' the first two co--ordinates.
Now one can show that $_{n}{\sf s(0,1)}^{\C}\B_u\subseteq \B_{[0,1]\circ u}=\B_{Id},$ 
so $|_{n}{\sf s}(0,1)^{\C}\B_u|$ is countable because $\B_{Id}$ was forced by construction to be 
countable. But $_{n}{\sf s}(0,1)$ is a Boolean automorpism with inverse
$_{n}{\sf s}(1,0)$, 
so that $|\B_{Id}|=|_{n}{\sf s(0,1)}^{\C}\B_u|>\omega$, contradiction.

Take the cardinality $\kappa$ specifying the signature of $\Mo$ to be $2^{2^{\omega}}$ and assume for contradiction that  
$\B\in \bold S_d\Nr_n\CA_{\omega}\cap \bf At$. 
Then $\B\subseteq_d \mathfrak{Nr}_n\D$, for some $\D\in \CA_{\omega}$ and $\mathfrak{Nr}_n\D$ is atomic. For brevity, 
let $\C=\mathfrak{Nr}_n\D$. Then $\B_{Id}\subseteq_d \Rl_{Id}\C$; the last algebra is the Boolean algebra with universe 
$\{x\in \C:  x\leq Id\}$.
Since $\C$ is atomic,  then $\Rl_{Id}\C$ is also atomic.  
Using the same reasoning as above, we get that $|\Rl_{Id}\C|>2^{\omega}$ (since $\C\in \Nr_n\CA_{\omega}$). 
By the choice of $\kappa$, we get that $|\At\Rl_{Id}\C|>\omega$. 
By $\B\subseteq_d \C$, we get that $\B_{Id}\subseteq_d \Rl_{Id}\C$, and 
that $\At\Rl_{Id}\C\subseteq \At\B_{Id}$, so $|\At\B_{Id}|\geq |\At\Rl_{Id}\C|>\omega$.   
But by the construction of $\B$, we have  $|\B_{Id}|=|\At\B_{Id}|=\omega$,   
which is a  contradiction and we are done.
The algebra $\B$ so constructed is atomic and 
is outside  $\bold S_d\Nr_n\CA_{\omega}$. Furthermore,  as proved above, $\B\in {\sf CRCA}_n$.

The algebra $\B$ in item (6) witnesses the strictness of the first inclusion while the algebra denoted also by $\B$ in this item  
witnesses the strictness of the second inclusion.
\end{proof}
Our next (and final) Theorem addresses (non-) first order definability of several classes of $\CA_n$s 
related to the classes $\Nr_n\CA_k$ $(k>n)$and ${\sf CRCA}_n$. 
\begin{theorem}\label{iiii}
For any class $\bold K$ such that $\Nr_n\CA_{\omega}\cap {\sf CRCA}_n\subseteq \bold K\subseteq \bold S_c\Nr_n\CA_{n+3}$, $\bold K$ is not elementary.
In particular, for $k\geq 3$, $\Nr_n\CA_{n+k}$ and ${\sf CRCA}_n$ are not elementary \cite{IGPL, HH}.
\end{theorem} 
\begin{proof}
We use the rainbow like algebra $\C_{\N^{-1}, \N}$ defined in Theorem \ref{fl} to prove $\Psi(n, n+3)_f$.
We proved in {\it op.cit} that \pa\ has a \ws\ in $\bold G^{n+3}(\At\C_{\N^{-1}, \N})$, implying  
by Lemma \ref{n}, that $\C_{\N^{-1}, \N}\notin \bold S_c{\sf Nr}_n\CA_{n+3}$.
For a start, we show the non-first order definabillity of any class $\bold K$ such that $\bold S_c\Nr_n\CA_{\omega}\cap {\sf CRCA}_n\subseteq \bold K\subseteq \bold S_c\Nr_n\CA_{n+3}$.
Towards this end, we prove that \pe\ has a \ws\ in $G_k(\At\C_{\N^{-1}, \N})$ for all $k\in \omega$. 
The \ws\ of \pe\ is similar but not identical to his \ws\ implemented in \cite{HH} played on $G_k(\At\CA_{\omega, \omega})(= G_k(\At\CA_{\N, \N})$).

Let $0<k<\omega$. We proceed inductively. Let $M_0, M_1, \ldots, M_r$, $r<k$ be the coloured graphs at the start of a play of $G_k$ just before round $r+1$.
Assume inductively, that \pe\ computes a partial function $\rho_s:\N^{-1}\to \N$, for $s\leq r:$

(i) $\rho_0\subseteq \ldots \rho_t\subseteq\ldots\subseteq\ldots  \rho_s$ is (strict) order preserving; if $i<j\in \dom\rho_s$ then $\rho_s(i)-\rho_s(j)\geq  3^{k-r}$, where $k-r$
is the number of rounds remaining in the game,
and 
$$\dom(\rho_s)=\{i\in \N^{-1}: \exists t\leq s, \text { $M_t$ contains an $i$--cone as a subgraph}\},$$

(ii) for $u,v,x_0\in \nodes(M_s)$, if $M_s(u,v)=\r_{\mu,k}$, $\mu, k\in \N$, $M_s(x_0,u)=\g_0^i$, $M_s(x_0,v)=\g_0^j$,
where $i,j\in \N^{-1}$ are tints of two cones, with base $F$ such that $x_0$ is the first element in $F$ under the induced linear order,
then $\rho_s(i)=\mu$ and $\rho_s(j)=k$.

For the base of the induction \pe\ takes $M_0=\rho_0=\emptyset.$ 
Assume that $M_r$, $r<k$  ($k$ the number of rounds) is the current coloured graph and that \pe\ has constructed $\rho_r:\N^{-1}\to \N$ to be a finite order preserving partial map
such conditions (i) and (ii) hold. We show that (i) and (ii) can be maintained in a 
further round.
We check the most difficult case. Assume that $\beta\in \nodes(M_r)$, $\delta\notin \nodes(M_r)$ is chosen by \pa\ in his cylindrifier move,
such that $\beta$ and $\delta$ are apprexes of two cones having
same base and green tints $p\neq  q\in \N^{-1}$. 
Now \pe\ adds $q$ to $\dom(\rho_r)$ forming $\rho_{r+1}$ by defining the value $\rho_{r+1}(p)\in \N$ 
in such a way to preserve the (natural) order on $\dom(\rho_r)\cup \{q\}$, that is maintaining property (i).
Inductively, $\rho_r$ is order preserving and `widely spaced' meaning that the gap between its elements is
at least $3^{k-r}$, so this can be maintained in a further round.

Now \pe\  has to define a (complete) coloured graph 
$M_{r+1}$ such that $\nodes(M_{r+1})=\nodes(M_r)\cup \{\delta\}.$ 
In particular, she has to find a suitable 
red label for the edge $(\beta, \delta).$
Having $\rho_{r+1}$ at hand she proceeds as follows. Now that $p, q\in \dom(\rho_{r+1})$, 
she lets $\mu=\rho_{r+1}(p)$, $b=\rho_{r+1}(q)$. The red label she chooses for the edge $(\beta, \delta)$ is: (*)\ \  $M_{r+1}(\beta, \delta)=\r_{\mu,b}$.
This way she maintains property (ii) for $\rho_{r+1}.$  Next we show that this is a \ws\ for \pe. 

We check consistency of newly created triangles proving that $M_{r+1}$ is a coloured graph completing the induction. 
Since $\rho_{r+1}$ is chosen to preserve order, no new forbidden triple (involving two greens and one red) will be created.
Now we check red triangles only of the form $(\beta, y, \delta)$ in $M_{r+1}$ $(y\in \nodes(M_r)$). 
We can assume that  $y$ is the apex of a cone with base $F$ in $M_r$ and green tint $t$, say,
and that $\beta$ is the appex of the $p$--cone having the same base. 
Then inductively by condition (ii), taking $x_0$ to be the first element of $F$, and taking  
the nodes $\beta, y$, and the tints $p, t$, for $u, v, i, j$,  respectively, we have by observing that 
$\beta, y\in \nodes(M_r)$, $\beta, y\in \dom(\rho_r)$ and $\rho_r\subseteq \rho_{r+1}$, 
the following:  
$M_{r+1}(\beta,y)=M_{r}(\beta, y)=\r_{\rho_{r}(p), \rho_{r}(t)}=r_{\rho_{r+1}(p), \rho_{r+1}(t)}.$
By  her strategy, we have  $M_{r+1}(y,\delta)=\r_{\rho_{r+1}(t), \rho_{r+1}(q)}$ 
and we know by (*) that $M_{r+1}(\beta, \delta)=\r_{\rho_{r+1}(p), \rho_{r+1}(q)}$. 
The triple $(\r_{\rho_{r+1}(p), \rho_{r+1}(t)}, \r_{\rho_{r+1}(t), \rho_{r+1}(q)}, \r_{\rho_{r+1}(p), \rho_{r+1}(q)})$
of reds is consistent and we are done with this case. 
All other edge labelling and colouring $n-1$ tuples in $M_{r+1}$ 
by yellow shades are  exactly like in \cite{HH}.

Using ultrapowers and an elementary chain argument like in  \cite[Theorem 3.3.5]{HHbook2}, 
one gets a countable algebra $\B$ 
such that $\B\equiv \C_{\N^{-1}, \N}$, and  \pe\ has a \ws\ in $G_{\omega}(\At\B)$. Then $\B$, being countable, 
is completely representable by \cite[Theorem 3.3.3]{HHbook2}. 
Let $\bold K$ be a class between ${\sf CRCA}_n$ 
and $\bold S_c\Nr_n\CA_{n+3}$. (It can be easily proved that ${\sf CRCA}_n\subseteq \bold S_c\Nr_n\CA_{\omega}\subseteq \bold S_c\Nr_n\CA_{n+3}$ so `between' here is valid).
Then $\C_{\N^{-1}, \N}\notin \bold S_c\Nr_n\CA_{n+3}(\supseteq \bold K)$ as shown during the proof of $\Psi(n, n+3)_f$, $\B\equiv \C_{\N^{-1}, \N}$, 
and $\B\in {\sf CRCA}_n(\subseteq \bold K)$.

But we can do better. One can define  
an $\omega$--rounded (atomic) game $\bold H_{\omega}(\alpha)$  
such that 
if $\alpha$ is a countable atom structure, and 
\pe\ has a \ws\ in $\bold H_{\omega}(\alpha)$, then 
any algebra $\F$ having atom structure $\alpha$ is completely representable, $\Cm\alpha\in \Nr_n\CA_{\omega}$ 
and $\alpha\in \At\Nr_{\alpha}\CA_{\omega}.$ In fact, there will exist a complete $\D\in \CA_{\omega}$ 
such that $\Cm\alpha\cong \Nr_n\D$ and $\alpha\cong \At\Nr_n\D$. 

Now we apply the new game $\bold H$ to the rainbow algebra $\C_{\N^{-1}, \N}$ based on the ordered structures $\N^{-1}$ and $\N$
used above.
With some (slightly more) effort one can prove that \pe\ can win 
the (stronger) game $\bold H_k(\At\C_{\N^{-1}, \N})$ which is the game $\bold H$ truncated to $k$ rounds (on the same $\C_{\N^{-1}, \N}$ based on $\N^{-1}$ and $\N$) for all $k<\omega$.  
Using 
ultrapowers followed by an elementary chain argument,  it follows \pe\ has a \ws\ in $\bold H(\alpha)$ for a countable atom structure $\alpha$, such
that $\C_{\N^{-1}, \N}\equiv \Tm\alpha$. 
Thus $\alpha\in \At(\Nr_n\CA_{\omega})$, any atomic $\F\in \CA_n$ 
having atom structure $\alpha$ is completely representable, 
and  $\Cm\alpha\in \Nr_n\CA_{\omega}$. So $\Tm\alpha\subseteq_d \Cm\alpha\in \bold {\sf Nr}_n\CA_{\omega}$, $\Tm\alpha\in {\sf CRCA}_n$
and  $\C_{\N^{-1}, \N}\notin \bold S_c{\sf Nr}_n\CA_{n+3}$. 
Now $\C_{\N^{-1}, \N}\notin \bold S_c\Nr_n\CA_{n+3}$, $\C\equiv \Tm\alpha$ and  $\Cm\alpha\in \Nr_n\CA_{\omega}$. 
So if $\bold K$ is a class betwen $\bold S_d\Nr_n\CA_{\omega}$ and $\bold S_c\Nr_n\CA_{n+3}$, then $\C\notin \bold K$, $\Tm\alpha\in \bold K$ 
and $\C\equiv \Tm\alpha$, hence $\bold K$ 
is not first order definable.

We need to remove the $\bold S_d$, that is, consider classes between $\Nr_n\CA_{\omega}$ and (the larger) $\bold S_d\Nr_n\CA_{\omega}$.  One might be 
tempted to think that we already did. However, this is not true because $\At\D\in \At(\Nr_n\CA_m)$ 
does not imply that  $\D\in \Nr_n\CA_m$, even if the \de\ completion of $\D$ is in $\Nr_n\CA_m$, which was the case with $\Tm\alpha$. 
Indeed, the algebra $\B$ used in item (6) alerts us to the fact that although $\alpha\in \At(\Nr_n\CA_{\omega})$ and $\Cm\alpha$ the \de\ completion of $\Tm\alpha$ is in 
$\Nr_n\CA_{\omega}$,  we do not guarantee that $\Tm\alpha$ itself is in $\Nr_n\CA_{\omega}$. 
 
So we are not done yet.
We still have some work to do.
For this purpose, we need the intervention of the construction used in item 7 in the proof of Theorem \ref{iii}. 
We use  
the (modified) algebras $\A$ and $\B$ that  we know  
satisfy: $\A\in {\sf Nr}_n\CA_{\omega}$, $\B\notin \bold S_d{\sf Nr}_n\CA_{n+1}$, $\A\equiv \B$
and both $\A$ and $\B$ are atomic. 

So far,  we have excluded any first order definable class between 
the two classes $\bold S_d\Nr_n\CA_{\omega}\cap \CRCA_n$ and $\bold S_c\Nr_n\CA_{n+3}$. 
So hoping for a contradiction, we  
assume that there is a class 
$\bold M$ between $\Nr_n\CA_{\omega}\cap {{\sf CRCA}_n}$ and $\bold S_d\Nr_n\CA_{\omega}\cap \CRCA_n$ that is first order definable. 
Then ${\bf El}(\Nr_n\CA_{\omega}\cap {{\sf CRCA}_n})\subseteq \bold M\subseteq \bold S_d\Nr_n\CA_{\omega}\cap \CRCA_n$. 
We have, $\B\equiv \A$, and $\A\in \Nr_n\CA_{\omega}\cap {{\sf CRCA}_n}$, hence $\B\in {\bf El}(\Nr_n\CA_{\omega}\cap {\sf CRCA}_n)\subseteq \bold S_d\Nr_n\CA_{\omega}\cap \CRCA_n$. 
But we know that $\B$  
is in fact outside $\bold S_d\Nr_n\CA_{n+1}\supseteq \bold S_d\Nr_n\CA_{\omega}\supseteq \bold S_d\Nr_n\CA_{\omega}\cap {\bf At}\supseteq \bold S_d\Nr_n\CA_{\omega}\cap \CRCA_n$,  
getting the hoped for contradiction,
and consequently the required.  

\end{proof}

\end{document}